\documentclass[10pt]{amsart}

\usepackage{comment, amsmath, amssymb, amsgen, amsthm, amscd, xspace, color, epsfig, float, epic, multirow, array, tabu}
\usepackage[hypertex]{hyperref}
\usepackage{extarrows}

\makeatletter
\newcommand{\rmnum}[1]{\romannumeral #1}
\newcommand{\Rmnum}[1]{\expandafter\@slowromancap\romannumeral #1@}

\newcommand{\eps}{\epsilon}

\newcommand{\vf}{\varphi}

\newcommand{\Ext}{\operatorname{Ext}}

\newcommand{\Hom}{\operatorname{Hom}}

\newcommand{\ch}{\operatorname{ch}}

\newcommand{\rank}{\mathrm{rank}}

\newcommand{\ra}{\rightarrow}
\newcommand{\lera}{\leftrightarrow}

\newcommand{\frb}{\mathfrak{b}}

\newcommand{\frg}{\mathfrak{g}}
\newcommand{\frh}{\mathfrak{h}}

\newcommand{\frl}{\mathfrak{l}}

\newcommand{\frp}{\mathfrak{p}}

\newcommand{\fru}{\mathfrak{u}}

\newcommand{\frgl}{\mathfrak{gl}}
\newcommand{\frsl}{\mathfrak{sl}}

\newcommand{\frso}{\mathfrak{so}}
\newcommand{\frsp}{\mathfrak{sp}}

\newcommand{\bbC}{\mathbb{C}}

\newcommand{\bbQ}{\mathbb{Q}}

\newcommand{\bbZ}{\mathbb{Z}}
\newcommand{\caA}{\mathcal{A}}

\newcommand{\caC}{\mathcal{C}}
\newcommand{\caD}{\mathcal{D}}

\newcommand{\caN}{\mathcal{N}}
\newcommand{\caO}{\mathcal{O}}
\newcommand{\caP}{\mathcal{P}}

\newtheorem{theorem}[equation]{Theorem}
\newtheorem{cor}[equation]{Corollary}
\newtheorem{prop}[equation]{Proposition}
\newtheorem{lemma}[equation]{Lemma}

\theoremstyle{remark}
\newtheorem{remark}[equation]{Remark}

\theoremstyle{definition}
\newtheorem{definition}[equation]{Definition}

\newtheorem{example}[equation]{Example}

\numberwithin{equation}{section}
\begin{document}

\title[blocks of category $\caO^\frp$]{Janzten coefficients and Blocks of category $\caO^\frp$}

\author{Wei Xiao}
\address[Xiao]{College of Mathematics and statistics, Shenzhen University,
Shenzhen, 518060, Guangdong, P. R. China}
\email{xiaow@szu.edu.cn}

\thanks{The author is supported by the National Science Foundation of China (Grant No. 11701381) and Guangdong Natural Science Foundation (Grant No. 2017A030310138).}



\bigskip

\begin{abstract}
The BGG category $\mathcal{O}$ and its generalization $\caO^\frp$ play essential roles in representation theory and have led to far-reaching work. Some elementary problems remain open for several decades, such as the block decomposition of category $\caO^\frp$. In this paper, we solve the problem of blocks by applying the theory of Jantzen coefficients.
\end{abstract}

\maketitle

%
%
\section{Introduction}
%
%
The category $\caO^\frp$ \cite{R} is a natural generalization of the well known category $\caO$ introduced by Joseph Bernstein, Israel Gelfand and Sergei Gelfand \cite{BGG}. If two simple modules of $\caO^\frp$ extend nontrivially, we put them in the same block; this relation generates an equivalence relation which partitions the simple modules of $\caO^\frp$ into blocks \cite{H3}. The goal of this paper is to completely determine the blocks of $\caO^\frp$. This problem has fundamental importance in the study of category $\caO^\frp$ and related topics, such as simplicity criterion of generalized Verma modules \cite{J2, He, HKZ, BX}, homomorphism between generalized Verma modules \cite{Bo, BC, BEJ, BN, L1, L2, M1, M2, M3, Xi} and representation types of blocks of $\caO^\frp$ \cite{BN, P2}. The results in this paper illustrate the inside symmetry of $\caO^\frp$ which revealed by the theory of Koszul duality \cite{So, BGS, B}.

Let $\frg\supset\frb\supset\frh$ be a complex semisimple Lie algebra $\frg$ containing $\frp$ with a fixed Borel subalgebra $\frb$ and a fixed Cartan subalgebra $\frh$. The category $\caO^\frp$ can be decomposed according to generalized infinitesimal characters:
\[
\caO^\frp=\bigoplus_{\chi}\caO_{\chi}^\frp,
\]
where $\chi=\chi_\lambda$ for some $\lambda\in\frh^*$. It suffices to find block decomposition of $\caO_{\chi_\lambda}^\frp=\caO_\lambda^\frp$ (In some paper, the full subcategory $\caO_\lambda^\frp$ is called a ``block'' of $\caO^\frp$). If $\frp=\frb$ or $\lambda$ is regular, the subcategory $\caO_\lambda^\frp$ is a block \cite{H3}. It was proved by Brundan that this is also true when the root systems of irreducible components of $\frg$ is of type $A$ \cite{Br}. However, there are examples showing that $\caO_\lambda^\frp$ can further decompose as sum of more than one block for all the other root systems (see \cite{ES, BN, P1} or examples in this paper). So the general problem turns out to be quite elusive. Only the two lower rank cases $F_4$ and $G_2$ was fully worked out using computer programs \cite{P2} (block is called ``linkage class'' there) in 2009. Recently, the blocks for semisimple Lie algebras of type $E$ are given in \cite{HXZ}

Recall that extensions between simple modules are determined by the leading coefficients (the $\mu$-function) of Kazhdan-Lusztig polynomials \cite{KL}. However, the calculation of $\mu$-function is a very difficult open problem. We do not understand it except for some special cases \cite{Lu, X, LX}. In this paper, we show that the blocks can also be determined by the Jantzen coefficients \cite{XZ}, which come from the well-known Jantzen filtration \cite{J1, J2} for standard modules. This is the first step towards a possible solution. The computation of Jantzen coefficents is relatively easy since the coefficients have many invariant properties \cite{XZ}. They can be used to efficiently calculate the radical filtration of generalized Verma modules in many cases \cite{HX}. Possible criteria emerge when we accumulated many interesting examples.

More preciously, suppose that $\frg$ is isomorphic to one of the Lie algebras $\frgl(n, \bbC)$, $\frso(2n+1, \bbC)$, $\frsp(n, \bbC)$ and $\frso(2n, \bbC)$ with the root system $\Phi$. Let $\Phi^+$ be the positive system associated with $\frb$ and $\Delta$ be the corresponding simple system. Let $W$ be the Weyl group of $\Phi$ with length function $\ell(\cdot)$. Any subset $I\subset\Delta$ generates a subsystem $\Phi_I$ of $\Phi$ and a subgroup $W_I$ of $W$. For $\lambda\in\frh^*$, denote by $L(\lambda)$ a simple module with highest weight $\lambda-\rho$, where $\rho$ is the half sum of positive roots. If $\lambda\in\frh^*$ is integral, let $\overline\lambda$ be the unique dominant weight in the orbit $W\lambda$. Put
\[
\Phi_\lambda=\{\alpha\in\Delta\mid\langle\lambda,\alpha\rangle=0\}.
\]
There exists $J\subset\Delta$ so that $\Phi_J=\Phi_{\overline\lambda}$. Suppose that $\frp$ is the standard parabolic subalgebra corresponding to $I$. The simple modules of $\caO_\lambda^\frp$ can be parameterized by the set
\[
{}^IW^J=\{w\in W\mid \ell(xwy)=\ell(x)+\ell(w)+\ell(y)\ \mbox{for any}\ x\in W_I, y\in W_J\},
\]
that is, $L(w\overline\lambda)\in\caO_\lambda^\frp$ for any $w\in{}^IW^J$. In particular, $\langle w\overline\lambda, \alpha\rangle>0$ for all $\alpha\in I$. The category $\caO_\lambda^\frp$ is fully determined by the triple $(\Phi, \Phi_I, \Phi_J)$ up to isomorphism of categories. If $\frg\simeq\frso(2n, \bbC)$ for $n\geq1$, then $s_{e_n}J:=\{s_{e_n}\alpha\mid\alpha\in J\}\subset\Delta$. In this case, set $\overline {{}^IW{}^J}={}^IW^J\sqcup{}^IW^{s_{e_n}J}$.

\bigskip
\noindent{\bf Main Theorem} (Theorem \ref{main}) Let $I, J\subset\Delta$ with ${}^IW^J\neq\emptyset$. Choose a dominant integral weight $\lambda$ so that $\Phi_{\lambda}=\Phi_J$.
\begin{itemize}
\item [(1)] There exist $k\geq0$ and subalgebras $\frg_i$ $(1\leq i\leq k)$ with simple system $\Delta_i$ and Weyl groups $W_i$ such that
\begin{equation*}\label{inteq1}
\overline{{}^IW^J}\simeq \overline{{}^{I^1}W_1^{J^1}}\times\ldots\times\overline{{}^{I^k}W_1^{J^k}}
\end{equation*}
when $\Phi=D_n$ with $(\Phi, \Phi_I, \Phi_J)$ being even and
\begin{equation*}\label{inteq2}
{}^IW^J\simeq {}^{I^1}W_1^{J^1}\times\ldots\times{}^{I^k}W_1^{J^k}
\end{equation*}
otherwise. Here $I^i, J^i\subset\Delta_i$ are determined by $I, J$. Moreover, the categories associated with ${}^{I^i}W_i^{J^i}$ are pseudo-indecomposable.

\item [(2)] $L(\mu), L(\nu)\in\caO_\lambda^\frp$ are in the same block if and only if $L(\mu^i), L(\nu^i)$ are in the same block for each subcategories. The weights $\mu^i, \nu^i\in\frh_i^*$ are determined by $\mu, \nu$, where $\frh_i^*$ is a Cartan subalgebra of $\frg_i$.
\end{itemize}

\bigskip

The definitions of even systems and pseudo-indecomposable categories are given in the paper (see \S7.3). A pseudo-indecomposable category contains at most two blocks depending on our criteria (Corallary \ref{bksscor2}). We also give the number of blocks in each case (see \S7). Here we can see from the main theorem that $\caO_\lambda^\frp$ has exactly $2^p$ blocks for some nonnegative integer $p\leq k$. If $\frg\simeq\frgl(n, \bbC)$, then $p=0$. This recovers Brundan's result for type $A$ \cite{Br}. If $\Phi$ is of type $B$, $C$ or $D$, there exists $\caO_\lambda^\frp$ with exactly $2^p$ blocks for any $p\geq0$ (see Example \ref{ccex2}).

Now we illustrate our main theorem by examples.



\begin{example}\label{intex1}
Let $\frg=\frgl(7, \bbC)$ and $I=\{e_1-e_2, e_2-e_3, e_4-e_5\}$. Choose a dominant weight $\lambda=(3, 2, 2, 2, 1, 1, 1)$. Thus $J=\{e_2-e_3, e_3-e_4, e_5-e_6, e_6-e_7\}$. The category $\caO_\lambda^\frp$ contains two simple modules $L(\mu)$, $L(\nu)$ with
\[
\mu=(3, 2, 1\mid2, 1\mid 2\mid1)\ \mbox{and}\ \nu=(3, 2, 1\mid2, 1\mid 1\mid2).
\]
Note that $I$ separates each weight into four segments (different segments are divided by vertical lines). We put entries of $\mu$ into a table $T(\mu)$ such that each segment are in the same column and equal entries are in the same row. Similarly we can get $T(\nu)$. Therefore
\begin{table}[htbp]
$T(\mu)=\ $\begin{tabu}{|p{0.1cm}<{\centering}|p{0.1cm}<{\centering}|[1.5pt]p{0.1cm}<{\centering}|p{0.1cm}<{\centering}|}
\hline
$3$ &  &  &\\
\tabucline[1.5pt]{-}
$2$ & $2$ & $2$ &\\
\hline
$1$ & $1$ & & $1$\\
\hline
\end{tabu}\quad \mbox{and}\quad
$T(\nu)=\ $\begin{tabu}{|p{0.1cm}<{\centering}|p{0.1cm}<{\centering}|[1.5pt]p{0.1cm}<{\centering}|p{0.1cm}<{\centering}|}
\hline
$3$ &  &  &\\
\tabucline[1.5pt]{-}
$2$ & $2$ &  & $2$\\
\hline
$1$ & $1$ & $1$ & \\
\hline
\end{tabu}\ ,
\bigskip
\end{table}

Note that the bold lines divide the table into four parts. The lower left of each table is always {\begin{tabular}{|p{0.1cm}<{\centering}|p{0.1cm}<{\centering}|}
\hline
 $2$ & $2$\\
\hline
 $1$ & $1$\\
\hline
\end{tabular}}, while the upper right is always empty. The weights $\mu^1$, $\nu^1$ are determined by the upper left, that is, $\mu^1=(3)=\nu^1$, while $\mu^2$, $\nu^2$ are determined by the lower right, that is, $\mu^2=(2\mid 1)$ and $\nu^2=(1\mid 2)$. Moreover, $\frg_1\simeq\frgl(1, \bbC)$ and $\frg_2=\frgl(2, \bbC)$, while $I^1=I^2=J^1=J^2=\emptyset$. It is easy to see that ${}^IW^J\simeq {}^{I^1}W_1^{J^1}\times {}^{I^2}W_1^{J^2}$.

\end{example}

\begin{example}\label{intex2}
Let $\frg=\frso(13, \bbC)$ and $\Delta\backslash I=\{e_4-e_5, e_6\}$. Choose a dominant weight $\lambda=(2, 1, 1, 1, 0, 0)$. Thus $\Delta\backslash J=\{e_1-e_2, e_4-e_5\}$. The category $\caO_\lambda^\frp$ contains four simple modules. Choose
\[
\mu=(2, 1, 0, -1\mid1, 0\mid)\ \mbox{and}\ \nu=(1, 0, -1, -2\mid0, -1\mid).
\]
Here $I$ separates each weight into three segments (including an empty one). We can construct tables $T(\mu)$ and $T(\nu)$. Now entries with equal absolute value are in the same row, while each nonempty segment still possesses a column, that is,
\begin{table}[htbp]
$T(\mu)=\ $\begin{tabu}{|p{0.3cm}<{\centering}|[1.5pt]p{0.3cm}<{\centering}|}
\hline
$2$ &  \\
\tabucline[1.5pt]{-}
$\pm1$ & $1$\\
\hline
$0$ & $0$\\
\hline
\end{tabu}\quad \mbox{and}\quad
$T(\nu)=\ $\begin{tabu}{|p{0.3cm}<{\centering}|[1.5pt]p{0.3cm}<{\centering}|}
\hline
$-2$ &  \\
\tabucline[1.5pt]{-}
$\pm1$ & $-1$\\
\hline
$0$ & $0$\\
\hline
\end{tabu} ,
\bigskip
\end{table}

Similarly, the lower left is stable and the upper right is empty for each table. The weights $\mu^1$, $\nu^1$ also depend on the upper left. Unlike the case of type $A$, we have to take  the $0$-entries from corresponding segments to form $\mu^1$ and $\nu^1$. Thus $\mu^1=(2, 0\mid)$ and $\nu^1=(0, -2\mid)$. The lower right of the tables yields $\mu^2=(1, 0\mid)$ and $\nu^2=(0, -1\mid)$. For $i\in\{1, 2\}$, one has $\frg_i\simeq\frso(5, \bbC)$, $I^i=\{e_1-e_2\}$ and $J^i=\{e_2\}$. The category associated with ${}^{I^i}W_i^{J^i}$ contains two simple modules $L(\mu^i)$, $L(\nu^i)$. The extensions between them are trivial, that is, they are not in the same block. The main theorem shows that $\caO_\lambda^\frp$ has four blocks. Moreover, $L(\mu)$, $L(\nu)$ are not in the same block.
\end{example}

\begin{example}\label{intex3}
Let $\frg=\frso(18, \bbC)$ and $\Delta\backslash I=\{e_4-e_5, e_6-e_7\}$. Choose a dominant weight $\lambda=(2, 2, 1, 1, 1, 1, 0, 0, 0)$. Thus $\Delta\backslash J=\{e_2-e_3, e_6-e_7\}$. The category $\caO_\lambda^\frp$ contains four simple modules. Choose
\[
\mu=(2, 1, 0, -1\mid1, 0\mid2, 1, 0)\ \mbox{and}\ \nu=(1, 0, -1, -2\mid0, -1\mid2, 1, 0).
\]
In the sense of Example \ref{intex2}, a table is constructed for each weight.
\begin{table}[htbp]
$T(\mu)=\ $\begin{tabu}{|p{0.3cm}<{\centering}|p{0.3cm}<{\centering}|[1.5pt]p{0.3cm}<{\centering}|}
\hline
$2$ & $2$ & \\
\tabucline[1.5pt]{-}
$1$ & $\pm1$ & $1$\\
\hline
$0$ & $0$ & $0$\\
\hline
\end{tabu}\quad \mbox{and}\quad
$T(\nu)=\ $\begin{tabu}{|p{0.3cm}<{\centering}|p{0.3cm}<{\centering}|[1.5pt]p{0.3cm}<{\centering}|}
\hline
$2$ & $-2$ & \\
\tabucline[1.5pt]{-}
$1$ & $\pm1$ & $-1$\\
\hline
$0$ & $0$ & $0$\\
\hline
\end{tabu}.
\end{table}

In order to get a similar partition, we put the third segment $(2, 1, 0)$ of each weight into the first column of each table. Thus $\lambda^1=(2, 0\mid 2, 0)$ and $\nu^1=(0, -2\mid 2, 0)$. For $\mu^2$ and $\nu^2$, we need to take the lower right and corresponding entries in the first column (the last segment of each weight). Therefore, $\mu^2=(1, 0\mid 1, 0)$ and $\nu^2=(0, -1\mid 1, 0)$. For $i\in\{1, 2\}$, one has $\frg_i\simeq\frso(8, \bbC)$, $I^i=\{e_1-e_2, e_3\pm e_4\}=J^i$. The category associated with ${}^{I^i}W_i^{J^i}$ contains two simple modules $L(\mu^i)$, $L(\nu^i)$ and $\Ext^1(L(\mu^i)$, $L(\nu^i))=0$. In view of the main theorem, $\caO_\lambda^\frp$ contains four blocks. Moreover, $L(\mu)$, $L(\nu)$ are not in the same block.
\end{example}

Note that $k$ is always $2$ in the previous examples. At last we give an example with $k=3$.

\begin{example}\label{intex4}
Let $\frg=\frso(18, \bbC)$ and $\Delta\backslash I=\{e_5-e_6, e_8\pm e_9\}$. Choose $\lambda=(3, 2, 2, 2, 1, 1, 1, 1, -1)$. Thus $\Delta\backslash J=\{e_1-e_2, e_4-e_5, e_8-e_9\}$. The category $\caO_\lambda^\frp$ contains four simple modules. Choose
\[
\mu=(3, 2, 1, -1, -2\mid2, 1, -1\mid1\mid)\ \mbox{and}\ \nu=(2, 1, -1, -2, -3\mid2, 1, -1\mid-1\mid).
\]
In this case, $(\Phi, \Phi_I, \Phi_J)$ is even (\S7.3). We obtain
\begin{table}[htbp]
$T(\mu)=\ $\begin{tabu}{|p{0.3cm}<{\centering}|[1.5pt]p{0.3cm}<{\centering}|[1.5pt]p{0.3cm}<{\centering}|}
\hline
$3$ & & \\
\tabucline[1.5pt]{-}
$\pm2$ & $2$ & \\
\tabucline[1.5pt]{-}
$\pm1$ & $\pm1$ & $1$\\
\hline
\end{tabu}\quad \mbox{and}\quad
$T(\nu)=\ $\begin{tabu}{|p{0.3cm}<{\centering}|[1.5pt]p{0.3cm}<{\centering}|[1.5pt]p{0.3cm}<{\centering}|}
\hline
$-3$ & & \\
\tabucline[1.5pt]{-}
$\pm2$ & $2$ & \\
\tabucline[1.5pt]{-}
$\pm1$ & $\pm1$ & $-1$\\
\hline
\end{tabu}.
\end{table}

Here we need to do the partitions twice. One has $\mu^1=(3)$, $\nu^1=(-3)$, $\mu^2=\nu^2=(2)$, $\mu^3=(1)$ and $\nu^3=(-1)$. Thus $\frg_i\simeq\frso(2, \bbC)\simeq\frgl(1, \bbC)$, $\Phi_i=I^i=J^i=\emptyset$ and ${}^{I^i}W_i^{J^i}=1$ for $i\in\{1,2,3\}$. Both $\caO_\lambda^\frp$ and $\caO_{s_{e_n}\lambda}^\frp$ contain four blocks and $L(\mu)$, $L(\nu)$ are not in the same block.
\end{example}

This paper is organized as follows. Some basic notations and definition are given in section 2. In section 3, we transfer the problem of blocks to the problem of nonzero Jantzen coefficients. In section 4, we describe some special roots which makes corresponding Jantzen coefficients to be nonzero. The majority of the paper (section 5-8) is devoted to provide a complete description of blocks.

%
%
\section{Notations and definitions}
%
%
\subsection{General notations}
Let $\frg\supset\frb\supset\frh$ be a complexed reductive Lie algebra with a fixed Borel subalgebra $\frb$ and a fixed Cartan subalgebra $\frh$. Then $\frg=Z_\frg\oplus[\frg, \frg]$, where $Z_\frg\subset\frh$ is the center of $\frg$ and $[\frg, \frg]$ is the semisimple part. Denote by $\Phi\subset\frh^*$ the root system of $(\frg, \frh)$ with a positive system $\Phi^+$ and a simple system $\Delta\subset\Phi^+$ corresponding to $\frb$. Let $\frg_\alpha$ be the root subspace of $\frg$ corresponding to $\alpha\in\Phi$. Let $\Phi_I$ be the subsystem of $\Phi$ generated by a subset $I\subset\Delta$ with a positive root system $\Phi^+_I:=\Phi_I\cap\Phi^+$. The Weyl group $W$ (resp. $W_I$) is produced by reflections $s_\alpha$ with $\alpha\in\Phi$ (resp. $\alpha\in \Phi_I$). Denote by $\ell(\cdot)$ the length function on $W$, which can also be view as the length function on $W_I$ via restriction. The action of $W$ on $\frh^*$ is given by $s_\alpha\lambda=\lambda-\langle\lambda, \alpha^\vee\rangle\alpha$ for $\alpha\in\Phi$ and $\lambda\in\frh^*$. Here $\langle\cdot, \cdot\rangle$ is a bilinear form on $\frh^*$ induced from a non-degenerate invariant form on $\frg$ (e.g., \S0.2.2 in \cite{W}), which is the direct sum of the Killing form on $[\frg, \frg]$ and any non-degenerate symmetric form on $Z(\frg)$. And $\alpha^\vee:=2\alpha/\langle\alpha, \alpha\rangle$ is the coroot of $\alpha$.

The weight $\lambda\in\frh^*$ is called \emph{regular} (resp. $\Phi_I$-\emph{regular}) if $\langle\lambda, \alpha^\vee\rangle\neq0$ for all roots $\alpha\in\Phi$ (resp. $\alpha\in\Phi_I$). Otherwise $\lambda$ is called \emph{sigular} (resp. $\Phi_I$-\emph{singular}). We say $\lambda$ is \emph{integral} if $\langle\lambda, \alpha^\vee\rangle\in\bbZ$ for all $\alpha\in\Phi$. An integral weight $\lambda\in\frh^*$ is \emph{dominant} (resp. \emph{anti-dominant}) if $\langle\lambda, \alpha^\vee\rangle\in\bbZ^{\geq0}$ (resp. $\langle\lambda, \alpha^\vee\rangle\in\bbZ^{\leq0}$) for all $\alpha\in\Delta$. When $\lambda$ is integral, there exists a unique dominant weight $\overline\lambda$ in the orbit $W\lambda$ such that $\lambda=w\overline\lambda$ for some $w\in W$. Then $\underline\lambda:=w_0\overline\lambda$ is the unique anti-dominant weight in $W\lambda$, where $w_0$ is the longest element in $W$. Let $w_I$ be the longest element in $W_I$.

Let $\frl_I:=\frh\oplus\sum_{\alpha\in\Phi_I}\frg_\alpha$ be the Levi subalgebra and $\fru_I:=\bigoplus_{\alpha\in\Phi^+\backslash\Phi_I^+}\frg_\alpha$ be the nilpotent radical corresponding to $I$. The $\frp_I:=\frl_I\oplus\fru_I$ is a standard parabolic subalgebra of $\frg$. We frequently drop the subscript when $I$ is fixed. Put
\[
\Lambda_I^+:=\{\lambda\in\frh^*\ |\ \langle\lambda, \alpha^\vee\rangle\in\bbZ^{>0}\ \mbox{for all}\ \alpha\in I\}.
\]
Set $\rho:=\frac{1}{2}\sum_{\alpha\in\Phi^+}\alpha$. For $\lambda\in\Lambda_I^+$, the generalized Verma module is defined by
\[
M_I(\lambda):=U(\frg)\otimes_{U(\frp_I)}F(\lambda-\rho),
\]
where $F(\lambda-\rho)$ is a finite dimensional simple $\frl_I$-modules of highest weight $\lambda-\rho$, and has trivial $\fru_I$-actions viewed as a $\frp_I$-module. The generalized Verma module $M_I(\lambda)$ and its simple quotients $L(\lambda)$ has the same infinitesimal character $\chi_\lambda$, where $\chi_\lambda$ is an algebra homomorphism from the center $Z(\frg)$ of $U(\frg)$ to $\bbC$ so that $z\cdot v=\chi_\lambda(z)v$ for all $z\in Z(\frg)$ and all $v\in M_I(\lambda)$. For a fixed $\frp=\frp_I$, let $\caO^\frp$ be the category of all finitely generated $\frg$-modules $M$, which is finitely semisimple as an $\frl$-module and locally $\fru$-finite. Then $M_I(\lambda)$ are basic objects in the subcategory $\caO^\frp$. In particular, if $I=\emptyset$, then $M(\lambda):=M_I(\lambda)$ is the Verma module with highest weight $\lambda-\rho$ and $\caO^\frp$ is the usual Berstein-Gelfand-Gelfand category $\caO$. Let $\caO^\frp_\chi$ be the full subcategory of $\caO^\frp$ containing modules $M$ on which $z-\chi(z)$ acts as locally nilpotent operator for all $z\in Z(\frg)$. Then we have a decomposition
\[
\caO^\frp=\bigoplus_{\chi}\caO^\frp_\chi,
\]
where $\chi=\chi_\lambda$ for some $\lambda\in\frh^*$. We often write $\caO^\frp_\lambda$ instead of $\caO^\frp_\chi$ if $\chi=\chi_\lambda$. In view of the Harish-Chandra isomorphism, $\caO^\frp_\lambda=\caO^\frp_\mu$ for $\mu\in W\lambda$ since $\chi_\lambda=\chi_\mu$ in this case. Denote by $\ch M$ the formal character for $\frg$-module $M$ of $\caO$. The module $M$ has a composition series with simple quotients isomorphic to some $L(\lambda)$. Denote by $[M : L(\lambda)]$ the multiplicity of $L(\lambda)$.

For $\lambda\in\frh^*$, set
\[
\Phi_\lambda:=\{\beta\in\Phi\mid\langle\lambda,\beta\rangle=0\}.
\]
Thus $\Phi_\lambda$ is a subsystem of $\Phi$ with a positive system $\Phi_\lambda^+:=\Phi_\lambda\cap\Phi^+$. Define
\[
{}^IW=:\{w\in W\mid \ell(s_\alpha w)=\ell(w)+1\ \mbox{for all}\ \alpha\in I\}.
\]
If $\lambda=w\overline\lambda$ is an integral weight contained in $\Lambda_I^+$, then
\[
\langle \lambda, \alpha\rangle=\langle w\overline\lambda, \alpha\rangle=\langle \overline\lambda, w^{-1}\alpha\rangle>0
\]
for $\alpha\in I$ yields $w^{-1}\alpha\in\Phi^+$. It follows that $\ell(s_\alpha w)=\ell(w)+1$ (see Lemma 1.6 and Corollary 1.7 in \cite{H2}) and $w\in {}^IW$. Define the set of singular simple roots for $\overline\lambda$ by
\[
J=\{\alpha\in\Delta\mid\langle\overline\lambda, \alpha\rangle=0\}.
\]
Then $\Phi_\lambda=w\Phi_J\simeq\Phi_J$ and $W_J=\{w\in W\mid w\overline\lambda=\overline\lambda\}$. With $\lambda=w\overline\lambda=ws_\alpha\overline\lambda$ for all $\alpha\in J$, one has $ws_\alpha\in {}^IW$ by a similar argument. Put
\[
{}^IW^J=\{w\in{}^IW\mid \ell(w)+1=\ell(ws_\alpha)\ \mbox{and}\ ws_\alpha\in {}^IW,\ \mbox{for all}\ \alpha\in J\}.
\]
Every integral weight $\lambda\in\Lambda_I^+$ can be uniquely written in the form $\lambda=w\overline\lambda$ for some $w\in{}^IW^J$. Denote $J'=-w_0J$ and $w'=w_Iww_{J}w_0$. Then $J'\subset\Delta$ and $w'\in{}^IW^{J'}$ (see \cite{BN}). We get another parametrization of $\lambda$:
\[
\lambda=w\overline\lambda=ww_J\overline\lambda=w_I(w_Iww_Jw_0)(w_0\overline\lambda)=w_Iw'\underline\lambda.
\]
Although our parametrization is more convenient in this paper, we will always be aware of such differences in the cited results.

Since $\caO_\lambda^\frp$ is determined by the system $(\Phi, \Phi_I, \Phi_J)$ up to isomorphism of categories, we will frequently use the notation $(\Phi, \Phi_I, \Phi_J)$ rather than $\caO_\lambda^\frp$ if the corresponding result is independent of the choice of $\lambda$.

%
%
\section{Jantzen filtration and Jantzen coefficients}
%
%

\emph{For the remainder of the paper, we assume that all weights such as $\lambda, \mu$ are integral.}

In this section, we will show that the blocks of $\caO^\frp$ is determined by the Jantzen coefficients of generalized Verma modules contained in $\caO^\frp$.

\subsection{Jantzen filtration} For any $\lambda\in\Lambda_I^+$, we have the following equation (see for example Proposition 9.6 in \cite{H3}):
\begin{equation}\label{jfeq1}
\ch M_I(\lambda)=\sum_{w\in W_I}(-1)^{\ell(w)}\ch M(w\lambda).
\end{equation}
The right side of (\ref{jfeq1}), which we denoted by $\theta(\lambda)$, is valid for any $\lambda\in\frh^*$. Denote
\[
\Psi_\lambda^+=\{\beta\in\Phi^+\backslash\Phi_I\mid \langle\lambda, \beta^\vee\rangle\in\bbZ^{>0}\}.
\]

In the case of category $\caO$, there is a remarkable filtration known as Jantzen filtration (see \cite{J1}) for every Verma module $M(\lambda)$. It leads to a conceptual proof of the BGG Theorem and provides information about the composition factors of $M(\lambda)$. As pointed out by Humphreys (see Remark 9.17 of \cite{H3}), although there is no parallel role for ``Jantzen filtration'' in generalized Verma module, Jantzen developed such a filtration in \cite{J2}. Indeed, keeping in mind of Lemma 3, Satz 2 and the observation in the Bemerkung before Lemma 4 in \cite{J2}, along Jantzen's line for Verma modules (see \cite{J1} or \cite{H3}), one can obtain a natural generalization of Jantzen filtration:

\begin{theorem}[Jantzen filtration and sum formula for generalized Verma modules]\label{jfthm1}
Let $\lambda\in\Lambda_I^+$, then $M_I(\lambda)$ has a filtration by submodules
\[
M_I(\lambda)=M_I(\lambda)^0\supset M_I(\lambda)^1\supset M_I(\lambda)^2\supset\ldots
\]
with $M_I(\lambda)^i=0$ for large $i$, such that
\begin{itemize}
\item [(1)] Every nonzero quotient $M_I(\lambda)^i/M_I(\lambda)^{i+1}$ has a nondegenerate contravariant form.

\item [(2)] $M_I(\lambda)^1$ is the unique maximal submodule of $M_I(\lambda)$.

\item [(3)] There is a formula:
\begin{equation}\label{jft1eq1}
\sum_{i>0}\ch M_I(\lambda)^i=\sum_{\beta\in\Psi_\lambda^+}\theta(s_\beta\lambda).
\end{equation}
\end{itemize}
\end{theorem}

\begin{remark}
It was showed in \cite{HX} that Theorem \ref{jfthm1} also holds for radical filtration of $M_I(\lambda)$. These two filtrations coincide for regular weight $\lambda\in\Lambda_I^+$. It was conjectured that this is also true for singular weights (\cite{BB, Sh}).
\end{remark}

\subsection{Jantzen coefficents and blocks of category $\caO^\frp$} First we recall the Jantzen coefficients defined in \cite{XZ}. In view of (\ref{jfeq1}), one has $\theta(w\lambda)=(-1)^{\ell(w)}\theta(\lambda)$ for $w\in W_I$ and $\lambda\in\frh^*$. Moreover, if $\lambda$ is $\Phi_I$-singular, then $\theta(\lambda)=0$. If $\lambda$ is $\Phi_I$-regular, there exists $w\in W_I$ with $w\lambda\in\Lambda_I^+$. Therefore $\theta(\lambda)=(-1)^{\ell(w)}\theta(w\lambda)=(-1)^{\ell(w)}\ch M_I(w\lambda)$. We obtain the following formula.

\begin{equation}\label{jfeq2}
\sum_{\beta\in\Psi_\lambda^+}\theta(s_\beta\lambda)=\sum_{\nu\in\Lambda_I^+}c(\lambda, \nu)\ch M_I(\nu),
\end{equation}
where $c(\lambda, \nu)$ are called \emph{Jantzen coefficients} associated with $(\Phi_I, \Phi)$. These coefficients are nonzero for only finitely many $\nu\in\Lambda_I^+$ and can be calculated by a reduction process (see \cite{XZ}, \S4). If we put
\begin{equation}\label{jfeq3}
\Psi_{\lambda, \mu}^+:=\{\beta\in\Psi_\lambda^+\mid \mu=w_\beta s_\beta\lambda\ \mbox{for some}\ w_\beta\in W_I\},
\end{equation}
then $c(\lambda, \mu)=\sum_{\beta\in \Psi_{\lambda, \mu}^+}(-1)^{\ell(w_\beta)}$.

For $\lambda, \mu\in\frh^*$, we write $\mu\leq\lambda$ if $\Hom_\caO(M(\mu), M(\lambda))\neq0$. This can be viewed as the Bruhat ordering on $\frh^*$ (see \cite{ES}, \S2). In particular, if $\lambda, \mu\in\Lambda_I^+$ and $\Ext^1_{\caO^\frp}(L(\mu), L(\lambda))\neq0$, then either $\mu<\lambda$ or $\lambda<\mu$.

\begin{definition}\label{jfdef1}
Fix $I\subset\Delta$. Let $\lambda, \mu\in\Lambda_I^+$. If $c(\lambda, \mu)\neq0$, we write $\lambda\succ \mu$. In general, write $\lambda\succeq\mu$ if there exist $\lambda^1, \ldots, \lambda^{k}\in\Lambda_I^+$ $(k\geq0)$  so that
\[
\lambda=\lambda^0\succ\lambda^1\succ \ldots\succ \lambda^k=\mu.
\]
This induces another ordering on $\Lambda_I^+$. In particular, if $\lambda\succ\mu$ and there exists no $\lambda>\nu>\mu$ with $\lambda\succ\nu\succ\mu$, we say $\lambda$ is \emph{adjacent to} $\mu$.
\end{definition}


\begin{lemma}\label{jflem1}
Let $\lambda, \mu\in\Lambda_I^+$. If $[M_I(\lambda), L(\mu)]>0$, then $\lambda\succeq\mu$. Conversely, if $\lambda$ is adjacent to $\mu$, then $[M_I(\lambda), L(\mu)]>0$.
\end{lemma}

\begin{proof}
Assume that $[M_I(\lambda), L(\mu)]>0$ and $\mu<\lambda$ (the case $\lambda=\mu$ is trivial). In view of (\ref{jft1eq1}) and (\ref{jfeq2}), one has
\begin{equation}\label{jfl1eq1}
\sum_{i>0}[M_I(\lambda)^i : L(\mu)]=\sum_{\nu\in\Lambda_I^+}c(\lambda, \nu)[M_I(\nu) : L(\mu)].
\end{equation}
The left side of (\ref{jfl1eq1}) is nonzero. If $c(\lambda, \mu)\neq0$, we already get $\lambda\succ\mu$. If $c(\lambda, \mu)=0$, there exists $\lambda>\lambda^1>\mu$ such that $c(\lambda, \lambda^1)[M_I(\lambda^1) : L(\mu)]\neq0$. This yields $\lambda\succ\lambda^1$ and $[M_I(\lambda^1) : L(\mu)]\neq0$. It suffices to consider the case $\lambda^1\not\succ \mu$, then, $c(\lambda^1, \mu)=0$. With $\lambda$ replaced by $\lambda^1$ in the previous argument, there exists $\lambda^1\succ \lambda^2>\mu$ so that $c(\lambda^1, \lambda^2)[M_I(\lambda^2) : L(\mu)]\neq0$. Since the set $W\lambda$ containing $\lambda^i$ is a finite set, we can eventually get $\lambda\succ\lambda^1\succ\ldots\succ \lambda^{k-1}\succ\mu$ in this fashion.

Conversely, assume that $\lambda$ is adjacent to $\mu$. Then $c(\lambda, \mu)\neq0$ by definition. If $c(\lambda, \nu)[M_I(\nu) : L(\mu)]\neq0$ for some $\nu\in\Lambda_I^+$, the argument of the first part shows that $\lambda\succ\nu\succeq \mu$. The adjacent condition implies $\nu=\lambda$. Thus $c(\lambda, \nu)[M_I(\nu) : L(\mu)]=0$ unless $\nu=\mu$. We get
\begin{equation}\label{jfl3eq2}
\sum_{i>0}[M_I(\lambda)^i : L(\mu)]=c(\lambda, \mu)\neq0,
\end{equation}
by applying (\ref{jfl1eq1}). Hence $[M_I(\lambda) : L(\mu)]>0$.
\end{proof}


\begin{definition}\label{jfdef2}
Fix $I\subset\Delta$. For $\lambda, \mu\in\Lambda_I^+$, we say $\lambda, \mu$ are \emph{connected by Jantzen coefficients} if $\lambda=\mu$ or there exist $\lambda=\lambda^0, \lambda^1, \ldots, \lambda^r=\mu\in\Lambda_I^+$ such that
\[
c(\lambda^{i-1}, \lambda^i)\neq0\ \mbox{or}\ c(\lambda^{i}, \lambda^{i-1})\neq0.
\]
for all $1\leq i\leq r$.
\end{definition}

\begin{definition}\label{jfdef3}
Fix $I\subset\Delta$. For $\lambda, \mu\in\Lambda_I^+$, write $\lambda\lera\mu$ if $\Ext^1_{\caO^\frp}(L(\mu), L(\lambda))\neq0$. In general, write $\lambda\lera\mu$ if $\lambda=\mu$ or there exist $\lambda^0, \lambda^1, \ldots, \lambda^r\in\Lambda_I^+$ such that
\[
\lambda=\lambda^0\leftrightarrow\lambda^1\leftrightarrow\ldots\leftrightarrow\lambda^r=\mu,
\]
we say $\lambda, \mu\in\Lambda_I^+$ are {\it $\Ext^1$-connected relative to $(\Phi_I, \Phi)$}.

For convenience, if $\lambda\lera\mu$ relative to $(\Phi_I, \Phi)$, we write $w_1\lambda\lera w_2\mu$ for any $w_1, w_2\in W_I$.
\end{definition}

The definition gives an equivalence relation on all the $\Phi_I$-regular weights such that each $W_I$ orbit is contained in one equivalence class. The following result is well known.

\begin{lemma}\label{jflem2}
Let $\lambda, \mu\in\Lambda_I^+$ with $\mu<\lambda$. If $\Ext^1_{\caO^\frp}(L(\mu), L(\lambda))\neq0$, then $[M_I(\lambda), L(\mu)]>0$. Conversely, if $[M_I(\lambda), L(\mu)]>0$, then $\lambda\lera\mu$.
\end{lemma}

Lemma \ref{jflem1} implies the following result.

\begin{prop}\label{jfprop1}
Fix $I\subset\Delta$. Let $\lambda, \mu\in\Lambda_I^+$. Then $\lambda, \mu$ are $\Ext^1$-connected if and only if $\lambda, \mu$ are connected by Jantzen coefficients.
\end{prop}

\begin{proof}
First assume that $\lambda\lera\mu$. By Definition \ref{jfdef2}, it suffices to consider the case $\Ext^1_{\caO^{\frp}}(L(\mu), L(\lambda))\neq0$ with $\mu<\lambda$. In view of Lemma \ref{jflem2}, we have $[M_I(\lambda), L(\mu)]>0$. Lemma \ref{jflem1} yields $\lambda\succ\mu$. Thus $\lambda, \mu$ are connected by Jantzen coefficients.

Conversely, suppose that $c(\lambda, \mu)\neq0$. By definition \ref{jfdef1}, we can further assume that $\lambda, \mu$ are adjacent. Then Lemma \ref{jflem1} yields $[M_I(\lambda), \mu]>0$, while this implies $\lambda\lera\mu$ by Lemma \ref{jflem2}.
\end{proof}

\begin{remark}\label{jfrmk1}
Recall that the extensions between simple modules are determined by the leading coefficients (so called the $\mu$-function) of Kazhdan-Lusztig polynomials \cite{KL, V}. The calculation of $\mu$-function is very difficult, while the computation of Jantzen coefficients are relative easy \cite{XZ}. So Proposition \ref{jfprop1} largely simplifies the problem of blocks for category $\caO^\frp$.
\end{remark}

\begin{lemma}\label{jflem3}
If $c(\lambda, \mu)\neq0$ for $\lambda, \mu\in\lambda_I^+$, there exists $\beta\in\Psi_\lambda^+$ and $w\in W_I$ so that $\mu=ws_\beta\lambda$. Moreover, $\lambda\lera \mu$.
\end{lemma}

\begin{proof}
The second statement follows from Proposition \ref{jfprop1}. For the first one, note that $c(\lambda, \mu)=\sum_{\beta\in \Psi_{\lambda, \mu}^+}(-1)^{\ell(w_\beta)}\neq0$ in view of (\ref{jfeq3}). So $\Psi_{\lambda, \mu}^+$ is not empty. Choose any $\beta\in \Psi_{\lambda, \mu}^+$ and set $w=w_\beta$. We get $\mu=ws_\beta \lambda$.
\end{proof}

\begin{lemma}\label{jflem4}
Let $I, J\subset\Delta$. Suppose that $\lambda$ $($resp. $\mu)$ is a dominant weight with $\Phi_\lambda=\Phi_J$ $($resp. $\Phi_\mu=\Phi_I)$. Choose $x, w\in{}^IW^J$. Then $x\lambda\lera w\lambda$ $($relative to $(\Phi_I, \Phi))$ if and only if $x^{-1}\mu\lera w^{-1}\mu$ $($relative to $(\Phi_J, \Phi))$.
\end{lemma}
\begin{proof}
This follows immediately from Lemma \ref{jflem3} and the dual invariance of Jantzen coefficients (see Lemma 4.17 in \cite{XZ}).
\end{proof}

\begin{lemma}\label{jflem5}
Let $I, J\subset\Delta$. Suppose that $\Phi_I$ and $\Phi_{J}$ are $W$-conjugate, choose $w\in W$ so that $\Phi_{J}^+=w\Phi_I^+$. Let $\lambda, \mu\in\Lambda_I^+$. Then $\lambda\lera\mu$ $($relative to $(\Phi_{I}, \Phi))$ if and only if $w\lambda\lera w\mu$ $($relative to $(\Phi_{J}, \Phi))$.
\end{lemma}
\begin{proof}
This is an easy consequence of Lemma \ref{jflem3} and the conjugate invariance of Jantzen coefficients (see Lemma 4.18 in \cite{XZ}).
\end{proof}

%
%
\section{Linked roots}
%
%
From now on in this paper, $\frg$ is always one of the classical Lie algebras $\frgl(n, \bbC)$, $\frso(2n+1, \bbC)$, $\frsp(n, \bbC)$ and $\frso(2n, \bbC)$. Thus $\rank\frg=\dim\frh=n$ and $\Phi=A_{n-1}$, $B_n$, $C_n$ or $D_n$. The roots of $\Phi$ and weights of $\frh^*$ can be realized as vectors of $\bbC^n$ (\cite{H1}, \S12.1). With Proposition \ref{jfprop1}, we can study the problem of blocks for classical Lie algebras by taking advantage of the theory of Jantzen coefficients developed in \cite{XZ}.

\subsection{Linked roots and the reduction process} Inspired by Lemma \ref{jflem3}, we give the following definition.

\begin{definition}\label{lrdef1}
Suppose $\lambda, \mu\in\Lambda_I^+$. If $c(\lambda, \mu)\neq0$ and $\mu=ws_\beta\lambda$ for some $w\in W_I$ and $\beta\in\Psi_\lambda^+$, we say $\beta$ is a {\it linked root} from $\lambda$ to $\mu$ or $\lambda, \mu$ are {\it linked by $\beta$} (relative to $(\Phi_I, \Phi)$). In this situation, we will also say $w\beta$ is a {\it linked root} from $\mu$ to $\lambda$ (keeping in mind that $\lambda=w^{-1}s_{w\beta}\mu$).

In general, if $\beta$ is linked root from $\lambda$ to $\mu$ with $\mu=ws_\beta\lambda$ for $w\in W_I$, we say $w_1\beta$ is a linked root from $w_1\lambda$ to $w_2\mu$ for any $w_1, w_2\in W_I$, keeping in mind that
\[
w_2\mu=(w_2ww_1^{-1})s_{w_1\beta}(w_1\lambda).
\]
\end{definition}

\begin{remark}\label{lrrmk1}
If $c(\lambda, \mu)\neq0$, the definition means every $\beta\in\Psi_{\lambda, \mu}$ is a linked root from $\lambda$ to $\mu$.
\end{remark}

In view of Proposition \ref{jfprop1} and Lemma \ref{jflem3}, the following lemma is an immediate consequence of the definition.

\begin{lemma}\label{lrlem0}
Let $\lambda, \mu$ be two $\Phi_I$-regular weight. Then $\lambda\lera\mu$ if and only if we can find $\lambda=\lambda^0, \lambda^1, \cdots, \lambda^r=\mu$ with linked roots $\beta_i$ from $\lambda^{i-1}$ to $\lambda^{i}$ for $1\leq i\leq r$.
\end{lemma}

\begin{example}\label{lrex1}
Let $\frg=\frso(7, \bbC)$. Then $\Phi=B_3$. Using the standard parametrization (\cite{H1}), Choose $I=J=\{e_1-e_2, e_3\}$. Let $\lambda=(1, 0, 1)$ and $\mu=(0, -1, 1)$. Then $\Psi_{\lambda, \mu}^+=\{e_1, e_1+e_2, e_1+e_3\}$ and $c(\lambda, \mu)=1$. So $\lambda, \mu$ are linked by three roots $e_1, e_1+e_2, e_1+e_3$.

Now let $\frg=\frso(5, \bbC)$ and $\Phi=B_2$. Choose $I=\{e_2\}$ and $J=\{e_1-e_2\}$. Let $\lambda=(1, 1)$ and $\mu=(-1, 1)$. Then $\Psi_{\lambda, \mu}^+=\{e_1, e_1+e_2\}$ and $c(\lambda, \mu)=0$. Thus $\lambda, \mu$ are not linked.
\end{example}

If $\Phi=D_n$ $(n\geq4)$, $e_{n-1}-e_n\not\in I$ and $e_{n-1}+e_n\in I$, we say $I$ is \emph{not standard}; otherwise $I$ is \emph{standard}. The notation get more complicated when $I$ is not standard. Normally we send a nonstandard $I$ to a standard one by the isomorphism
\begin{equation}\label{lreq0}
\vf: \frh^*\ra\frh^*\ \mbox{with}\ \vf(\lambda)=s_{e_n}\lambda.
\end{equation}
It interchanges $e_{n-1}\pm e_n$ and fixes the others.

\begin{remark}\label{lrrmk2}
Because of the category equivalence between $\caO_\lambda^{\frp_I}$ and $\caO_{\vf(\lambda)}^{\frp_{\vf(I)}}$, the results for $(\Phi, \Phi_I, \Phi_J)$ work equally well for $(\Phi, \Phi_{\vf(I)}, \Phi_{\vf(J)})$. We only need to consider the standard $I$ in most cases. However, even when $I$ is standard, it is possible that $J$ is not standard. One should be aware of this since some of our arguments depend on the dual relations between $(\Phi, \Phi_I, \Phi_J)$ and $(\Phi, \Phi_J, \Phi_I)$.
\end{remark}

We need more notations to describe all the linked roots. Suppose that $\Delta\backslash I=\{\alpha_{q_1},\ldots, \alpha_{q_{m-1}}\}$ with $1\leq q_1<\ldots<q_{m-1}\leq n$ ($q_{m-1}<n$ for $\Phi=A_{n-1}$). Set $q_0=0$ and $q_m=n+1$. If $I$ is standard or $s<m-1$, let $I_s=\{\alpha_i\in I\ |\ q_{s-1}<i<q_s\}$. If $I$ is not standard, then $q_{m-1}=n-1$. In this case, set $I_{m-1}=\{\alpha_i\in I\ |\ q_{m-2}<i<q_{m-1}\}\cup\{\alpha_n\}$ and $I_m=\emptyset$. We have
\[
\Phi_I=\bigsqcup_{s=1}^m \Phi_{I_s}.
\]
Put $n_s=|I_s|+1$ for $1\leq s<m$ and $n_m=|I_m|$. In particular, $n_m\neq1$ for $\Phi=D_n$. If $I$ is standard, then $n_s=q_s-q_{s-1}$ for $1\leq s<m$ and $n_m=n-q_{m-1}$.
Any vector $\lambda\in\frh^*$ can be divided into $m$ segments $(\lambda_{q_{s-1}+1}, \ldots, \lambda_{q_{s-1}+n_s})$ for $1\leq s\leq m$. If $\lambda\in\Lambda_I^+$, one has
\begin{equation}\label{lreq1}
\lambda_{q_{s-1}+1}>\lambda_{q_{s-1}+2}>\ldots>\lambda_{q_{s-1}+n_s}
\end{equation}
for $1\leq s<m$ or $\Phi=A_{n-1}$ and
\begin{equation}\label{lreq2}
\lambda_{q_{m-1}+1}>\lambda_{q_{m-1}+2}>\ldots>\lambda_n>0
\end{equation}
for $\Phi=B_n$, $C_n$ with $q_{m-1}<n$ and
\begin{equation}\label{lreq3}
\lambda_{q_{m-1}+1}>\lambda_{q_{m-1}+2}>\ldots>\lambda_{n-1}>|\lambda_n|
\end{equation}
for $\Phi=D_n$ with $q_{m-1}<n-1$. The following result is evident.

\begin{lemma}\label{lrlem1}
Let $I\subset\Delta$ be standard. Then $\lambda\in\frh^*$ is $\Phi_I$-regular if and only if the following conditions are satisfied:
\begin{itemize}
\item [(1)] $\lambda_i\neq\lambda_j$ for $q_{s-1}<i<j\leq q_s$ when $s<m$ or $\Phi=A_{n-1}$;

\item [(2)] $|\lambda_i|\neq|\lambda_j|$ for $q_{m-1}<i<j\leq n$ and $|\lambda_i|\neq 0$ for $q_{m-1}<i\leq n$ when $\Phi=B_n$ or $C_n$;

\item [(3)] $|\lambda_i|\neq|\lambda_j|$ for $q_{m-1}<i<j\leq n$ when $\Phi=D_n$.
\end{itemize}
\end{lemma}

Let $\lambda$ be a $\Phi_I$-regular weight with $\Phi_{\overline\lambda}=\Phi_J$. Write $\Delta\backslash J=\{\alpha_{\overline q_1}, \dots, \alpha_{\overline q_{\overline m-1}}\}$ for $1\leq \overline q_1<\ldots< \overline q_{\overline m-1}\leq n$. Put $\overline q_0=0$ and $\overline q_{\overline m}=n+1$. If $J$ is standard or $s<\overline m-1$, let $J_s=\{\alpha_i\in J\ |\ \overline q_{s-1}<i< \overline q_s\}$. If $J$ is not standard, let $J_{\overline m-1}=\{\alpha_i\in I\ |\ \overline q_{\overline m-2}<i<n-1\}\cup\{\alpha_n\}$ and $J_m=\emptyset$. So
\begin{equation*}
\Phi_J=\bigsqcup_{s=1}^{\overline m} \Phi_{J_s}.
\end{equation*}
Set $a_s=\overline\lambda_{\overline q_s}$ and for $1\leq s< \overline m$ and $a_{\overline m}=0$. Then $a_1>\ldots>a_{\overline m-1}\geq a_{\overline m}=0$. Let
\begin{equation*}
\caA=\{a_1, a_2, \ldots, a_{\overline m-1}, a_{\overline m}\}.
\end{equation*}
We also set $\overline n_s=|J_s|+1$ for $1\leq s<\overline m$ and $\overline n_{\overline m}=|J_{\overline m}|$.

\begin{remark}\label{lrrmk3}
One might have $a_{\overline m-1}=0=a_{\overline m}$ for some very special categories $\caO^\frp_\lambda$ ($\Phi=D_n$, $\overline q_{\overline m-2}=n-1$ and $\overline q_{\overline m-1}=n$). This will cause some problems on notation. Fortunately, any of these categories is isomorphic to some other category corresponding to $(\Phi, \Phi_I, \Phi_J)$ with $a_{\overline m-1}>0$ (see Remark 7.3 in \cite{XZ}).

For convenience, we will always assume that $a_{\overline m-1}>0$ in this paper.
\end{remark}

If $\Phi=A_n$, we set
\begin{equation*}
n_s^\lambda(a)=|\{q_{s-1}<i\leq q_s\mid \lambda_i=a\}|,
\end{equation*}
for any $a\in\caA$ and $1\leq s\leq m$. If $\Phi\neq A_n$ and $I$ is standard, we set
\begin{equation}\label{lreq4}
n_s^\lambda(a)=|\{q_{s-1}<i\leq q_s\mid \lambda_i=a\ \mbox{or}\ -a\}|.
\end{equation}
If $I$ is not standard, the definition of $n_s^\lambda(a)$ is the same as (\ref{lreq4}) except that $n_{m-1}^\lambda(a)=|\{q_{m-2}<i\leq n\mid \lambda_i=a\ \mbox{or}\ -a\}|$ and $n_m^\lambda(a)=0$. In particular, $n_s^\lambda(a)$ is invariant under the map $\vf(\cdot)$ when $\Phi=D_n$. In view of Lemma \ref{lrlem1}, the following results are evident.

\begin{lemma}\label{lrlem2}
Let $I, J\subset\Delta$ and $\lambda$ be a $\Phi_I$-regular weight with $\Phi_{\overline\lambda}=\Phi_J$. Then
\begin{itemize}
  \item [(\rmnum{1})] $n_s^{w\lambda}(a_t)=n_s^{\lambda}(a_t)$ for $1\leq s\leq m$, $1\leq t\leq\overline m$ and $w\in W_I$;
  \item [(\rmnum{2})] $n_s=\sum_{t=1}^{\overline m}n_s^\lambda(a_t)$ for $1\leq s\leq m$;
  \item [(\rmnum{3})] $\overline n_t=\sum_{s=1}^mn_s^\lambda(a_t)$ for $1\leq t\leq \overline m$;
  \item [(\rmnum{4})] If $\Phi=A_{n-1}$ or $s=m$ or $a=0$, then $n_s^\lambda(a)\leq1$, while $n_m^\lambda(0)=0$ for $\Phi=B_n, C_n$. In general $n_s^\lambda(a)\leq2$.
\end{itemize}
\end{lemma}

Fix $a\in\caA$ and $1\leq s\leq m$. We say $n_s^\lambda(a)$ is \emph{maximal} if
\begin{equation}\label{bklkeq3}
n_s^\lambda(a)=\max\{n_s^\mu(a)\mid \mu\ \mbox{is}\ \Phi_I\mbox{-regular}\}.
\end{equation}
Then Lemma \ref{lrlem2}(\rmnum{4}) shows that the maximal value must be $0, 1$ or $2$. The following proposition is an easy consequence of Theorem 7.27 in \cite{XZ}.

\begin{prop}\label{lrprop1}
Choose $\lambda\in\Lambda_I^+$ and $\beta\in\Psi_\lambda^+$ so that $s_\beta\lambda$ is $\Phi_I$-regular. Then $\beta$ is a linked root from $\lambda$ to $s_\beta\lambda$ except the following cases:
\begin{itemize}
\item [(1)] $\Phi=B_n$ $($resp. $C_n)$, $n^\lambda_s(a)=1$, $n^\lambda_s(0)+n^\lambda_m(a)=1$ for $1\leq s<m$ and $0<a\in\caA$. Moreover, $\beta$ satisfies one of the following conditions:
    \begin{itemize}
    \item [(1a)] $\beta=e_i$ $($resp. $2e_i)$ or $e_i+e_j$ for $q_{s-1}<i<j\leq q_s$ with $\lambda_i=a>0=\lambda_j$.

    \item [(1b)] $\beta=e_i$ $($resp. $2e_i)$  or $e_i+e_j$ for $q_{s-1}<i\leq q_s\leq q_{m-1}<j\leq n$ with $\lambda_i=\lambda_j=a$.
    \end{itemize}

\item [(2)] $\Phi=D_n$, $n^\lambda_s(a)=n^\lambda_m(a)=1$ and $n^\lambda_s(0)=n^\lambda_m(0)=1$ for $1\leq s<m$ and $0<a\in\caA$. Moreover, $\beta=e_i+e_j$ or $e_i+e_k$ for $q_{s-1}<i<j\leq q_s\leq q_{m-1}<k<n$ with $\lambda_i=\lambda_k=a$ and $\lambda_j=\lambda_n=0$.
\end{itemize}
\end{prop}

\begin{remark}\label{lrrmk4}
Assume that $\lambda$ is $\Phi_I$-regular and $\langle\lambda, \beta^\vee\rangle>0$ for some $\beta\in\Phi^+\backslash\Phi_I$. There exists $w\in W_I$ so that $w\lambda\in\Lambda_I^+$ and $w\beta\in\Psi_{w\lambda}^+$. We can obtain linked roots for $\Phi_I$-regular weights by applying Proposition \ref{lrprop1}. Moreover, if $\beta$ is a linked root from $\lambda$ to $s_\beta\lambda$, it is also a linked root from $s_\beta\lambda$ to $\lambda$. These observations give us the following generalization.
\end{remark}

\begin{prop}\label{lrprop2}
Choose $\beta\in\Phi^+\backslash\Phi_I$. Suppose that $\lambda, s_\beta\lambda$ are distinct $\Phi_I$-regular weights. Then $\beta$ is a linked root from $\lambda$ to $s_\beta\lambda$ except the following cases:
\begin{itemize}
\item [(1)] $\Phi=B_n$ $($resp. $C_n)$, $n^\lambda_s(a)=1$, $n^\lambda_s(0)+n^\lambda_m(a)=1$ for $1\leq s<m$ and $0<a\in\caA$. Moreover, $\beta$ satisfies one of the following conditions:
    \begin{itemize}
    \item [(1a)] $\beta=e_i$ $($resp. $2e_i)$ or $e_i+e_j$ for $q_{s-1}<i, j\leq q_s$ with $|\lambda_i|=a>0=\lambda_j$.

    \item [(1b)] $\beta=e_i$ $($resp. $2e_i)$  or $e_i+e_j$ for $q_{s-1}<i\leq q_s\leq q_{m-1}<j\leq n$ with $\lambda_i=\lambda_j=\pm a$.

    \item [(1c)] $\beta=e_i$ $($resp. $2e_i)$  or $e_i-e_j$ for $q_{s-1}<i\leq q_s\leq q_{m-1}<j\leq n$ with $\lambda_i=-\lambda_j=\pm a$.
    \end{itemize}

\item [(2)] $\Phi=D_n$, $n^\lambda_s(a)=n^\lambda_m(a)=1$ and $n^\lambda_s(0)=n^\lambda_m(0)=1$ for $1\leq s<m$ and $0<a\in\caA$. Moreover, $\beta$ satisfies one of the following conditions:
    \begin{itemize}
    \item [(2a)] $\beta=e_i+e_j$ or $e_i+e_k$ for $q_{s-1}<i, j\leq q_s\leq q_{m-1}<k,l\leq n$ with $\lambda_i=\lambda_k=\pm a$ and $\lambda_j=\lambda_l=0$.

    \item [(2b)] $\beta=e_i+e_j$ or $e_i-e_k$ for $q_{s-1}<i, j\leq q_s\leq q_{m-1}<k,l\leq n$ with $\lambda_i=-\lambda_k=\pm a$ and $\lambda_j=\lambda_l=0$.
    \end{itemize}
\end{itemize}
\end{prop}

\subsection{Criteria of connectedness} In this subsection, we will give several connectedness criterion based on the previous results about linked roots. Although Proposition \ref{lrprop2} is fairly complete, it may not be so convenient to use it to verify all the linked roots. So we give the following sufficient condition.

\begin{lemma}\label{cclem1}
Let $\Phi=B_n, C_n$ or $D_n$ and $I$ be standard. Assume that $\beta=e_i\pm e_j\in\Phi$ for $q_{s-1}<i\leq q_s$ and $q_{t-1}<j\leq q_t$. Suppose that $\lambda, s_\beta\lambda$ are different $\Phi_I$-regular weights. They are linked by $\beta$ when one of the following conditions holds:
\begin{itemize}
\item [(1)] $s\neq t$ and $|\lambda_i|\neq |\lambda_j|$;

\item [(2)] $s\neq t$, $s, t<m$ and $|\lambda_i|=|\lambda_j|$;

\item [(3)] $s=t$, $|\lambda_i|, |\lambda_j|>0$ and $|\lambda_i|\neq |\lambda_j|$.
\end{itemize}
\end{lemma}
\begin{proof}
If $\beta\in\Phi_I$, there is nothing to prove. If $\beta\in\Phi\backslash\Phi_I$, the corollary is an immediate consequence of Proposition \ref{lrprop2}.
\end{proof}

If $\lambda, s_\beta\lambda$ are $\Phi_I$-regular, Proposition \ref{lrprop2} shows that they are linked except some very special cases. The following example shows that even two weights are not liked, it is still possible that they are connected. Because they might be connected by a chain of linked roots.

\begin{example}\label{ccex1}
Let $\Phi=B_3$ and $I=\{e_1-e_2\}$. Suppose that $\lambda=(1, 0\ |\ 2)$. Then $s_{e_1}\lambda=(-1, 0\ |\ 2)$. Note that the vertical lines separate different segment of the weights. In view of Proposition \ref{lrprop2}, $\lambda$ and $s_{e_1}\lambda$ are not linked. However, we still have $\lambda\lera s_{e_1}\lambda$ since
\[
\lambda=(1, 0\ |\ 2)\stackrel{\underleftrightarrow{e_1-e_3}}{}(2, 0\ |\ 1) \stackrel{\underleftrightarrow{e_2-e_3}}{}(2, 1\ |\ 0)
\stackrel{\underleftrightarrow{e_2+e_3}}{}(2, 0\ |\ -1)
\stackrel{\underleftrightarrow{e_1-e_3}}{}(-1, 0\ |\ 2)=s_{e_1}\lambda.
\]
In view of Prop \ref{lrprop2}, the roots marked on arrows are linked roots of corresponding weights.
\end{example}

Inspired by the above example, we present two more criteria of connectedness.

\begin{lemma}\label{cclem2}
Let $\Phi=B_n, C_n$ or $D_n$ and $I$ be standard. Let $\lambda$ be a $\Phi_I$-regular weight. Denote $a=|\lambda_i|$ for $1\leq i\leq n$. There is $1\leq s\leq m$ such that $q_{s-1}<i\leq q_s$. Assume that $a\neq0$. Then $\lambda\lera s_{e_i}\lambda$ when all the following conditions hold:
\begin{itemize}
\item [(1)] $n^\lambda_s(a)=1$;

\item [(2)] $n^\lambda_s(0)+n^\lambda_t(0)=1$;

\item [(3)] $1\leq n^\lambda_t(0)+n^\lambda_t(a)\leq 2$;
\end{itemize}
for some $1\leq t\leq m$.
\end{lemma}
\begin{proof}
Obviously (2) yields $t\neq s$. First assume that $n^\lambda_s(0)=1$. Then $n^\lambda_t(0)=0$ by (2) and $1\leq n^\lambda_t(a)\leq 2$ by (3). There are $q_{s-1}<j\leq q_s$ with $\lambda_j=0$ and $q_{t-1}<k\leq q_t$ with $|\lambda_k|=a$. In view of Lemma \ref{lrlem1}, if $s, t<m$, there exists $\beta\in\{e_j\pm e_k\}$ so that $s_\beta\lambda$ is $\Phi_I$-regular ($\beta=e_j+e_k$ when $\lambda_k=\lambda_i$ and $\beta=e_j-e_k$ when $\lambda_k=-\lambda_i$, one always have $(s_\beta\lambda)_j=-\lambda_i=-(s_\beta\lambda)_i$). Let $\gamma\neq\beta$ be the other root contained in $\{e_j\pm e_k\}$. Set $\mu=s_{e_i-e_j}s_\beta\lambda$. If $\beta=e_j+e_k$, then $\gamma=e_j-e_k$. One has $\mu_i=(s_\beta\lambda)_j=-\lambda_k=-\lambda_i$, $\mu_j=(s_\beta\lambda)_i=\lambda_i=\lambda_k$, $\mu_k=(s_\beta\lambda)_k=0$ and $\mu_l=\lambda_l$ for $l\neq i, j, k$. Therefore
\[
\lambda\lera s_\beta\lambda\lera s_{e_i-e_j}s_\beta\lambda\lera s_\gamma s_{e_i-e_j}s_\beta\lambda=s_{e_i}\lambda,
\]
keeping in mind of Lemma \ref{cclem1}(\rmnum{1}) and $e_i-e_j\in\Phi_I$. If $\lambda_i=a$, we can illustrate the transformation of entries involved here (the case $\lambda_i=-a$ is similar):
\begin{equation}\label{cceq1}
(\lambda_i, \lambda_j\ |\ \lambda_k)=(a, 0\ |\ a)\stackrel{\underleftrightarrow{e_j+e_k}}{}(a, -a\ |\ 0)\stackrel{\underleftrightarrow{e_i-e_j}}{}(-a, a\ |\ 0)\stackrel{\underleftrightarrow{e_j-e_k}}{}(-a, 0\ |\ a).
\end{equation}
In a similar spirit, we can obtain $\lambda\lera s_{e_i}\lambda$ when $\beta=e_j-e_k$. If $s<t=m$ and $\Phi=D_n$, the argument is similar to the case $s, t<m$. If $s<t=m$ and $\Phi=B_n$ (resp. $C_n$), $\nu=s_{e_k\pm e_j}\lambda$ are not $\Phi_I$-regular (since $\nu_k=0$ and $n_m^\nu(0)=1$, contradicts Lemma \ref{lrlem2}). Fortunately, now $e_i$ (resp. $2e_i$) is a linked root from $\lambda$ to $s_{e_i}\lambda$ by Proposition \ref{lrprop2} (with $n_s^\lambda(0)+n_m^\lambda(a)=2$). If $s=m$, then $\lambda\lera s_{e_i-e_j}s_{e_i+e_j}\lambda=s_{e_i}\lambda$ since $e_i\pm e_j\in\Phi_{I_m}\subset\Phi_I$.

Next assume that $n^\lambda_s(0)=0$. Then $n^\lambda_t(0)=1$ by (2) and $n^\lambda_t(a)\leq 1$ by (3). There exists $q_{t-1}<j\leq q_t$ with $\lambda_j=0$. If $s, t<m$, we can choose $\beta\in\{e_i\pm e_j\}$ such that $s_\beta\lambda$ is $\Phi_I$-regular. Let $\gamma\neq \beta$ be the other root of $\{e_i\pm e_j\}$. We obtain $\lambda\lera s_\beta\lambda\lera s_\gamma s_\beta\lambda=s_{e_i}\lambda$ by Lemma \ref{cclem1}. If $s=m$ and $\Phi=D_n$, the proof is similar. If $s=m$ and $\Phi=B_n$ (resp. $C_n$), we still get $\lambda\lera s_{e_i}\lambda$ since $e_i\in\Phi_I$ ($2e_i\in\Phi_I$). If $t=m$, we get $e_j\in\Phi_I$ (resp. $2e_j\in\Phi_I$) for $\Phi=B_n$ (resp. $C_n$) and thus $\lambda_j\neq0$, a contradiction. This forces $\Phi=D_n$. In this case, if $n^\lambda_m(a)=0$, then $\lambda\lera s_{e_i-e_j}s_{e_i+e_j}\lambda=s_{e_i}\lambda$. If $n^\lambda_m(a)=1$, there exists $q_{m-1}<k\leq n$ with $|\lambda_k|=a$. If $\lambda_k=\lambda_i$, then $\lambda\lera s_{e_k+e_j}s_{e_k-e_j}s_{e_i+e_k}\lambda=s_{e_i}\lambda$ by Proposition \ref{lrprop2} and $e_k\pm e_j\in\Phi_I$. If $\lambda_k=-\lambda_i$, then $\lambda\lera s_{e_k+e_j}s_{e_k-e_j}s_{e_i-e_k}\lambda=s_{e_i}\lambda$ by a similar argument.
\end{proof}

\begin{remark}\label{ccrmk1}
The formula (\ref{cceq1}) helps us to understand the corresponding argument more clearly. It could be very useful if one keeps this kind of formula in mind while going through a similar argument.
\end{remark}

\begin{lemma}\label{cclem3}
Let $\Phi=B_n, C_n$ or $D_n$ and $I$ be standard. Denote $a=|\lambda_i|$ and $b=|\lambda_j|$ for $1\leq i, j\leq q_{m-1}$. There exist $1\leq s, t<m$ such that $q_{s-1}<i\leq q_s$ and $q_{t-1}<j\leq q_t$. Assume that $ab\neq0$. Then $\lambda\lera s_{e_i}s_{e_j}\lambda$ when the following conditions hold:
\begin{itemize}
\item [(1)] $n^\lambda_s(a)=n^\lambda_t(b)=1$;

\item [(2)] $n^\lambda_s(b)\leq1$ or $n^\lambda_t(a)\geq 1$;

\item [(3)] $n^\lambda_s(b)\geq 1$ or $n^\lambda_t(a)\leq 1$.
\end{itemize}
\end{lemma}
\begin{proof}
If $i=j$, there is nothing to prove. So we can assume that $i\neq j$. If $a=b$, then $s\neq t$ by (1)(note that $i\neq j$). So $s_{e_i\pm e_j}\lambda=\lambda, s_{e_i}s_{e_j}\lambda$ are $\Phi_I$-regular (see Lemma \ref{lrlem1}) and $\lambda\lera s_{e_i\pm e_j}\lambda$ by Lemma \ref{cclem1}(2). If $s=t$, then $a\neq b$ by (1), which also implies that $s_{e_i+e_j}\lambda$ is $\Phi_I$-regular. Thus $\lambda\lera s_{e_i+e_j}\lambda\lera s_{e_i-e_j}s_{e_i+e_j}\lambda=s_{e_i}s_{e_j}\lambda$ by Lemma \ref{cclem1}(3) and $e_i-e_j\in\Phi_I$. From now on assume that $a\neq b$ and $s\neq t$. Obviously, $s_{e_i}s_{e_j}\lambda$ is $\Phi_I$-regular in view of (1). By symmetry, we can also assume that $n^\lambda_s(b)\geq n^\lambda_t(a)$.

If $n^\lambda_s(b)=0$, then $n^\lambda_t(a)=0$. It is easy to verify that $s_{e_i+e_j}\lambda$ is $\Phi_I$-regular. In view of Lemma \ref{cclem1}, $e_i+e_j$ is a linked root from $\lambda$ and $e_i-e_j$ is a linked root from $s_{e_i+e_j}\lambda$. It follows that
\[
\lambda\lera s_{e_i+e_j}\lambda\lera s_{e_i-e_j}s_{e_i+e_j}\lambda=s_{e_i}s_{e_j}\lambda.
\]

If $n^\lambda_s(b)=1$, then $n^\lambda_t(a)=0$ or $1$. There exists $q_{s-1}<k\leq q_s$ with $|\lambda_k|=b$. Moreover, $s_{e_k\pm e_j}\lambda=\lambda, s_{e_k}s_{e_j}\lambda$ are $\Phi_I$-regular, so are $s_{e_i+e_k}s_{e_k\pm e_j}\lambda$. If $\lambda_k=\lambda_i$, then
\[
\lambda\lera s_{e_k+e_j}\lambda\lera s_{e_i+e_k}s_{e_k+e_j}\lambda\lera s_{e_i-e_k}s_{e_i+e_k}s_{e_k+e_j}\lambda=s_{e_i}s_{e_j}\lambda
\]
by Lemma \ref{cclem1} and $e_i-e_k\in\Phi_I$. If $\lambda_k=-\lambda_i$, then
\[
\lambda\lera s_{e_k-e_j}\lambda\lera s_{e_i+e_k}s_{e_k-e_j}\lambda\lera s_{e_i-e_k}s_{e_i+e_k}s_{e_k-e_j}\lambda=s_{e_i}s_{e_j}\lambda.
\]

If $n^\lambda_s(b)=2$, then $n^\lambda_t(a)\geq1$ by (2). First assume that $n^\lambda_t(a)=1$, there is $q_{t-1}\leq k\leq q_t$ with $|\lambda_k|=a$. We can get
\[
\lambda\lera s_{e_i+e_k}\lambda\lera s_{e_k+e_j}s_{e_i+e_k}\lambda\lera s_{e_k-e_j}s_{e_k+e_j}s_{e_i+e_k}\lambda=s_{e_i}s_{e_j}\lambda
\]
for $\lambda_k=\lambda_i$ and
\[
\lambda\lera s_{e_i-e_k}\lambda\lera s_{e_k+e_j}s_{e_i-e_k}\lambda\lera s_{e_k-e_j}s_{e_k+e_j}s_{e_i-e_k}\lambda=s_{e_i}s_{e_j}\lambda
\]
for $\lambda_k=-\lambda_i$ by a similar argument.

Now suppose that $n^\lambda_s(b)=n^\lambda_t(a)=2$.
There exist $q_{s-1}< k, l\leq q_s$ and $q_{t-1}< p, r\leq q_t$ with $\lambda_k=b$, $\lambda_l=-b$, $\lambda_p=a$ and $\lambda_r=-a$. If $\lambda_i=a$ and $\lambda_j=b$, then both $s_{e_l-e_r}\lambda$ and $s_{e_i+e_j}s_{e_l-e_r}\lambda$ are $\Phi_I$-regular. With Lemma \ref{cclem1} and $e_j-e_r, e_i-e_l\in\Phi_I$, one has
\[
\lambda\lera s_{e_l-e_r}\lambda\lera s_{e_i+e_j}s_{e_l-e_r}\lambda\lera s_{e_j -e_r}s_{e_i-e_l}s_{e_i+e_j}s_{e_l-e_r}\lambda=s_{e_i}s_{e_j}\lambda.
\]
To make the argument more clear, we give entries involved here:
\[
\begin{aligned}
(\lambda_i, \lambda_k, \lambda_l\ |\ \lambda_p, \lambda_r, \lambda_j)=&(a, b, -b\ |\ a, -a, b)\stackrel{\underleftrightarrow{e_l-e_r}}{}(a, b, -a\ |\ a, -b, b)\\
\stackrel{\underleftrightarrow{e_i+e_j}}{}&(-b, b, -a\ |\ a, -b, -a)\stackrel{\underleftrightarrow{e_i-e_l}}{}(-a, b, -b\ |\ a, -b, -a)\\
\stackrel{\underleftrightarrow{e_j-e_r}}{}&(-a, b, -b\ |\ a, -a, -b)=(-\lambda_i, \lambda_k, \lambda_l\ |\ \lambda_p, \lambda_r, -\lambda_j).
\end{aligned}
\]
\end{proof}

Some other cases are covered by following useful result.

\begin{lemma}\label{cclem4}
Let $\Phi=B_n, C_n$ or $D_n$ and $I$ be standard. Let $\lambda$ be a $\Phi_I$-regular weight. Choose $1\leq s, t\leq m$ and $0\leq a, b\in\caA$ with $s\neq t$ and $a\neq b$. Suppose that
\begin{itemize}
\item [(\rmnum{1})] Both $n^\lambda_s(a)$ and $n^\lambda_t(b)$ are nonzero;

\item [(\rmnum{2})] Both $n^\lambda_s(b)$ and $n^\lambda_t(a)$ are not maximal;

\item [(\rmnum{3})] $n^\lambda_s(a)+n^\lambda_t(b)>n^\lambda_s(b)+n^\lambda_t(a)$.
\end{itemize}
There exists $\beta=e_i\pm e_j\in\Phi$ so that the following conditions hold:
\begin{itemize}
\item [(1)] $|\lambda_i|=a$, $|\lambda_j|$=b with $q_{s-1}<i\leq q_s$ and $q_{t-1}<j\leq q_t$.

\item [(2)] $\mu=s_\beta\lambda$ is $\Phi_I$-regular and $\beta$ is a linked root from $\lambda$ to $\mu$;

\item [(3)] $n^\mu_s(a)=n^\lambda_s(a)-1$, $n^\mu_s(b)=n^\lambda_s(b)+1$ and $n^\mu_s(c)=n^\lambda_s(c)$ for $c\neq a, b$.
\end{itemize}
\end{lemma}
\begin{proof}
If (1) holds and $s_\beta\lambda$ is $\Phi_I$-regular, then $\beta$ is a linked root by Lemma \ref{cclem1} and (3) follows from (1). With (\rmnum{1}), we can find $q_{s-1}<i\leq q_s$ and $q_{t-1}<j\leq q_t$ such that $|\lambda_i|=a$ and $|\lambda_j|=b$.

Keeping in mind of Lemma \ref{lrlem1}, if $n^\lambda_s(b)=n^\lambda_t(a)=0$, then both $s_{e_i+e_j}\lambda$ and $s_{e_i-e_j}\lambda$ are $\Phi_I$-regular. Choose $\beta=e_i+e_j$. If $n^\lambda_s(b)=0$ and $n^\lambda_t(a)=1$, then (\rmnum{2}) implies $t<m$ and $a>0$ (see Lemma \ref{lrlem2}). There exists $q_{t-1}<k\leq q_t$ such that $|\lambda_k|=a$. We choose $\beta=e_i+e_j$ when $\lambda_k=\lambda_i$ and $\beta=e_i-e_j$ when $\lambda_k=-\lambda_i$. Since $n^\lambda_s(b)=0$ is not maximal, Lemma \ref{lrlem1} implies that $s_{\beta}\lambda$ is $\Phi_I$-regular. If $n^\lambda_s(b)=1$ and $n^\lambda_t(a)=0$, the argument the similar. If $n^\lambda_s(b)=n^\lambda_t(a)=1$, then (\rmnum{2}) implies $s, t<m$ and $a, b>0$. We can find $q_{s-1}<k\leq q_{s}$ and $q_{t-1}<l\leq q_t$ so that $|\lambda_k|=b$ and $|\lambda_l|=a$. It follows from (\rmnum{3}) that $n^\lambda_s(a)=2$ or $n^\lambda_t(b)=2$. By symmetry, it suffices to consider the case $n^\lambda_s(a)=2$. There exists $q_{s-1}<r\leq q_s$ with $\lambda_r=-\lambda_i$. If $\lambda_i=\lambda_l$ and $\lambda_j=\lambda_k$, choose $\beta=e_i+e_j$. If $\lambda_i=-\lambda_l$ and $\lambda_j=-\lambda_k$, choose $\beta=e_i-e_j$. If $\lambda_i=\lambda_l$ and $\lambda_j=-\lambda_k$, choose $\beta=e_r-e_j$ and interchange the index $r$ with $i$ (note that $|\lambda_r|=|\lambda_i|=a$). If $\lambda_i=-\lambda_l$ and $\lambda_j=\lambda_k$, choose $\beta=e_r+e_j$ and interchange $r$ with $i$. In either case, $s_\beta\lambda$ is $\Phi_I$-regular.
\end{proof}

\begin{example}\label{ccex3}
The above lemma will be applied in our main results. The condition (\rmnum{3}) may seem odd at first sight. It is used to rule out the following counterexample. Suppose that $\Phi=D_4$ and $I=\{e_1-e_2, e_3-e_4\}$. Let $\lambda=(2, 1\ |\ 2, -1)$. Then $\lambda\in\Lambda_I^+$. With $(s, t)=(1, 2)$ and $(a, b)=(2, 1)$, we get $n^\lambda_s(a)=n^\lambda_s(b)=n^\lambda_t(a)=n^\lambda_t(b)=1$. The conditions (\rmnum{1}) and (\rmnum{2}) in the above lemma are satisfied, while (\rmnum{3}) is failed. In this case, we can not find $\beta\in\Phi$ such that (1)-(3) hold.
\end{example}

\begin{remark}\label{ccrmk2}
In view of Proposition \ref{lrprop2}, it is expected that the subcategory $\caO^\frp_\lambda$ or the system $(\Phi, \Phi_I, \Phi_J)$ for classical Lie algebras contains only one block in most cases. One might wonder what kind of system contains more blocks. How many blocks could it have? Is there an upper limit for the number of blocks? For classical Lie algebras, there are many examples of systems having two blocks \cite{ES, P1}. Platt \cite{P1} also described a system containing four blocks (which he called linkage classes there). Now we are going to give a general construction in the following example.
\end{remark}

\begin{example}\label{ccex2}
First, let $\Phi=B_2$ and $I=\{e_1-e_2\}=\Delta\backslash\{\alpha_2\}$. Let $\lambda=(1, 0)$ with $J=\{e_2\}$. Then ${}^IW^J\lambda=W\lambda\cap\Lambda_I^+$ contains two weights $\lambda^1=(1, 0)$ and $\lambda^2=(0, -1)$.  By Theorem 7.5 in \cite{XZ} or Jantzen's simplicity criteria \cite{J2}, the generalized Verma modules $M_I(\lambda^1)$ and $M_I(\lambda^2)$ are simple. Thus the category $\caO_\lambda^{\frp_I}$ is semisimple and has two blocks.

Next, let $\Phi=B_6$ and $I=\{e_1-e_2, e_3-e_4, e_4-e_5, e_5-e_6\}=\Delta\backslash\{\alpha_2, \alpha_6\}$, that is, $q_1=2$ and $q_2=6$. So $m=3$. It follows that $n_1=2$, $n_2=4$ and $n_3=0$. Let $\lambda=(2, 1, 1, 1, 0, 0)$. Then $\Phi_\lambda=\Phi_J$, where $J=\{e_2-e_3, e_3-e_4, e_5-e_6, e_6\}=\Delta\backslash\{\alpha_1, \alpha_4\}$. If $\mu\in{}^IW^J\lambda=W\lambda\cap\Lambda_I^+$, it can be verified that $\mu$ must be one of the following four weights
\[
\begin{aligned}
\lambda^1=&(1, 0\ |\ 2, 1, 0, -1),&\ \lambda^3&=(0, -1\ |\ 2, 1, 0, -1)\\
\lambda^2=&(1, 0\ |\ 1, 0, -1, -2),&\ \lambda^4&=(0, -1\ |\ 1, 0, -1, -2)
\end{aligned}
\]
All the generalized Verma modules $M_I(\lambda^l)$ ($1\leq l\leq 4$) are simple (in view of \cite{XZ, J2}). The category $\caO_\lambda^{\frp_I}$ is also semisimple and has four blocks.

Then let $\Phi=B_{12}$, $I=\Delta\backslash\{\alpha_2, \alpha_6, \alpha_{12}\}$ and $J=\Delta\backslash\{\alpha_1, \alpha_4, \alpha_9\}$. Choose $\lambda=(3, 2, 2, 2, 1, 1, 1, 1, 1, 0, 0, 0)$. Then $\caO_\lambda^{\frp_I}$ has eight blocks, each of which contains exactly a (simple) generalized Verma modules $M_I(\lambda^l)$ ($1\leq l\leq 8$), where
\[
\begin{aligned}
\lambda^1&=(1, 0\ |\ 2, 1, 0, -1\ |\ 3, 2, 1, 0, -1, -2),\\
\lambda^2&=(1, 0\ |\ 2, 1, 0, -1\ |\ 2, 1, 0, -1, -2, -3),\\
\lambda^3&=(1, 0\ |\ 1, 0, -1, -2\ |\ 3, 2, 1, 0, -1, -2),\\
\lambda^4&=(1, 0\ |\ 1, 0, -1, -2\ |\ 2, 1, 0, -1, -2, -3),\\
\lambda^5&=(0, -1\ |\ 2, 1, 0, -1\ |\ 3, 2, 1, 0, -1, -2),\\
\lambda^6&=(0, -1\ |\ 2, 1, 0, -1\ |\ 2, 1, 0, -1, -2, -3),\\
\lambda^7&=(0, -1\ |\ 1, 0, -1, -2\ |\ 3, 2, 1, 0, -1, -2),\\
\lambda^8&=(0, -1\ |\ 1, 0, -1, -2\ |\ 2, 1, 0, -1, -2, -3).
\end{aligned}
\]

In general, let $\Phi=B_n$ with $n=k(k+1)$ for $k\in \bbZ^{>0}$. Choose $I=\Delta\backslash\{\alpha_{i(i+1)}\mid 1\leq i\leq k\}$ and $J=\Delta\backslash\{\alpha_{i^2}\mid 1\leq i\leq k\}$. Define a weight $\lambda$ such that $\lambda_{(i-1)^2+j}=k+1-i$, where $1\leq j\leq 2i-1$ for $1\leq i\leq k$ and $\lambda_{k^2+j}=0$ for $1\leq j\leq k$. Then $\lambda$ is a dominant weight with $\Phi_\lambda=\Phi_J$. The set ${}^IW^J\lambda=W\lambda\cap\Lambda_I^+$ contains weights $\lambda^l$ ($1\leq l\leq 2^k$). If $l=1+\sum_{i=1}^k\eps_i2^{k-i}$ with $\eps_i\in\{0, 1\}$, then $\lambda^l_{i(i-1)+j}=i+1-j-\eps_i$, where $1\leq j\leq 2i$ for $1\leq i\leq k$. It can be checked that every generalized Verma module $M_I(\lambda^l)$ is simple. The category $\caO_\lambda^{\frp_I}$ is semisimple and has $2^k$ blocks. Each block contains a unique simple module (also a generalized Verma module).
\end{example}

%
%
\section{separable weights and separable systems}
%
%

It can be found in many examples that the connectedness between $\lambda$ and $s_{e_i}\lambda$ and $s_{e_i}s_{e_j}\lambda$ is a crucial point for the number of blocks. In this section, we will given criterion for the connectedness of these weights.

\subsection{separable pairs} For a better understanding of the previous examples, we put entries of any $\Phi_I$-regular weight $\lambda$ into a table $T(\lambda)$ as follows. If $|\lambda_i|=a_t\in\caA$ for $q_{s-1}<i\leq q_{s-1}+n_s$ and $1\leq t\leq \overline m$, then put $\lambda_i$ in the $s^{\mbox{\tiny \texttt{th}}}$ column and the $t^{\mbox{\tiny \texttt{th}}}$ row. Thus the table has $m$ or $m-1$ (when $n_m=0$) columns and $\overline m$ or $\overline m-1$ (when $\overline n_{\overline m}=0$) rows.

Now consider the case $\Phi=B_6$ in Example \ref{ccex3}. The tables corresponding to $\lambda^i$ ($1\leq i\leq 4$) are
\begin{table}[htbp]
$T(\lambda^1)=\ $\begin{tabular}{|p{0.3cm}<{\centering}|p{0.3cm}<{\centering}|}
\hline
 & $2$ \\
\hline
$1$ & $\pm1$ \\
\hline
$0$ & $0$ \\
\hline
\end{tabular}\ ,
$T(\lambda^2)=\ $\begin{tabular}{|p{0.3cm}<{\centering}|p{0.3cm}<{\centering}|}
\hline
 & $-2$ \\
\hline
$1$ & $\pm1$ \\
\hline
$0$ & $0$ \\
\hline
\end{tabular}\ ,
$T(\lambda^3)=\ $\begin{tabular}{|p{0.3cm}<{\centering}|p{0.3cm}<{\centering}|}
\hline
 & $2$ \\
\hline
$-1$ & $\pm1$ \\
\hline
$0$ & $0$ \\
\hline
\end{tabular}\ ,
$T(\lambda^4)=\ $\begin{tabular}{|p{0.3cm}<{\centering}|p{0.3cm}<{\centering}|}
\hline
 & $-2$ \\
\hline
$-1$ & $\pm1$ \\
\hline
$0$ & $0$ \\
\hline
\end{tabular}\ .
\bigskip
\caption{Weights of the system $(B_6, A_1\times A_3, A_2\times B_2)$}
\label{sstb1}
\end{table}

Observe that the lower right of each table is always \resizebox{15pt}{15pt}{\begin{tabular}{|p{0.3cm}<{\centering}|}
\hline
 $\pm1$ \\
\hline
 $0$ \\
\hline
\end{tabular}}. Moreover, the first column and the second column never ``interfere'' each other in all these tables. Inspired by this example and many others, we give the following definitions.

\begin{definition}\label{ssdef1}
Let $\lambda$ be a $\Phi_I$-regular weight with $\Phi_{\overline\lambda}=\Phi_J$. Let $S$ (resp. $\overline S$) be a nonempty subset of $\{1, \ldots, m\}$ (resp. $\{1, \ldots, \overline m\}$). We say $\lambda$ is {\it separable} relative to {\it separable pair} $(S, \overline S)$ if the following conditions are satisfied.
\begin{itemize}
\item [(1)] $m>1$ (i.e., $I\neq\Delta$) and $\overline m>1$ (i.e., $J\neq\Delta$);
\item [(2)] $n^\lambda_s(a_t)$ is maximal when $s\in S$ and $t\in \overline S$;
\item [(3)] $n^\lambda_s(a_t)=0$ when $s\not\in S$ and $t\not\in \overline S$;
\item [(4)] For $\Phi=B_n$, $C_n$, $m\in S$ if and only if $\overline m\not\in \overline S$;
\item [(5)] For $\Phi=D_n$, $m\in S$ if and only if $\overline m\in \overline S$.
\end{itemize}
\end{definition}


For any $\lambda\in{}^IW^J\overline\lambda$, let $\caD=\caD(\lambda, \Phi_I, \Phi)$ be the set of all separable pairs of $\lambda$.

\begin{definition}\label{ssdef2}
Suppose $\caD$ is not empty. If $\Phi\neq A_{n-1}$ and $(S, \overline S)\in\caD$ is one of $(\{1, \ldots, m-1\}, \{\overline m\})$, $(\{1, \ldots, m\}, \{\overline m\})$, $(\{m\}, \{1, \ldots, \overline m-1\})$ and $(\{m\}, \{1, \ldots, \overline m\})$, we say $(S, \overline S)$ is \emph{weakly separable}, otherwise $(S, \overline S)$ is \emph{strongly separable}. We say $\lambda$ is {\it strongly separable} if $\caD$ contains a strongly separable pair; otherwise we say $\lambda$ is {\it weakly separable}.
\end{definition}

If $\Phi=A_{n-1}$, every separable pair is strongly separable.

\begin{example}\label{ssex2}
When $\Phi=B_2$, $I=\{e_1-e_2\}$ and $J=\{e_2\}$ in Example \ref{ccex2}, the weight $\lambda^1=(1, 0)$ and $\lambda^2=(0, -1)$ are separable relative to $(\{1\}, \{2\})$ ($n_1^\lambda(0)=1$ is maximal, $n_2^\lambda(1)=0$). Both $\lambda^1$ and $\lambda^2$ are weakly separable.

When $\Phi=B_6$, $I=\Delta\backslash\{\alpha_2, \alpha_6\}$ and $J=\Delta\backslash\{\alpha_1, \alpha_4\}$. The weights $\lambda^l$ ($1\leq l\leq4$) are separable relative to $(\{1, 2\}, \{3\})$ and $(\{2\}, \{2, 3\})$ (see Table \ref{sstb1}). Hence they are strongly separable.

In the general cases, all the weights $\lambda^l$ ($1\leq l\leq 2^k$) are separable relative to $(\{h, h+1, \ldots, k\}, \{k+2-h, \ldots, k+1\})$ for any $1\leq h\leq k$. Hence they are strongly separable if and only if $k>1$.
\end{example}

This example illustrates the following useful result which is an easy consequence of Lemma \ref{bkptlem4} in the last section.

\begin{lemma}\label{sslem1}
Let $\lambda$ be a $\Phi_I$-regular weight with $\Phi_{\overline\lambda}=\Phi_J$. If $\lambda$ is separable relative to $(S, \overline S)$, then $\mu$ is separable relative to $(S, \overline S)$ for any $\mu\in{}^IW^J\overline\lambda$.
\end{lemma}

With the above result, the following definition is predictable.

\begin{definition}\label{ssdef3}
Let $\overline\lambda$ be a dominant weight with $\Phi_{\overline\lambda}=\Phi_J$. We say the system $(\Phi, \Phi_I, \Phi_J)$ is {\it separable} relative to separable pair $(S, \overline S)$ if every $\lambda\in{}^IW^J\overline\lambda$ is separable relative to $(S, \overline S)$. Similarly, we say $(\Phi, \Phi_I, \Phi_J)$ is {\it strongly} (resp. {\it weakly}) {\it separable} if every $\lambda\in{}^IW^J\overline\lambda$ is strongly (resp. weakly) separable.
\end{definition}

If $I, J$ are fixed, it is easy to see that $(S, \overline S)$ is independent of the choices of $\overline\lambda$. The following result also comes from Lemma \ref{bkptlem4}.

\begin{lemma}\label{sslem2}
Let $I, J\subset\Delta$. Then $(\Phi, \Phi_I, \Phi_J)$ is strongly $($resp. weakly$)$ separable if and only if $(\Phi, \Phi_J, \Phi_I)$ is strongly $($resp. weakly$)$ separable.
\end{lemma}

The remainder of this section is devoted to the proofs of three critical propositions.

\begin{prop}\label{ssprop1}
Assume that $\Phi=B_n, C_n$. Let $I, J\subset\Delta$ and $\lambda$ be a $\Phi_I$-regular weight with $\Phi_{\overline\lambda}=\Phi_J$. Suppose that $s_{e_i}\lambda$ is a $\Phi_I$-regular weight with $\lambda\not\lera s_{e_i}\lambda$ for some $1\leq i\leq n$. If $(n_m, \overline n_{\overline m})\neq(0, m-1), (\overline m-1, 0)$, then $(\Phi, \Phi_I, \Phi_J)$ is strongly separable.
\end{prop}

\begin{prop}\label{ssprop2}
Assume that $\Phi=D_n$. Let $I, J\subset\Delta$ and $\lambda$ be a $\Phi_I$-regular weight with $\Phi_{\overline\lambda}=\Phi_J$. Suppose that $s_{e_i}\lambda$ is a $\Phi_I$-regular weight with $\lambda\not\lera s_{e_i}\lambda$ for some $1\leq i\leq n$. If $\overline n_{\overline m}\neq 0$ and $(n_m, \overline n_{\overline m})\neq(\overline m, m)$, then $(\Phi, \Phi_I, \Phi_J)$ is strongly separable.
\end{prop}

\begin{prop}\label{ssprop3}
Assume that $\Phi=B_n, C_n$ or $D_n$. Let $I, J\subset\Delta$ and $I$ be standard. Let $\lambda$ be a $\Phi_I$-regular weight with $\Phi_{\overline\lambda}=\Phi_J$. Suppose that $s_{e_i}s_{e_j}\lambda$ is any $\Phi_I$-regular weight with $1\leq i, j\leq q_{m-1}$ and $\lambda_i\lambda_j\neq 0$. If $\lambda\not\lera s_{e_{i}}s_{e_j}\lambda$, then $(\Phi, \Phi_I, \Phi_J)$ is strongly separable.
\end{prop}

\subsection{The proof of Proposition \ref{ssprop1}} We will prove Proposition \ref{ssprop1} by several lemmas. It suffices to consider the case $\Phi=B_n$, while the proof for $\Phi=C_n$ is similar. The idea of the proof comes from Example \ref{ccex1} and many others. Indeed, Example \ref{ccex1} shows how to ``connect" two weights which are not ``linked". This strategy was already used in the proof of Lemma \ref{cclem2} and Lemma \ref{cclem3}. In these arguments, sequences of linked roots were constructed to connect two weights. Here we need longer sequences which can be built up by induction. We start with some subsets of $\{1, 2, \ldots, m\}$. Those subsets will be used to produce separable pairs in some cases.

Now $\Phi=B_n$. Fix $0<a\in\caA$. For any $\Phi_I$-regular weight $\lambda$ with $\Phi_{\overline\lambda}=\Phi_J$, denote
\[
\begin{aligned}
S_0^\lambda=&\{1\leq s\leq m\mid n^\lambda_s(a)\geq1=n^\lambda_s(0)\};\\
T_0^\lambda=&\{1\leq s\leq m\mid 1\geq n^\lambda_s(a)\geq 0=n^\lambda_s(0)\}.
\end{aligned}
\]
One has $m\in T_0^\lambda$ by Lemma \ref{lrlem1} and $S_0^\lambda\cap T_0^\lambda=\emptyset$. For $r\geq1$, if $T_{r-1}^\lambda$ is defined, let
\[
T_r^\lambda=\{1\leq s\leq m\mid \exists 0<b\in\caA,\ t\in T_{r-1}^\lambda\ \mbox{so that}\ n^\lambda_s(b)\leq1\leq n^\lambda_t(b)\}\cup T_{r-1}^\lambda,
\]
Denote $T^\lambda=\bigcup_{r=0}^\infty T_{r}^\lambda$ and $S^\lambda=\{1, 2, \ldots, m\}\backslash T^\lambda$.

Once $i$ is chosen, there exists $1\leq s_0\leq m$ so that $q_{s_0-1}<i\leq q_{s_0}$. The case $\lambda_i=0$ is trivial. Assume that $\lambda_i\neq0$ and set $a=|\lambda_i|$ at the beginning. One has $n_{s_0}^\lambda(a)=1$ since $s_{e_i}\lambda$ is $\Phi_I$-regular. Note that $n_{s_0}^\lambda(0)$ is either $1$ or $0$, we have to consider the two cases separately.

\begin{lemma}\label{dp1lem1}
If $S_0^\lambda\cap T^\lambda\neq\emptyset$, then $\lambda\lera s_{e_i}\lambda$ for any $q_{s_0-1}<i\leq q_{s_0}$ with $|\lambda_i|=a$ and $n_{s_0}^\lambda(a)=n_{s_0}^\lambda(0)=1$, where $1\leq s_0\leq m$.
\end{lemma}

\begin{proof}
If $s_0=m$, then $e_i\in\Phi_I$ and $\lambda\lera s_{e_{i}}\lambda$ by definition. So we can assume that $s_0<m$. Evidently, $s_0\in S_0^\lambda$. If $S_0^\lambda\cap T^\lambda\neq\emptyset$, choose the smallest positive integer $r=r_\lambda$ so that $S_0^\lambda\cap T_{r}^\lambda\neq\emptyset$. We use induction on $r$ to prove the lemma. Assume that $s\in S_0^\lambda\cap T_r^\lambda$. Then $n^\lambda_s(a)\geq 1=n^\lambda_s(0)$ and $s<m$ (since $m\in T_0^\lambda$ and $S_0^\lambda\cap T_0^\lambda=\emptyset$). There exist $q_{s-1}<j\leq q_s$ such that $\lambda_j=0$. On the other hand, $s\in T_r^\lambda\backslash T_{r-1}^\lambda$ implies the existence of $0<b\in \caA$ and $s\neq t\in T_{r-1}^\lambda$ with $n^\lambda_s(b)\leq1\leq n^\lambda_t(b)$. If $n^\lambda_t(a)\geq1$, then $n^\lambda_t(0)=0$ (otherwise $t\in S_0^\lambda\cap T_{r-1}^\lambda$, contradicts the minimality of $r$). With Lemma \ref{cclem2} (relative to $s_0, t$), we get $\lambda\lera s_{e_{i}}\lambda$. So it suffices to consider the case $n^\lambda_t(a)=0$. In this case, $s_0\neq t$ and $b\neq a$.

Now we have $n^\lambda_s(b)\leq1\leq n^\lambda_t(b)$ and $n^\lambda_t(a)=0<n^\lambda_s(a)$ with $s<m$. Lemma \ref{cclem4} gives $\beta\in\{e_k\pm e_l\}$ such that $\mu=s_\beta\lambda$ is $\Phi_I$-regular, where $q_{s-1}<k\leq q_s$ and $q_{t-1}<l\leq q_t$ with $|\lambda_k|=a$ and $|\lambda_l|=b$. Moreover, $\beta$ is a linked root from $\lambda$, that is, $\lambda\lera s_\beta\lambda$. We obtain $|\mu_l|=a$ and $n^\mu_t(a)=1+n^\lambda_t(a)=1$, which means $s_{e_l}\mu$ is $\Phi_I$-regular. In addition, $n^\mu_{s_0}(0)=n^\lambda_{s_0}(0)=1$ and $n^\mu_{s_0}(a)=0$ ($s_0=s$) or $1$ $(s_0\neq s)$. Since the above reasoning does not depend on the sign of $\lambda_i$ (whether or not $s_0=s$, keeping in mind that $s_0\neq t$), we also have $s_{e_i}\lambda\lera s_\beta s_{e_i}\lambda$.

If $r_\lambda=1$, then $t\in T_0^\lambda$. So $n^\mu_t(0)=n^\lambda_t(0)=0$. Recall that $n^\mu_t(a)=1$ and $n^\mu_{s_0}(a)\leq 1=n^\mu_{s_0}(0)$. With $|\mu_l|=a$ for $q_{t-1}<l\leq q_t$, we get $\mu\lera s_{e_l}\mu$ by Lemma \ref{cclem2} (relative to $t, s_0$). If $s_0=s$, then $i=k$ and
\[
\lambda\lera s_\beta\lambda=\mu\lera s_{e_l}\mu=s_{e_l}s_\beta\lambda=s_\beta s_{e_k}\lambda=s_\beta s_{e_i}\lambda\lera s_{e_i}\lambda.
\]
If $s_0\neq s$, then $i\neq k$. Since $t\neq s_0$, we also get $l\neq i$. With $n^\mu_{s_0}(a)=n^\mu_t(a)=1$ and $n^\mu_{s_0}(0)=1>0=n^\mu_t(0)$, we can similarly get $\mu\lera s_{e_i}\mu$ by Lemma \ref{cclem2} (relative to $s_0, t$). Therefore
\[
\lambda\lera s_\beta\lambda=\mu\lera s_{e_i}\mu=s_{e_i}s_\beta\lambda=s_\beta s_{e_i}\lambda\lera s_{e_i}\lambda.
\]

If $r=r_\lambda>1$, then $t\not\in T_{r-2}^\lambda$ (otherwise $s\in S_0^\lambda\cap T_{r-1}^\lambda$, a contradiction). Recalling that $n^\lambda_t(a)=0$, $t\not\in T_{r-2}^\lambda$ yields $t\not\in T_0^\lambda$ and thus $n^\lambda_t(0)=1$. Thus $n^\mu_t(0)=1$. Combined with $n^\mu_t(a)=1$ ($|\mu_l|=a$), one has $t\in S_0^\mu$. If $t\in T_{r-1}^\mu$, we get $r_\mu\leq r-1<r_\lambda$. The induction hypothesis (with $(\lambda, s_0, i)$ replaced by $(\mu, t, l)$) implies $\mu\lera s_{e_l}\mu$. If $s_0=s$, then $i=k$ and one obtains $\lambda\lera\mu\lera s_{e_l}\mu=s_\beta s_{e_k}\lambda\lera s_{e_i}\lambda$. If $s_0\neq s$, then $i\neq k$, $|\mu_i|=|\lambda_i|=a$ and $s_0\in S_0^\mu$. We can also use the induction hypothesis (with $(\lambda, s_0, i)$ replaced by $(\mu, s_0, i)$) to obtain $\mu\lera s_{e_i}\mu$. Hence $\lambda\lera \mu\lera s_{e_i}\mu=s_\beta s_{e_i}\lambda\lera s_{e_i}\lambda$. It remains to consider the case $t\not\in T_{r-1}^\mu$. With $n^\mu_s(0)=n^\lambda_s(0)=1$, obviously $s\not\in T_0^\mu$. Notice that $n^\mu_{s'}(a')=n^\lambda_{s'}(a')$ for any $a'\in\caA$ and $s'\neq s, t$. If $s\not\in T_{r-1}^\mu$, we obtain $T_{p}^\mu=T_{p}^\lambda$ for $0\leq p\leq r-1$ by definition. If $s\in T_h^\mu\backslash T_{h-1}^\mu$ for $1\leq h<r$, there is $t'\in T_{h-1}^\mu=T_{h-1}^\lambda$ and $0<b'\in\caA$ such that $n^\mu_s(b')\leq 1\leq n^\mu_{t'}(b')$. Obviously $t'\neq s, t$ and $n^\mu_{t'}(b')=n^\lambda_{t'}(b')$. Note that $n^\mu_s(b')-n^\lambda_s(b')=-1$ ($b'=a$), $1$ ($b'=b$) or $0$ ($b'\neq a, b$). If $b'\neq a$, then
\[
n^\lambda_s(b')\leq n^\mu_s(b')\leq 1\leq n^\mu_{t'}(b')=n^\lambda_{t'}(b'),
\]
that is, $s\in T_{h}^\lambda\cap S_0^\lambda$, which contradicts the minimality of $r$. Now assume that $b'=a$. Then $n^\lambda_{t'}(a)=n^\mu_{t'}(b')\geq1$. If $n^\lambda_{t'}(0)=0$, we must have $\lambda\lera s_{e_i}\lambda$ by Lemma \ref{cclem2} (relative to $s_0, t'$). If $n^\lambda_{t'}(0)=1$, then $t'\in S_0^\lambda\cap T_{h-1}^\lambda$. This also contradicts the minimality of $r$.
\end{proof}

With the above lemma, we can prove a weak version of Proposition \ref{ssprop1}.

\begin{lemma}\label{dp1lem2}
With the setting in Proposition \ref{ssprop1}, we further assume that $|\lambda_i|=a$ and $n_{s_0}^\lambda(0)=1$, where $q_{s_0-1}<i\leq q_{s_0}$ and $1\leq s_0\leq m$. If $\lambda\not\lera s_{e_i}\lambda$, then either $(n_m, \overline n_{\overline m})=(0, m-1)$ or $(\Phi, \Phi_I, \Phi_J)$ is strongly separable.
\end{lemma}

\begin{proof}
Since $s_{e_i}\lambda$ is $\Phi_I$-regular, one must have $n_{s_0}^\lambda(a)=1$. If $\lambda\not\lera s_{e_i}\lambda$, we can apply Lemma \ref{dp1lem1} and get $S_0^\lambda\cap T^\lambda=\emptyset$. In this case, $s_0\in S_0^\lambda\subset S^\lambda$ and $m\in T_0^\lambda\subset T^\lambda$. Thus $S^\lambda$ and $T^\lambda$ are not empty. Let
\[
\overline\caA=\{b\in\caA\mid n^\lambda_{t}(b)\geq1\ \mbox{for}\ t\in T^\lambda\}.
\]
Fix $0<b\in \overline\caA$. There exists $t\in T^\lambda$ with $n^\lambda_{t}(b)\geq 1$. We can choose $r\geq0$ so that $t\in T_{r}^\lambda$. If $n^\lambda_{s}(b)\leq 1$ for some $s\in S^\lambda$, then $n^\lambda_{s}(b)\leq 1\leq n^\lambda_{t}(b)$ yields $s\in T_{r+1}^\lambda\subset T^\lambda$. We get $s\in S^\lambda\cap T^\lambda$, a contradiction. This forces $n^\lambda_{s}(b)=2$ (hence maximal by Lemma \ref{lrlem2}) for all $0<b\in \overline\caA$, $s\in S^\lambda$. If $n^\lambda_{s}(0)=0$ for some $s\in S^\lambda$, then $n^\lambda_{s}(a)=0$. Otherwise Lemma \ref{cclem2} (relative to $s_0, s$) yields $\lambda\lera s_{e_i}\lambda$, a contradiction. However, $n^\lambda_{s}(0)=n^\lambda_{s}(a)=0$ means $s\in T_0^\lambda\cap S^\lambda=\emptyset$, another contradiction. So $n^\lambda_{s}(0)=1$ (also maximal) for all $s\in S^\lambda$. Denote $\overline S=\{1\leq r\leq \overline m\mid a_r\in \overline\caA\}$. Then $(\Phi, \Phi_I, \Phi_J)$ is separable relative to $(S^\lambda, \overline S\cup\{\overline m\})$ in view of Lemma \ref{sslem1}.

If $\overline S$ is empty, then $T^\lambda=\{m\}$ and $S^\lambda=\{1, \ldots, m-1\}$. This forces $n_m=0$ and $n^\lambda_{s}(0)=1$ for $1\leq s\leq m-1$. Thus $\overline n_{\overline m}=m-1$. If $\overline S=\{\overline m\}$ and $(\Phi, \Phi_I, \Phi_J)$ is not strongly separable, then $S^\lambda=\{1, \ldots, m-1\}$. It follows that $n_m^\lambda(0)\geq1$. In view of Lemma \ref{lrlem1}, we arrive at a contradiction that $\lambda$ is $\Phi_I$-singular. If $\overline S\neq\{\overline m\}$ is not empty, it can be verified that $(S^\lambda, \overline S\cup\{\overline m\})$ is a strongly separable pair.
\end{proof}

Proposition \ref{ssprop1} is proved if the following result holds.
\begin{lemma}\label{dp1lem3}
With the setting in Proposition \ref{ssprop1}, we further assume that $|\lambda_i|=a$ and $n_{s_0}^\lambda(0)=0$, where $q_{s_0-1}<i\leq q_{s_0}$ and $1\leq s_0\leq m$. If $\lambda\not\lera s_{e_i}\lambda$, then either $(n_m, \overline n_{\overline m})=(\overline m-1, 0)$ or $(\Phi, \Phi_I, \Phi_J)$ is strongly separable.
\end{lemma}

\begin{proof}
Note that $\lambda\not\lera s_{e_i}\lambda$ is equivalent to $s_{e_i}\lambda\not\lera\lambda= s_{e_i}(s_{e_i}\lambda)$. We can assume that $\lambda_i=a$. Since $\lambda$ is $\Phi_I$-regular, there is $w\in W_I$ so that $w\lambda\in\Lambda_I^+$. Therefore, $\lambda\not\lera s_{e_i}\lambda$ if and only if $w\lambda\not\lera s_{we_i}w\lambda$. In view of Lemma \ref{lrlem2}, it suffices to consider the case $\lambda\in\Lambda_I^+$.

With $n^\lambda_{s_0}(0)=0$, one can transfer the problem associated with $(\Phi, \Phi_I, \Phi_J)$ to a dual problem associated with $(\Phi, \Phi_J, \Phi_I)$. Then apply Lemma \ref{dp1lem2}. In fact, $n^\lambda_{s_0}(a)=1$ (since $s_{e_i}\lambda$ is $\Phi_I$-regular) and $n^\lambda_{s_0}(0)=0$ yield $n^\lambda_m(a)=1$ by Proposition \ref{lrprop2}: otherwise $e_i$ is a linked root from $\lambda$ to $s_{e_i}\lambda$. There exists $q_{m-1}<j\leq n$ with $\lambda_j=a$, in view of $\lambda\in\Lambda_I^+$. Recall that $\lambda=w\overline\lambda$ for some $w\in {}^IW^J$ and dominant weight $\overline\lambda$. Let $\overline\lambda'$ be a dominant root with $\Phi_{\overline\lambda'}=\Phi_I$ and $\lambda'=w^{-1}\overline\lambda'$. Note that
\begin{equation}\label{dp1l3eq1}
\langle\overline\lambda, w^{-1}e_i\rangle=\langle w\overline\lambda, e_i\rangle=\lambda_i=a=\lambda_j=\langle\lambda, e_j\rangle=\langle\overline\lambda, w^{-1}e_j\rangle>0.
\end{equation}
So we can assume that $w^{-1}e_i=e_{i'}$ and $w^{-1}e_j=e_{j'}$ (obviously $i'\neq j'$). The above equation implies that $\overline\lambda_{i'}=\overline\lambda_{j'}=a$. Thus $e_{i'}-e_{j'}\in\Phi_J$. There exists $1\leq s<\overline m$ so that $\overline q_{s-1}<i', j'\leq \overline q_{s}$. Since $s_{e_i}\in W$ and $s_{e_i}\lambda$ is $\Phi_I$-regular, there exist $w_1\in W_I$ and $x\in {}^IW^J$ such that $x\overline\lambda=w_1s_{e_i}\lambda=w_1s_{e_i}w\overline\lambda$. We can find $w_2\in W_J$ such that $xw_2=w_1s_{e_i}w$. Note that $w_1\overline\lambda'=\overline\lambda'$ for $w_1\in\Phi_I$. So
\begin{equation}\label{dp1l3eq2}
s_{e_{i'}}\lambda'=w^{-1}s_{e_i}w\lambda'=w^{-1}s_{e_i}\overline\lambda'
=w_2^{-1}x^{-1}w_1\overline\lambda'=w_2^{-1}x^{-1}\overline\lambda'
\end{equation}
is $\Phi_J$-regular. With $\lambda\not\lera s_{e_i}\lambda$, we get $w\overline\lambda\not\lera x\overline\lambda$. Lemma \ref{jflem4} implies $\lambda'=w^{-1}\overline\lambda'\not\lera x^{-1}\overline\lambda'$. It follows from (\ref{dp1l3eq2}) that $\lambda'\not\lera s_{e_{i'}}\lambda'$. Denote $a'=|\lambda'_{i'}|$. It follows from $e_j\in\Phi_I$ that $0=\langle\overline\lambda', e_j\rangle=\langle w^{-1}\overline\lambda', w^{-1}e_j\rangle=\langle\lambda', e_{j'}\rangle=\lambda'_{{j'}}$. We get $n^{\lambda'}_s(a')=n^{\lambda'}_s(0)=1$ (relative to $J$). With Lemma \ref{dp1lem2} in hand, we get either $(\overline n_{\overline m}, n_m)=(0, \overline m-1)$ or $(\Phi, \Phi_J, \Phi_I)$ is strongly separable. By Lemma \ref{sslem2}, one obtains either $(n_m, \overline n_{\overline m})=(\overline m-1, 0)$ or $(\Phi, \Phi_I, \Phi_J)$ is strongly separable.
\end{proof}

\subsection{The proof of Proposition \ref{ssprop2}} Similarly, the proof of this proposition divides into several lemmas involving particular subsets of $\{1, \ldots, m\}$.

Now $\Phi=D_n$. Fix $0<a\in\caA$. For any $\Phi_I$-regular weight $\lambda$ with $\Phi_{\overline\lambda}=\Phi_J$, denote
\[
\begin{aligned}
S_0^\lambda=&\{1\leq s\leq m\mid n^\lambda_s(a)\geq1=n^\lambda_s(0)\};\\
T_0^\lambda=&\{1\leq s\leq m\mid 1\geq n^\lambda_s(a)\geq 0=n^\lambda_s(0)\}.
\end{aligned}
\]
Then we always have $S_0^\lambda\cap T_0^\lambda=\emptyset$. For $r\geq1$, if $T_{r-1}^\lambda$ is defined, let
\[
T_r^\lambda=\{1\leq s\leq m\mid \exists 0<b\in\caA,\ t\in T_{r-1}^\lambda\ \mbox{so that}\ n^\lambda_s(b)+\delta_{s, m}\leq1\leq n^\lambda_t(b)\}\cup T_{r-1}^\lambda,
\]
where $\delta_{s, m}=0$ for $s\neq m$ and $\delta_{m, m}=1$. Denote $T^\lambda=\bigcup_{r=0}^\infty T_{r}^\lambda$ and $S^\lambda=\{1, 2, \ldots, m\}\backslash T^\lambda$.

\begin{lemma}\label{dp2lem1}
If $I$ is standard and $S_0^\lambda\cap T^\lambda\neq\emptyset$, then $\lambda\lera s_{e_i}\lambda$ for any $q_{s_0-1}<i\leq q_{s_0}$ with $|\lambda_i|=a$ and $n_{s_0}^\lambda(a)=n_{s_0}^\lambda(0)=1$, where $1\leq s_0\leq m$.
\end{lemma}

\begin{proof}
Evidently, $s_0\in S_0^\lambda$. With $n_{s_0}^\lambda(a)=n_{s_0}^\lambda(0)=1$, there is $q_{s_0-1}<u\leq q_{s_0}$ so that $\lambda_u=0$ ($u\neq i$). If $s_0=m$, then $e_i\pm e_u\in\Phi_I$ and $\lambda\lera s_{e_i-e_u}s_{e_i+e_u}\lambda=s_{e_i}\lambda$. So we can assume that $s_0<m$. With $e_i-e_u\in\Phi_I$, if $e_i+e_u$ is a linked root from $\lambda$, then $\lambda\lera s_{e_i+e_u}\lambda\lera s_{e_i-e_u}s_{e_i+e_u}\lambda=s_{e_i}\lambda$, a contradiction. By Proposition \ref{lrprop2}, ${e_i+e_u}$ is not a linked root only when $n^\lambda_m(a)=n^\lambda_m(0)=1$. So we can assume that $m\in S_0^\lambda$. If $S_0^\lambda\cap T^\lambda\neq\emptyset$, choose the smallest positive integer $r_\lambda$ so that $S_0^\lambda\cap T_{r_\lambda}^\lambda\neq\emptyset$. As in Lemma \ref{dp1lem1}, we use induction on $r=r_\lambda$. Assume that $s\in S_0^\lambda\cap T_r^\lambda$. There exist $q_{s-1}<j\leq q_s$ such that $\lambda_j=0$. On the other hand, $s\in T_r^\lambda\backslash T_{r-1}^\lambda$ implies the existence of $0<b\in \caA$ and $s\neq t\in T_{r-1}^\lambda$ with $n^\lambda_s(b)+\delta_{s, m}\leq1\leq n^\lambda_t(b)$. If $n^\lambda_t(a)\geq1$, then $n^\lambda_t(0)=0$ since $S_0^\lambda\cap T_{r-1}^\lambda=\emptyset$. With Lemma \ref{cclem2} (relative to $s_0, t$), we get $\lambda\lera s_{e_{i}}\lambda$. So it suffices to consider the case $n^\lambda_t(a)=0$ (thus $t<m$). In this case, $n^\lambda_t(b)>0$ implies $b\neq a$.

Now we have $n^\lambda_s(b)+\delta_{s, m}\leq1\leq n^\lambda_t(b)$ and $n^\lambda_t(a)=0<n^\lambda_s(a)$. Lemma \ref{cclem4} provides $\beta\in\{e_k\pm e_l\}$ such that $\mu=s_\beta\lambda$ is $\Phi_I$-regular, where $q_{s-1}<k\leq q_s$ and $q_{t-1}<l\leq q_t$ with $|\lambda_k|=a$ and $|\lambda_l|=b$. Moreover, $\beta$ is a linked root from $\lambda$, that is, $\lambda\lera s_\beta\lambda$. We obtain $|\mu_l|=a$ and $n^\mu_t(a)=1$, which implies $s_{e_l}\mu$ is $\Phi_I$-regular. In addition, $n^\mu_{s_0}(0)=1$ and $n^\mu_{s_0}(a)=0$ ($s_0=s$) or $1$ $(s_0\neq s)$. Similar reasoning shows $s_{e_i}\lambda\lera s_\beta s_{e_i}\lambda$.

If $r_\lambda=1$, then $t\in T_0^\lambda$ yields $n^\mu_t(0)=n^\lambda_t(0)=0$. Recall that $n^\mu_t(a)=1$ and $n^\mu_{s_0}(a)\leq 1=n^\mu_{s_0}(0)$. With $|\mu_l|=a$, we get $\mu\lera s_{e_l}\mu$ by Lemma \ref{cclem2} (relative to $t, s_0$). If $s_0=s$, then $i=k$ and $\lambda\lera \mu\lera s_{e_l}\mu=s_{e_l}s_\beta\lambda=s_\beta s_{e_k}\lambda=s_\beta s_{e_i}\lambda\lera s_{e_i}\lambda$. If $s_0\neq s$, then $i\neq k$. With $n^\mu_{s_0}(a)=n^\mu_t(a)=1$ and $n^\mu_{s_0}(0)=1>0=n^\mu_t(0)$, we can get $\mu\lera s_{e_i}\mu$ by Lemma \ref{cclem2}. Therefore $\lambda\lera \mu\lera s_{e_i}\mu=s_{e_i}s_\beta\lambda=s_\beta s_{e_i}\lambda\lera s_{e_i}\lambda$.

If $r=r_\lambda>1$, then $t\not\in T_{r-2}^\lambda$. It implies $n^\lambda_t(0)=1$ since $n^\lambda_t(a)=0$. Thus $n^\mu_t(0)=1$. It follows from $n^\mu_t(a)=1$ ($|\mu_l|=a$) that $t\in S_0^\mu$. If $t\in T_{r-1}^\mu\cap S_0^\mu$, it yields $r_\mu\leq r-1<r_\lambda$. The induction hypothesis (with $(\lambda, s_0, i)$ replaced by $(\mu, t, l)$) implies $\mu\lera s_{e_l}\mu$. If $s_0=s$, then $k=i$ and one obtains $\lambda\lera\mu\lera s_{e_l}\mu=s_\beta s_{e_k}\lambda=s_\beta s_{e_i}\lambda\lera s_{e_i}\lambda$. If $s_0\neq s$, then $k\neq i$, $|\mu_i|=|\lambda_i|=a$ and $s_0\in S_0^\mu$. We can also use the induction hypothesis (with $(\lambda, s_0, i)$ replaced by $(\mu, s_0, i)$) to obtain $\mu\lera s_{e_i}\mu$. Hence $\lambda\lera \mu\lera s_{e_i}\mu=s_\beta s_{e_i}\lambda\lera s_{e_i}\lambda$. It remains to consider the case $t\not\in T_{r-1}^\mu$. Notice that $n^\mu_{s'}(a')=n^\lambda_{s'}(a')$ for any $a'\in\caA$ and $s'\neq s, t$. If $s\not\in T_{r-1}^\mu$, one obtains $T_{p}^\mu=T_{p}^\lambda$ for $0\leq p\leq r-1$ by definition. If $s\in T_h^\mu\backslash T_{h-1}^\mu$ for $1\leq h<r$ ($s\not\in T_0^\lambda$ since $n_s^\mu(0)=n_s^\lambda(0)=1$), there exist $t'\in T_{h-1}^\mu=T_{h-1}^\lambda$ and $0<b'\in\caA$ such that $n^\mu_s(b')+\delta_{s, m}\leq 1\leq n^\mu_{t'}(b')$. Evidently $t'\neq s, t$ and $n^\mu_{t'}(b')=n^\lambda_{t'}(b')$. Note that $n^\mu_s(b')-n^\lambda_s(b')=-1$ ($b'=a$), $1$ ($b'=b$) or $0$ ($b'\neq a, b$). If $b'\neq a$, then
\[
n^\lambda_s(b')+\delta_{s, m}\leq n^\mu_s(b')+\delta_{s, m}\leq 1\leq n^\mu_{t'}(b')=n^\lambda_{t'}(b'),
\]
that is, $s\in T_h^\lambda$, which contradicts the minimality of $r$. Now suppose $b'=a$, then $n^\lambda_{t'}(a)=n^\mu_{t'}(a)\geq1$. If $n^\lambda_{t'}(0)=0$, we must have $\lambda\lera s_{e_i}\lambda$ by Lemma \ref{cclem2} (relative to $s_0, t'$). If $n^\lambda_{t'}(0)=1$, then $t'\in S_0^\lambda\cap T_{h-1}^\lambda$. This also contradicts the minimality of $r$.
\end{proof}

\begin{lemma}\label{dp2lem2}
With the setting in Proposition \ref{ssprop2}, we further assume that $I$ is standard, $|\lambda_i|=a$ and $n_{s_0}^\lambda(0)=1$, where $q_{s_0-1}<i\leq q_{s_0}$ and $1\leq s_0\leq m$. If $\lambda\not\lera s_{e_i}\lambda$, then either $(n_m, \overline n_{\overline m})=(\overline m, m)$ or $(\Phi, \Phi_I, \Phi_J)$ is strongly separable.
\end{lemma}

\begin{proof}
If $\lambda\not\lera s_{e_i}\lambda$, Lemma \ref{dp2lem1} yields $S_0^\lambda\cap T^\lambda=\emptyset$. The proof of Lemma \ref{dp2lem1} also shows that $s_0, m\in S_0^\lambda\subset S^\lambda$ and $s_0\neq m$. Thus $S^\lambda$ is not empty. Let
\[
\overline\caA=\{b\in\caA\mid n^\lambda_{t}(b)\geq1\ \mbox{for}\ t\in T^\lambda\}.
\]
Fix $0<b\in \overline\caA$ and $t\in T^\lambda$ with $n^\lambda_{t}(b)\geq 1$. Choose $r\geq0$ so that $t\in T_{r}^\lambda$. If $n^\lambda_{s}(b)\leq 1-\delta_{s, m}$ for some $s\in S^\lambda$, then $n^\lambda_{s}(b)+\delta_{s, m}\leq 1\leq n^\lambda_{t}(b)$ yields $s\in T_{r+1}^\lambda\subset T^\lambda$. We get $s\in S^\lambda\cap T^\lambda$, a contradiction. This forces $n^\lambda_{s}(b)=2-\delta_{s, m}$ (hence maximal by Lemma \ref{lrlem2}) for all $0<b\in \overline\caA$, $s\in S^\lambda$. We also have $n^\lambda_s(0)=1$ for all $s\in S^\lambda$ by an argument similar to that of the case $\Phi=B_n$ in Lemma \ref{dp1lem2}. Thus $(\Phi, \Phi_I, \Phi_J)$ is separable relative to $(S^\lambda, \overline S\cup\{\overline m\})$, where $\overline S=\{1\leq r\leq \overline m\mid a_r\in \overline\caA\}$. If $0<|\overline S|$, then $T^\lambda\neq\emptyset$. it can be verified that $\lambda$ is strongly separable relative to $(S^\lambda, \overline S\cup\{\overline m\})$, so is $(\Phi, \Phi_I, \Phi_J)$ in view of Lemma \ref{sslem1}.

If $\overline S$ is empty, then $T^\lambda=\emptyset$ and $S^\lambda=\{1, \ldots, m\}$. This forces $\overline n_{\overline m}=|S^\lambda|=m$. With $n^\lambda_{s_0}(a)=n^\lambda_m(a)=1$ and $n^\lambda_{s_0}(0)=n^\lambda_m(0)=1$, one can also transfer the problem to the dual case. The essential ideas can be found in Lemma \ref{dp1lem3}, but modified slightly. It suffices to consider the case $\lambda\in\Lambda_I^+$ and $\lambda_i=a$. So $\lambda_n=0$ and there exists $q_{m-1}<j<n$ with $\lambda_j=a$. Assume that $\lambda=w\overline\lambda$ for $w\in {}^IW^J$. Let $\overline\lambda'$ be a dominant weight with $\Phi_{\overline\lambda'}=\Phi_I$ and $\lambda'=w^{-1}\overline\lambda'\in\Lambda_J^+$. In view of (\ref{dp1l3eq1}), we can set $w^{-1}e_i=e_{i'}$, $w^{-1}e_j=e_{j'}$ and obtain $e_{i'}-e_{j'}\in\Phi_J$. Moreover, there exists $1\leq s<\overline m$ so that $\overline q_{s-1}<i', j'\leq \overline q_{s}$ ($a_s=a$). Since $e_i\pm e_n\in\Phi$ and $s_{e_i}\lambda=s_{e_i-e_n}s_{e_i+e_n}\lambda\in W\lambda$ is $\Phi_I$-regular, there exist $w_1\in W_I$ and $x\in {}^IW^J$ such that $x\overline\lambda=w_1s_{e_i}\lambda=w_1s_{e_i-e_n}s_{e_i+e_n}w\overline\lambda$. We can find  $w_2\in W_J$ such that $xw_2=w_1s_{e_i+e_n}s_{e_i-e_n}w$. Denote $a'=|\lambda'_{i'}|$. Then $e_j\pm e_n\in\Phi_I$ yields $\lambda'_{{j'}}=\lambda'_{n'}=0$, where $e_{n'}=\pm w^{-1}e_n$. So
\[
s_{e_{i'}}\lambda'=s_{e_{i'}}s_{e_{n'}}\lambda'=w^{-1}s_{e_i+e_n}s_{e_i-e_n}w\lambda'
=w_2^{-1}x^{-1}w_1\overline\lambda'=w_2^{-1}x^{-1}\overline\lambda'
\]
is $\Phi_J$-regular. Lemma \ref{jflem4} implies $\lambda'\not\lera s_{e_{i'}}\lambda'$. By the previous argument, $(\Phi, \Phi_J, \Phi_I)$ is strongly separable when $n_m\neq \overline m$ ($J$ is standard since $\overline n_{\overline m}\geq2$). This combines with Lemma \ref{sslem2} completes the proof.
\end{proof}

\begin{lemma}\label{dp2lem3}
If $I$ is standard and $S_0^\lambda\cap T^\lambda\neq\emptyset$, then $\lambda\lera s_{e_i}\lambda$ for $q_{s_0-1}<i\leq q_{s_0}$ with $|\lambda_i|=a$, $n_{s_0}^\lambda(a)=1$ and $n_{s_0}^\lambda(0)=0$, where $1\leq s_0\leq m$.
\end{lemma}

\begin{proof}
Evidently, $s_0\in T_0^\lambda$. If $n^\lambda_m(0)=1$, then Lemma \ref{cclem2} (relative to $s_0, m$) yields $\lambda\lera s_{e_i}\lambda$. So we can assume that $n^\lambda_m(0)=0$. It follows that $m\in T_0^\lambda$ (it is possible that $s_0=m$). If $S_0^\lambda\cap T^\lambda\neq\emptyset$, choose the smallest positive integer $r_\lambda$ so that $S_0^\lambda\cap T_{r_\lambda}^\lambda\neq\emptyset$. Denote $r=r_\lambda$. Assume that $s\in S_0^\lambda\cap T_r^\lambda$. There exist $q_{s-1}<j\leq q_s$ such that $\lambda_j=0$. On the other hand, $s\in T_r^\lambda\backslash T_{r-1}^\lambda$ yields $0<b\in \caA$ and $s\neq t\in T_{r-1}^\lambda$ with $n^\lambda_s(b)+\delta_{s, m}\leq1\leq n^\lambda_t(b)$. With $s\in S_0^\lambda$, we get $n^\lambda_s(0)=1$. If $n^\lambda_s(a)\leq 1$, then Lemma \ref{cclem2} (relative to $s_0, s$) implies $\lambda\lera s_{e_i}\lambda$. If $n^\lambda_t(0)=1$ and $n^\lambda_t(a)\leq 1$, then we also get $\lambda\lera s_{e_i}\lambda$ by Lemma \ref{cclem2} (relative to $s_0, t$). If $n^\lambda_t(0)=1$ and $n^\lambda_t(a)=2$, then we arrive at a contradiction $t\in S_0^\lambda\cap T_{r-1}^\lambda=\emptyset$. So we only need to consider the case $n^\lambda_s(a)=2$ and $n^\lambda_t(0)=0$. Therefore $s<m$ (so $\delta_{s, m}=0$) and $b\neq a$ (since $n^\lambda_s(b)\leq 1$).

With $n^\lambda_s(b)+\delta_{s, m}\leq1\leq n^\lambda_t(b)$ and $n^\lambda_t(0)=0<1=n^\lambda_s(0)$. Lemma \ref{cclem4} provides $\beta\in\{e_j\pm e_l\}$ such that $\mu=s_\beta\lambda$ is $\Phi_I$-regular, where $q_{t-1}<l\leq q_t$ with $|\lambda_l|=b$. Moreover, $\beta$ is a linked root from $\lambda$, that is, $\lambda\lera s_\beta\lambda$. We also have $\mu_l=0$ and $n^\mu_t(0)=1$. Since $b\neq a$, similar reasoning shows that $s_{e_i}\lambda\lera s_\beta s_{e_i}\lambda$. Moreover, $n^\mu_{s_0}(a)=n^\lambda_{s_0}(a)=1$ (since $b\neq a$) and $n^\mu_{s_0}(0)=0$ ($s_0\neq t$) or $1$ $(s_0=t)$.

For $r=1$, we get $t\in T_0^\lambda$ and $n^\mu_t(a)=n^\lambda_t(a)\leq 1$ (since $b\neq a$). With $n^\mu_t(0)=1$ and $n^\mu_{s_0}(a)=1$, if $s_0\neq t$, we get $n^\mu_{s_0}(0)=0$. Then $\mu\lera s_{e_i}\mu$ by Lemma \ref{cclem2} (relative to $s_0, t$). If $s_0=t<m$, then $n^\mu_{s_0}(0)=1$ and $n^\mu_m(0)=n^\lambda_m(0)=0$. With $q_{s_0-1}<i, l\leq q_{s_0}$ and $|\mu_i|=|\lambda_i|=a>0=\mu_l$, $\mu\lera s_{e_i+e_l}s_{e_i-e_l}\mu=s_{e_i}\mu$ by Proposition \ref{lrprop2}. If $s_0=t=m$, then $\mu\lera s_{e_i+e_l}s_{e_i-e_l}\mu=s_{e_i}\mu$ in view of $e_i\pm e_l\in\Phi_I$. In either case,
\[
\lambda\lera s_\beta\lambda=\mu\lera s_{e_i}\mu=s_{e_i}s_\beta\lambda=s_\beta s_{e_i}\lambda\lera s_{e_i}\lambda.
\]

Now assume that $r>1$. Then $t\in T_{r-1}^\lambda\backslash T_{r-2}^\lambda$ implies $t\not\in T_{r-2}^\lambda$. Recalling that $n^\lambda_t(0)=0<1=n^\mu_t(0)$, $t\not\in T_0^\lambda$ indicates $n^\lambda_t(a)>1$. It follows that $t\neq s_0$ and $n^\mu_t(a)=n^\lambda_t(a)=2$ (since $b\neq a$). With $n^\mu_t(0)=1$, one obtains $t\in S_0^\mu$. If $t\in T_{r-1}^\mu$, then $r_\mu<r=r_\lambda$. We can use the induction hypothesis (with $(\lambda, s_0, i)$ replaced by $(\mu, s_0, i)$) to get $\mu\lera s_{e_i}\mu$. Therefore $\lambda\lera\mu\lera s_{e_i}\mu=s_\beta s_{e_i}\lambda\lera s_{e_i}\lambda$. We still need to consider the case $t\not\in T_{r-1}^\mu$. Notice that $n^\mu_{s'}(a')=n^\lambda_{s'}(a')$ for any $a'\in\caA$ and $s'\neq s, t$. If $s\not\in T_{r-1}^\mu$, we obtain $T_{p}^\mu=T_{p}^\lambda$ for $0\leq p\leq r-1$ and thus $t\not\in T_{r-1}^\mu=T_{r-1}^\lambda$, a contradiction. If $s\in T_h^\mu\backslash T_{h-1}^\mu$ for $1\leq h\leq r-1$ ($s\not\in T_0^\mu$ since $n_s^\mu(a)=n_s^\lambda(a)=2$), there exist $t'\in T_{h-1}^\mu=T_{h-1}^\lambda$ and $0<b'\in\caA$ such that $n^\mu_s(b')\leq 1\leq n^\mu_{t'}(b')$. Since $t'\neq s, t$, we have $n^\lambda_{t'}(b')=n^\mu_{t'}(b')$. Note that $n^\mu_s(b')-n^\lambda_s(b')=1$ ($b'=b$) or $0$ ($b'\neq b$). We have
\[
n^\lambda_s(b')+\delta_{s, m}=n^\lambda_s(b')\leq n^\mu_s(b')\leq 1\leq n^\mu_{t'}(b')=n^\lambda_{t'}(b'),
\]
that is, $s\in S_0^\lambda\cap T_{h}^\lambda$, which contradicts the minimality of $r_\lambda$.
\end{proof}

\begin{lemma}\label{dp2lem4}
With the setting in Proposition \ref{ssprop2}, we further assume that $I$ is standard, $|\lambda_i|=a$ and $n_{s_0}^\lambda(0)=0$, where $q_{s_0-1}<i\leq q_{s_0}$ and $1\leq s_0\leq m$. If $\lambda\not\lera s_{e_i}\lambda$, then either $\overline n_{\overline m}=0$ or $(\Phi, \Phi_I, \Phi_J)$ is strongly separable.
\end{lemma}

\begin{proof}
If $\lambda\not\lera s_{e_i}\lambda$, Lemma \ref{dp2lem3} implies $S_0^\lambda\cap T^\lambda=\emptyset$. The proof of Lemma \ref{dp2lem3} also shows that $s_0, m\in T_0^\lambda\subset T^\lambda$. Thus $T^\lambda$ is not empty. Let
\[
\overline\caA=\{b\in\caA\mid n^\lambda_{t}(b)\geq1\ \mbox{for}\ t\in T^\lambda\}.
\]
Then $a\in \overline\caA$. Fix $0<b\in \overline\caA$ and $t\in T^\lambda$ with $n^\lambda_{t}(b)\geq 1$. We can show that $n^\lambda_{s}(b)$ is maximal for any $s\in S^\lambda$ as in Lemma \ref{dp2lem2}. For any $1\leq s\leq m$ with $n^\lambda_{s}(0)=1$, $\lambda\not\lera s_{e_i}\lambda$ and Lemma \ref{cclem2} yield $n^\lambda_{s}(a)=2$. So $s\in S_0^\lambda\subset S^\lambda$. Therefore $\overline n_{\overline m}=0$ when $S^\lambda=\emptyset$. If $S^\lambda\neq\emptyset$, then $(S^\lambda, \overline S)$ is a strongly separable pair of $(\Phi, \Phi_I, \Phi_J)$, where $\overline S=\{1\leq r\leq\overline m\mid a_r\in \overline\caA\}$.
\end{proof}

The proof of Proposition \ref{ssprop2} is completed once we show the following result.

\begin{lemma}\label{dp2lem5}
Proposition \ref{ssprop2} holds when $I$ is not standard.
\end{lemma}
\begin{proof}
Proposition \ref{ssprop2} is proved in Lemma \ref{dp2lem1}-\ref{dp2lem4} when $I$ is standard. If $I$ is not standard, we can consider $\vf(I)\subset\Delta$, which is standard, keeping in mind that $n_m, \overline n_{\overline m}$ are invariant under the map $\vf$. We also need the easy fact that $(\Phi, \Phi_I, \Phi_J)$ is strongly separable when $(\Phi, \Phi_{\vf(I)}, \Phi_{\vf(J)})$ is strongly separable.
\end{proof}

\subsection{The proof of Proposition \ref{ssprop3}}

Let $\Phi=B_n$, $C_n$ or $D_n$.

Fix $0<a, b\in\caA$. For any $\Phi_I$-regular weight $\lambda$, denote
\[
\begin{aligned}
S_0^\lambda=&\{1\leq s\leq m-1\mid n^\lambda_s(a)\geq1\};\\
T_0^\lambda=&\{1\leq s\leq m-1\mid 1\geq n^\lambda_s(b)\geq0=n^\lambda_s(a)\}.
\end{aligned}
\]
Then $S_0^\lambda\cap T_0^\lambda=\emptyset$. If $T_{r-1}^\lambda$ is defined for $r\geq1$, let
\[
T_r^\lambda=\{1\leq s\leq m-1\mid \exists 0<c\in\caA\  \mbox{and}\ t\in T_{r-1}^\lambda\ \mbox{such that}\ n^\lambda_s(c)\leq1\leq n^\lambda_t(c)\}\cup T_{r-1}^\lambda.
\]
Denote $T^\lambda=\bigcup_{r=0}^\infty T_r^\lambda$ and $S^\lambda=\{1, 2, \ldots, m-1\}\backslash T^\lambda$.

\begin{lemma}\label{dp3lem1}
Let $I$ be standard. If $S_0^\lambda\cap T^\lambda\neq\emptyset$ and $n^\lambda_{s_0}(b)\geq n^\lambda_{t_0}(a)$, then $\lambda\lera s_{e_i}s_{e_j}\lambda$ for $q_{s_0-1}<i\leq q_{s_0}$, $q_{t_0-1}<j\leq q_{t_0}$ with $|\lambda_i|=a$, $|\lambda_j|=b$ $(a\neq b)$ and $n_{s_0}^\lambda(a)=n_{t_0}^\lambda(b)=1$, where $1\leq s_0, t_0\leq m-1$ $(s_0\neq t_0)$.
\end{lemma}

\begin{proof}
With Lemma \ref{cclem3}, it suffices to consider the case $n^\lambda_{s_0}(b)=2$ and $n^\lambda_{t_0}(a)=0$. Thus $s_0\in S_0^\lambda$ and $t_0\in T_0^\lambda$. If $S_0^\lambda\cap T^\lambda\neq\emptyset$, let $r_\lambda$ be the smallest positive integer so that $S_0^\lambda\cap T_{r_\lambda}^\lambda\neq\emptyset$. Denote $r=r_\lambda$. Assume that $s\in S_0^\lambda\cap T_r^\lambda$. Then $n_s^\lambda(a)\geq1$ and $s\neq t_0$. Moreover, $s\in T_r^\lambda$ implies the existence of $0<c\in \caA$ and $t\in T_{r-1}^\lambda$ with $n^\lambda_s(c)\leq1\leq n^\lambda_t(c)$. Since $S_0^\lambda\cap T_{r-1}^\lambda=\emptyset$, we get $n^\lambda_t(a)=0$, $c\neq a$ and $t\neq s_0$. Applying Lemma \ref{cclem4}, we can find $\beta\in\{e_k\pm e_l\}$ such that $\mu=s_\beta\lambda$ is $\Phi_I$-regular, where $q_{s-1}<k\leq q_s$ and $q_{t-1}<l\leq q_t$ with $|\lambda_k|=a$ and $|\lambda_l|=c$. Moreover, $\lambda\lera s_\beta\lambda$. So $|\mu_l|=a$ and $n^\mu_t(a)=1$. It follows that $s_{e_l}\mu$ is $\Phi_I$-regular. Recall that $n^\lambda_{s_0}(a)=n^\lambda_{t_0}(b)=1$. So $n^\mu_{s_0}(a)=0$ ($s_0=s$) or $1$ $(s_0\neq s)$, $n^\mu_{t_0}(a)=0$ ($t_0\neq t$) or $1$ ($t_0=t$) and $n^\mu_{t_0}(b)=0$ ($t_0=t$ and $c=b$) or $1$ ($t_0\neq t$ or $c\neq b$). If $t_0\neq t$ or $c\neq b$, then $j\neq l$. Similar reasoning shows that $s_{e_i}s_{e_j}\lambda\lera s_\beta s_{e_i}s_{e_j}\lambda$ (whether or not $s_0=s$). If $t_0=t$ and $c=b$, then $j=l$. This happens only when $s_0\neq s$ since $n^\lambda_{s_0}(b)=2>n^\lambda_s(c)=n^\lambda_s(b)$. Thus $s_{e_i}s_{e_j}\mu$ is $\Phi_I$-regular, keeping in mind that $s_{e_l}\mu=s_{e_j}\mu$ is $\Phi_I$-regular. Let $\gamma\neq\beta$ be the other root of $\{e_k\pm e_l\}$. Then $s_\gamma s_{e_i}s_{e_j}\lambda=s_{e_i}s_{e_j}\mu$ and $s_{e_i}s_{e_j}\lambda\lera s_\gamma s_{e_i}s_{e_j}\lambda$ by Lemma \ref{cclem1}.

Consider the case $r=1$. So $t\in T_0^\lambda$ and $n^\lambda_t(b)\leq 1$. Then $n^\mu_t(b)=0$ ($c=b$) or $n^\mu_t(b)=n^\lambda_t(b)\leq 1$ ($c\neq b$). First assume that $s_0\neq s$, that is, $k\neq i$, $|\mu_i|=a$ and $n^\mu_{s_0}(a)=1$. If $t_0\neq t$ or $c\neq b$, recall that $l\neq j$, $|\mu_j|=b$ and $n^\mu_{t_0}(b)=1\geq n^\mu_{t_0}(a)$. Also recall that $|\mu_l|=a$, $n^\mu_t(a)=1\geq n^\mu_t(b)$. In view of Lemma \ref{cclem3} (relative to $i, l$ and then $l, j$), we have $\mu\lera s_{e_i}s_{e_l}\mu\lera s_{e_l}s_{e_j}(s_{e_i}s_{e_l}\mu)=s_{e_j}s_{e_i}\mu.$
It follows that
\[
\lambda\lera\mu\lera s_{e_i}s_{e_j}\mu=s_{e_i}s_{e_j}s_\beta\lambda=s_\beta s_{e_i}s_{e_j}\lambda\lera s_{e_i}s_{e_j}\lambda.
\]
If $t_0=t$ and $c=b$, then $l=j$, $|\mu_j|=a$ and $n^\mu_{t_0}(a)=1>n^\mu_{t_0}(b)=0$. With Lemma \ref{cclem3} (relative to $i, j$), one gets
\[
\lambda\lera\mu\lera s_{e_i}s_{e_j}\mu=s_{e_i}s_{e_j}s_\beta\lambda=s_{e_i}s_\gamma s_{e_j}\lambda=s_\gamma s_{e_i}s_{e_j}\lambda\lera s_{e_i}s_{e_j}\lambda.
\]
Next assume that $s=s_0$, that is, $k=i$ and $n^\mu_{s_0}(a)=0$. If $t\neq t_0$ or $c\neq b$, we also get $l\neq j$, $|\mu_j|=b$ and $n^\mu_{t_0}(b)=1$. With $|\mu_l|=a$, $n^\mu_t(b)\leq 1=n^\mu_t(a)$ and $n^\mu_{t_0}(a)\leq1$, similar argument shows that $\mu\lera s_{e_l}s_{e_j}\mu$ and
\[
\lambda\lera\mu\lera s_{e_l}s_{e_j}\mu=s_{e_l}s_{e_j}s_\beta\lambda=s_{e_l}s_\beta s_{e_j}\lambda=s_\beta s_{e_i}s_{e_j}\lambda\lera s_{e_i}s_{e_j}\lambda.
\]
If $t_0=t$ and $c=b$, we get $2=n_{s_0}^\lambda(b)\leq 1\leq n_{t_0}^\lambda(b)$. This is impossible.

Assume that $r>1$. With $t_0\in T_0^\lambda$, if $n^\lambda_s(b)\leq 1=n^\lambda_{t_0}(b)$, then $s\in T_1^\lambda$ and $r=1$, a contradiction. So $n^\lambda_s(b)=2$. It follows from $n^\lambda_s(c)\leq1$ that $c\neq b$ and $n^\mu_s(b)=2$ (thus $s\not\in T_0^\mu$). On the other hand, the minimality of $r$ implies $t\in T_{r-1}^\lambda\backslash T_{r-2}^\lambda$. We obtain $t\not\in T_{0}^\lambda$ and thus $t\neq t_0$. So $n^\lambda_t(a)=0$ (since $S_0^\lambda\cap T_{r-1}^\lambda=\emptyset$) yields $n^\lambda_t(b)>1$. Then $n^\mu_t(b)=n^\lambda_t(b)=2$ in view of $c\neq b$. Recalling that $|\mu_l|=a$ and $n^\mu_t(a)=1$, we have $t\in S_0^\mu$.

If $t\in T_{r-1}^\mu$, then $r_\mu\leq r-1<r_\lambda$. The induction hypothesis (with $(\lambda, i, j)$ replaced by $(\mu, l, j$) yields $\mu\lera s_{e_l}s_{e_j}\mu$. If $s_0=s$, then $k=i$. One obtains
\[
\lambda\lera\mu\lera s_{e_l}s_{e_j}\mu=s_{e_l}s_{e_j}s_\beta\lambda=s_{e_l}s_\beta s_{e_j}\lambda=s_\beta s_{e_i}s_{e_j}\lambda\lera s_{e_i}s_{e_j}\lambda.
\]
If $s_0\neq s$, then $|\mu_i|=|\lambda_i|=a$ and $k\neq i$. We can also use the induction hypothesis (with $(\lambda, i, j)$ replaced by $(\mu, i, j)$) to obtain $\mu\lera s_{e_i}s_{e_j}\mu$. Hence
\[
\lambda\lera \mu\lera s_{e_i}s_{e_j}\mu=s_{e_i}s_{e_j}s_\beta\lambda=s_\beta s_{e_i}s_{e_j}\lambda\lera s_{e_i}s_{e_j}\lambda.
\]
Now assume that $t\not\in T_{r-1}^\mu$. Note that $n^\mu_{s'}(a')=n^\lambda_{s'}(a')$ for any $a'\in\caA$ and $s'\neq s, t$. If $s\not\in T_{r-1}^\mu$, one has $T_p^\mu=T_p^\lambda$ for $0\leq p<r$ and arrives at the contradiction $t\not\in T_{r-1}^\mu=T_{r-1}^\lambda$. If $s\in T_h^\mu\backslash T_{h-1}^\mu$ for $1\leq h<r$, there exist $t'\in T_{h-1}^\mu=T_{h-1}^\lambda$ and $0<c'\in\caA$ such that $n^\mu_s(c')\leq 1\leq n^\mu_{t'}(c')$. Since $S_0^\lambda\cap T_{r-1}^\lambda=\emptyset$ implies $n^\mu_{t'}(a)=n^\lambda_{t'}(a)=0$, we obtain $c'\neq a$. Note that $n^\mu_s(c')-n^\lambda_s(c')=1$ ($c'=c$) or $0$ ($c'\neq c$). Thus $n^\lambda_s(c')\leq n^\mu_s(c')\leq 1\leq n^\mu_{t'}(c')=n^\lambda_{t'}(c')$. By definition, $s\in T_{h}^\lambda$, which contradicts the minimality of $r=r_\lambda$.
\end{proof}

\begin{lemma}\label{dp3lem2}
Proposition \ref{ssprop3} holds for $\Phi=B_n$, $C_n$ and $D_n$.
\end{lemma}

\begin{proof}
There exist $1\leq s_0, t_0<m$ so that $q_{s_0-1}<i\leq q_{s_0}$, $q_{t_0-1}<j\leq q_{t_0}$. If $\lambda\not\lera s_{e_i}s_{e_j}\lambda$, then $i\neq j$. If $s_{e_i}\lambda$ is not $\Phi_I$-regular, there is $q_{s_0-1}<u\leq q_{s_0}$ such that $\lambda_u=-\lambda_i$. Then $s_{e_{i}}s_{e_j}\lambda$ is $\Phi_I$-regular only when $u=j$. This forces $s_{e_{i}}s_{e_j}\lambda=s_{e_i-e_j}\lambda\lera\lambda$ (since $e_i-e_j\in\Phi_I$ in this case), a contradiction. If $s_{e_j}\lambda$ is not $\Phi_I$-regular, we can arrive at a similar contradiction. Therefore, if we set $a=|\lambda_i|$ and $b=|\lambda_j|$, one must have $n^\lambda_{s_0}(a)=n^\lambda_{t_0}(b)=1$. If $a=b$, then $s_0\neq t_0$ (keeping in mind that $i\neq j$). In view of Lemma \ref{cclem1}, $\lambda\lera s_{e_i+e_j}\lambda=s_{e_{i}}s_{e_j}\lambda$ when $\lambda_i=\lambda_j$, and $\lambda\lera s_{e_i-e_j}\lambda=s_{e_{i}}s_{e_j}\lambda$ when $\lambda_i=-\lambda_j$. If $a\neq b$ and $s_0=t_0$, we also have $\lambda\lera s_{e_i+e_j}\lambda\lera s_{e_i-e_j}s_{e_i+e_j}\lambda=s_{e_{i}}s_{e_j}\lambda$ by Lemma \ref{cclem1} and $e_i-e_j\in\Phi_I$. So we obtain $a\neq b$ and $s_0\neq t_0$. By symmetry, we can also assume that $n_{s_0}^\lambda(b)\geq n_{t_0}^\lambda(a)$. Lemma \ref{cclem3} implies $n^\lambda_{s_0}(b)=2$ and $n^\lambda_{t_0}(a)=0$. Thus $s_0\in S_0^\lambda$ and $t_0\in T_0^\lambda\subset T^\lambda$.

Now apply Lemma \ref{dp3lem1} and get $S_0^\lambda\cap T^\lambda=\emptyset$. Denote
\[
\overline\caA=\{c\in\caA\mid n^\lambda_{s}(c)\geq1\ \mbox{for}\ s\in T^\lambda\}.
\]
If $n^\lambda_s(c)\leq1$ for $0<c\in \overline\caA$, then $n^\lambda_t(c)\leq1$ for some $t\in T^\lambda$. It follows that $s\in T^\lambda$. Therefore $n^\lambda_s(c)=2$ for all $s\in S^\lambda$ and $0<c\in \overline\caA$.

First consider the case $\Phi=B_n$, $C_n$. Since $s_{e_i}\lambda$ is $\Phi_I$-regular. It follows from $\lambda\not\lera s_{e_i}s_{e_j}\lambda$ that $\lambda\not\lera s_{e_i}\lambda$ or $s_{e_i}\lambda\not\lera s_{e_i}s_{e_j}\lambda$. In view of Proposition \ref{ssprop1}, we must have $(n_m, \overline n_{\overline m})=(0, m-1)$ or $(\overline m-1, 0)$ when $(\Phi, \Phi_I, \Phi_J)$ is not strongly separable. If $(n_m, \overline n_{\overline m})=(0, m-1)$, then $n^\lambda_{t_0}(0)=1$ and $0\in \overline\caA$. As in Lemma \ref{dp1lem1} and \ref{dp2lem1}, it can be verified that $(\Phi, \Phi_I, \Phi_J)$ is strongly separable relative to $(S^\lambda, \overline S)$, where $\overline S=\{1\leq r\leq \overline m\mid a_r\in \overline\caA\}$. If $(n_m, \overline n_{\overline m})=(\overline m-1, 0)$, then $(\Phi, \Phi_I, \Phi_J)$ is strongly separable relative to $(S^\lambda\cup\{m\}, \overline S)$.

Now $\Phi=D_n$. If $\lambda\not\lera s_{e_i}s_{e_j}\lambda$ for some $1\leq i, j\leq q_{m-1}$, then  $\lambda\not\lera s_{e_i}\lambda$ or $s_{e_i}\lambda\not\lera s_{e_i}s_{e_j}\lambda$. Proposition \ref{ssprop2} implies $(n_m, \overline n_{\overline m})=(\overline m, m)$ or $\overline n_{\overline m}=0$ when $(\Phi, \Phi_I, \Phi_J)$ is not strongly separable. If $(n_m, \overline n_{\overline m})=(\overline m, m)$, then $(\Phi, \Phi_I, \Phi_J)$ is strongly separable relative to $(S^\lambda\cup\{m\}, \overline S)$, where $\overline S=\{1\leq r\leq \overline m\mid a_r\in \overline\caA\}$. If $(n_m, \overline n_{\overline m})=(0, 0)$, then $(\Phi, \Phi_I, \Phi_J)$ is strongly separable relative to $(S^\lambda, \overline S)$. We only need to consider the case when $\overline n_{\overline m}=0$ and $n_m>0$.

In this case, we investigate the dual problem. It suffices to consider the case when $\lambda\in\Lambda_I^+$. Assume that $\lambda=w\overline\lambda$ for $w\in{}^IW^J$. Let $\lambda'=w^{-1}\overline\lambda'$, where $\overline\lambda'$ is a dominant weight with $\Phi_{\overline\lambda'}=\Phi_I$. We can set $w^{-1}e_i=\pm e_{i'}$ and $w^{-1}e_j=\pm e_{j'}$. Note that $s_{e_i+e_j}s_{e_i-e_j}\lambda=s_{e_i}s_{e_j}\lambda$ is $\Phi_I$-regular, there exist $w_1\in W_I$ and $x\in {}^IW^J$ so that $x\overline\lambda=w_1s_{e_i+e_j}s_{e_i-e_j}\lambda$. We can find $w_2\in W_J$ with $xw_2=w_1s_{e_i+e_j}s_{e_i-e_j}w$. So
\[
s_{e_{i'}}s_{e_{j'}}\lambda'=w^{-1}s_{e_i}s_{e_j}w\lambda'=w^{-1}s_{e_i-e_j}s_{e_i+e_j}\overline\lambda'
=w_2^{-1}x^{-1}w_1\overline\lambda'=w_2^{-1}x^{-1}\overline\lambda'
\]
is $\Phi_J$-regular. Applying Lemma \ref{jflem4}, we get $\lambda'\not\lera s_{e_{i'}}s_{e_{j'}}\lambda'$. Then $s_{e_{i'}}\lambda'$ is $\Phi_J$-regular (whether or not $J$ is standard). It follows from $\lambda'\not\lera s_{e_{i'}}s_{e_{j'}}\lambda'$ that $\lambda'\not\lera s_{e_{i'}}\lambda'$ or $s_{e_{i'}}\lambda'\not\lera s_{e_{i'}}s_{e_{j'}}\lambda'$. With $\overline n_{\overline m}=0$ and $n_m>0$, Proposition \ref{ssprop2} shows that $(\Phi, \Phi_J, \Phi_I)$ is strongly separable. Hence $(\Phi, \Phi_I, \Phi_J)$ is strongly separable by Lemma \ref{sslem2}.
\end{proof}


%
%
\section{Blocks of weakly separable systems}
%
%
In this section, we will determine the blocks of weakly separable systems. We will also show that a system which are not separable contains at most one block.

\subsection{Separable pairs} We need to discover more properties of separable pairs before we can prove our main results. Here $A\subsetneqq B$ means $A$ is a proper subset of $B$. Write $(S_1, \overline S_1)\leq (S_2, \overline S_2)$ when $S_1\subset S_2$ and $\overline S_1\supset \overline S_2$. This defines a partial ordering on $\caD=\caD(\lambda, \Phi_I, \Phi)=\caD(\Phi, \Phi_I, \Phi_J)$, where $\Phi_{\overline\lambda}=\Phi_J$.

\begin{lemma}\label{splem3}
Suppose $(S_1, \overline S_1), (S_2, \overline S_2)\in \caD$. If $S_1\neq S_2$ and $\overline S_1\neq \overline S_2$, then either $(S_1, \overline S_1)\leq (S_2, \overline S_2)$ or $(S_1, \overline S_1)\geq (S_2, \overline S_2)$.
\end{lemma}
\begin{proof}
Fix $\lambda\in{}^IW^J\overline\lambda$. If $S_2$ is not a subset of $S_1$ and $\overline S_2$ is not a subset of $\overline S_1$, we can choose $s\in S_2\backslash S_1$ and $t\in \overline S_2\backslash \overline S_1$. Then $n_s^\lambda(a_t)$ is maximal since $s\in S_2$ and $t\in \overline S_2$. Moreover, $n_s^\lambda(a_t)=0$ since $s\not\in S_1$ and $t\not\in \overline S_1$. This happens only when $\Phi=B_n$, $C_n$, $s=m$ and $t=\overline m$. However, $m\not\in S_1$ and $\overline m\not\in \overline S_1$ imply that $(S_1, \overline S_1)\not\in \caD$ in this case, a contradiction. Now we obtain $S_1\supset S_2$ or $\overline S_1\supset \overline S_2$. If $\overline S_1\supsetneqq \overline S_2$ and $S_1\not\subset S_2$, one can find $s\in S_1\backslash S_2$ and $t\in \overline S_1\backslash \overline S_2$ and lead to a similar contradiction. The argument for the other case is similar.
\end{proof}

Define an equivalent relation ``$\sim$'' on $\caD$. It is generated by the following relations:
\begin{itemize}
\item [(\rmnum{1})] $(S_1, \overline S)\sim (S_2, \overline S)$ when $S_1\subset S_2$.
\item [(\rmnum{2})] $(S, \overline S_1)\sim (S, \overline S_2)$ when $\overline S_1\subset \overline S_2$.
\end{itemize}
Let $(S, \overline S)\sim(S', \overline{S'})$. We say
\[
(S, \overline S)=(S_1, \overline S_1)\sim(S_2, \overline S_2)\sim\ldots\sim(S_k, \overline S_k)=(S', \overline{S'})
\]
is a \emph{fundamental chain} of $(S, \overline S)$ and $(S', \overline{S'})$ if the following conditions holds:
\begin{itemize}
\item [(1)] $S_{i}=S_{i+1}$ or $\overline S_i=\overline S_{i+1}$ for any $1\leq i<k$;
\item [(2)] $(S_{i}, \overline S_{i})>(S_{i+1}, \overline S_{i+1})$ or $(S_{i}, \overline S_{i})<(S_{i+1}, \overline S_{i+1})$ for any $1\leq i<k$.
\end{itemize}

\begin{lemma}\label{splem4}
Let $(S_1, \overline S), (S_2, \overline S)\in \caD$. Then $(S, \overline S)\in \caD$ for any nonempty set $S$ with $S_1\cap S_2\subset S\subset S_1\cup S_2$. Moreover, $(S_1, \overline S)\sim(S_2, \overline S)$.
\end{lemma}

\begin{proof}
Fix $\lambda\in{}^IW^J\overline\lambda$. It suffices to prove the first assertion. For any $s\in S$, we get $s\in S_1$ or $S_2$. So $n_s^\lambda(a_t)$ is maximal when $t\in \overline S$. For any $s\not\in S$, we obtain $s\not\in S_1$ or $s\not\in S_2$. So $n_s^\lambda(a_t)=0$ when $t\not\in \overline S$. It can be verified that $(S, \overline S)$ is a separable pair.
\end{proof}

The proof of the following dual lemma is similar.

\begin{lemma}\label{splem5}
Let $(S, \overline S_1), (S, \overline S_2)\in \caD$. Then $(S, \overline S)\in \caD$ for any nonempty set $\overline S$ with $\overline S_1\cap \overline S_2\subset \overline S\subset \overline S_1\cup \overline S_2$. Moreover, $(S, \overline S_1)\sim(S, \overline S_2)$.
\end{lemma}

\begin{lemma}\label{splem6}
Let $(S_1, \overline S_1), (S_2, \overline S_2)\in \caD$. Suppose that $|S_1|=|S_2|$ or $|\overline S_1|=|\overline S_2|$. Then either $S_1=S_2$ or $\overline S_1=\overline S_2$. Moreover, $(S_1, \overline S_1)\sim(S_2, \overline S_2)$.
\end{lemma}
\begin{proof}
In view of Lemma \ref{splem3}, one has $S_1=S_2$ or $\overline S_1=\overline S_2$. Then the lemma follows from Lemma \ref{splem4} and Lemma \ref{splem5}.
\end{proof}

\begin{lemma}\label{splem7}
Suppose that $(S, \overline S)>(S', \overline{S'})$ are two equivalent separable pairs with a shortest fundamental chain $(S_1, \overline S_1)\sim(S_2, \overline S_2)\sim\ldots\sim(S_k, \overline S_k)$. Then
\begin{itemize}
  \item [(1)] For any $1\leq i<k$, $S_{i}=S_{i+1}$ or $\overline S_i=\overline S_{i+1}$;
  \item [(2)] For any $1\leq i<j\leq k$, $S_i\neq S_j$ $($resp. $\overline S_i\neq \overline S_j$$)$ unless $j=i+1$;
  \item [(3)] $(S_1, \overline S_1)>(S_2, \overline S_2)>\ldots>(S_k, \overline S_k)$;
\end{itemize}
\end{lemma}
\begin{proof}
The definition gives (1). We consider (2). Choose $1\leq i<j\leq k$ so that
\[
\overline S_i=\overline S_{i+1}=\ldots=\overline S_{j-1}=\overline S_j
\]
and $j-i$ as large as possible. If $j<k$, then $\overline S_j\neq \overline S_{j+1}$. (1) implies $S_j=S_{j+1}$. Moreover, we have $\overline S_j\subsetneqq \overline S_{j+1}$ or $\overline S_j\supsetneqq\overline S_{j+1}$ since the chain is fundamental. Lemma \ref{splem3} yields either $(S_{i}, \overline S_{i})>(S_{j+1}, \overline S_{j+1})$ or $(S_{i}, \overline S_{i})<(S_{j+1}, \overline S_{j+1})$. Thus
\[
(S_1, \overline S_1)\sim\ldots\sim(S_i, \overline S_i)\sim (S_j, \overline S_j)\sim(S_{j+1}, \overline S_{j+1})\sim\ldots\sim(S_k, \overline S_k)
\]
is a fundamental chain of length $k-j+1+i\geq k$, that is, $j\leq i+1$. If $i>1$, we can get $\overline S_{i-1}\neq \overline S_{i}$ and $j\leq i+1$ by a similar reasoning. If $i=1$ and $j=k$, then $\overline S_1=\overline S_k$. So $(S, \overline S)\sim(S', \overline{S'})$ is a fundamental chain and $j=2=i+1$. Now suppose $\overline S_i=\overline S_j$ for some $1\leq i<j\leq k$. Then
\[
(S_1, \overline S_1)\sim\ldots\sim(S_i, \overline S_i)\sim (S_i\cup S_j, \overline S_j)\sim(S_j, \overline S_j)\sim\ldots\sim(S_k, \overline S_k)
\]
is a fundamental chain. The previous argument shows that it is not shortest. Thus  $k-j+2+i>k$, that is, $j\leq i+1$. The dual case is similar.

Next we prove the following result by induction on $i$.
\begin{equation}\label{spl7eq1}
(S_1, \overline S_1)>(S_2, \overline S_2)>\ldots>(S_i, \overline S_i)\geq(S_k, \overline S_k).
\end{equation}
Then (3) is an easy consequence. The case $i=1$ is trivial. Suppose (\ref{spl7eq1}) holds for $i\geq1$. If $i=k$, there is nothing to prove. If $i<k$ and $S_i=S_k$ (resp. $\overline S_i=\overline S_k$), then (2) yields $i=k-1$. Moreover, (\ref{spl7eq1}) is true for $i+1=k$. Now suppose that $S_i\supsetneqq S_k$ and $\overline S_i\subsetneqq \overline S_k$. We can choose $i\leq j<k$ so that $|S_i|\leq|S_j|$ and $|S_i|>|S_{j+1}|$. Then (2) implies $\overline S_j=\overline S_{j+1}$. So $\overline S_i\subset\overline S_{j+1}=\overline S_j$ in view of Lemma \ref{splem3}. If $\overline S_i\neq \overline S_{j}$, one must have $S_i\supset S_j$ and thus $S_i=S_j$ (since $|S_i|\leq |S_j|$). Thus (2) implies $j=i+1$. With $S_{i+1}=S_i\supsetneqq S_k$, we obtain $(S_{i+1}, \overline S_{i+1})>(S_k, \overline S_k)$ by Lemma \ref{splem3}. It remains to consider the case $\overline S_i=\overline S_{j}=\overline S_{j+1}$. Then (2) forces $j=i$. Hence $(S_{i}, \overline S_{i})>(S_{i+1}, \overline S_{i+1})>(S_k, \overline S_k)$ by Lemma \ref{splem3}.

\end{proof}

\begin{lemma}\label{splem8}
Let $\Phi=B_n$, $C_n$ $($resp. $D_n$$)$, $\Phi_{\overline\lambda}=\Phi_J$ and $I$ be standard. Suppose that $(n_m, \overline n_{\overline m})=(0, m-1)$ or $(\overline m-1, 0)$ $($resp. $\overline n_{\overline m}=0$ or $(n_m, \overline n_{\overline m})=(\overline m, m)$$)$. Let $(S, \overline S)\in \caD$. Then the following conditions are equivalent.
\begin{itemize}
  \item [(1)] $s_{e_i}\lambda$ is $\Phi_I$-singular for $\Phi_I$-regular weight $\lambda\in W\overline\lambda$ with $q_{s-1}<i\leq q_s$, $\lambda_i\neq0$ and $s\in S\backslash\{m\}$.
  \item [(2)] There exists $(S', \overline{S'})\sim (S, \overline S)$ so that $S'\subset S$, $|S'|=1$ and $|\overline{S'}\backslash\{\overline m\}|=\overline m-1$.
\end{itemize}
Under these hypothesis, $\lambda_j$ is fixed for any $q_{s-1}<j\leq q_s$ with $s\in S$ and $\lambda\in{}^IW^J\overline\lambda$.
\end{lemma}

\begin{proof}
Let $\lambda\in W\overline\lambda$ be any $\Phi_I$-regular weight. Use induction on $|S|$. First assume that $|S|=1$. If $S=\{m\}$, then either $\Phi=B_n, C_n$ with $(n_m, \overline n_{\overline m})=(\overline m-1, 0)$ or $\Phi=D_n$ with $(n_m, \overline n_{\overline m})=(\overline m, m)$. We obtain $(S, \overline S)\sim(\{m\}, \{1, \ldots, \overline m-1\})$ (resp. $(\{m\}, \{1, \ldots, \overline m\})$) for $\Phi=B_n, C_n$ (resp. $D_n$). The equivalence of (1) and (2) is trivial. If $S=\{s_0\}$ for some $s_0<m$, then either $\Phi=B_n, C_n$ with $(n_m, \overline n_{\overline m})=(0, m-1)$ or $\Phi=D_n$ with $\overline n_{\overline m}=0$. With $(S, \overline S)\in\caD$, one gets $n_s^\lambda(a_t)=0$ for $s\neq s_0$ and $t\not\in\overline S$. In view of Lemma \ref{lrlem2}, we obtain $n_{s_0}^\lambda(a_t)\geq1$ for $t\not\in\overline S\backslash\{m\}$ since $\overline n_t>0$. Thus (1) holds if and only if $n_{s_0}^\lambda(a_t)=2$ for $t\not\in\overline S\backslash\{m\}$, keeping in mind of Lemma \ref{lrlem1}. In the other direction, (2) holds if and only if $S'=\{s_0\}$ and $n_{s_0}^\lambda(a_t)=2$ for $t\not\in\overline S\backslash\{m\}$.

Now suppose $|S|>1$. Assume that (1) holds. choose $s_0\in S$ such that $n_{s_0}=\min\{n_s\mid m>s\in S\}$. One must have $n^\lambda_{s}(a_t)=0$ or $2$ for $s\in S\backslash\{m\}$ and $1\leq t<\overline m$. Denote $S_2=S\backslash\{s_0\}$ and
\[
\overline S_2=\{1\leq t\leq\overline m\mid n_{s_0}^\lambda(a_t)>0\}.
\]
We show that $(S_2, \overline S_2)\sim (S, \overline S)$. Evidently $\overline S_2\supset \overline S$ since $n^\lambda_{s_0}(a_t)$ is maximal for $t\in\overline S$. If $m\in S$, then either $\Phi=B_n, C_n$ with $(n_m, \overline n_{\overline m})=(\overline m-1, 0)$ or $\Phi=D_n$ with $(n_m, \overline n_{\overline m})=(\overline m, m)$. We always have $n_m^\lambda(a_t)$ is maximal for $1\leq t\leq \overline m$. If $\overline m\in\overline S$, then either $\Phi=B_n, C_n$ with $(n_m, \overline n_{\overline m})=(0, m-1)$ or $\Phi=D_n$ with $(n_m, \overline n_{\overline m})=(\overline m, m)$. So $n_s^\lambda(0)$ is maximal for $s\in S$. If $n_s^\lambda(a_t)=0$ for some $s\in S_2\backslash\{m\}$ and $t\in\overline S_2\backslash\{\overline m\}$, we can find $t_0\not\in\overline S_2$ with $n^\lambda_{s}(a_{t_0})=2$ and $n^\lambda_{s_0}(a_{t_0})=0$, keeping in mind that $n_{s_0}\leq n_s$. So $t_0\not\in \overline S$ (otherwise $n^\lambda_s(a_{t_0})$ is maximal). Lemma \ref{cclem4} gives $\Phi_I$-regular weight $\mu=s_{e_j-e_k}\lambda$, where $q_{s_0-1}<j\leq q_{s_0}$ and $q_{s-1}<k\leq q_s$. Moreover, $s_{e_j}\mu$ and $s_{e_k}\mu$ are also $\Phi_I$-regular, which contradicts (1). Therefore $n_s^\lambda(a_t)$ is maximal for $s\in S$ and $t\in S_2$. One has $(S_2, \overline S_2), (S, \overline S_2)\in\caD$ and thus $(S_2, \overline S_2)\sim (S, \overline S)$. With $(S, \overline S)$ replaced by $(S_2, \overline S_2)$, the induction hypothesis yields $(S', \overline{S'})\sim (S_2, \overline S_2)$ so that $S'\subset S_2$, $|S'|=1$ and $|\overline{S'}\backslash\{m\}|=m-1$. Now (2) is an immediate consequence.

Conversely, if (2) is true, let
\[
(S, \overline S)=(S_1, \overline S_1)\sim(S_2, \overline S_2)\sim\ldots\sim(S_k, \overline S_k)=(S', \overline{S'})
\]
be a shortest fundamental chain of $(S, \overline S)$ and $(S', \overline{S'})$. If $s\in S_i\backslash S_{i+1}$, then $\overline S_i=\overline S_{i+1}$ in view of Lemma \ref{splem7}. We obtain $n_s^\lambda(a_t)$ is maximal for $t\in \overline S_i$ and is zero for $t\not\in \overline S_i$. This is also true for $s\in S'=S_k$. Then $s_{e_i}\lambda$ is $\Phi_I$-singular for $q_{s-1}<i\leq q_s$. Hence (1) holds since $S=\bigcup_{i=1}^{k}(S_i\backslash S_{i+1})$, where $S_{k+1}=\emptyset$.
\end{proof}

In a similar spirit, we can prove the following result by a dual argument.

\begin{lemma}\label{splem9}
Let $\Phi=B_n$, $C_n$ $($resp. $D_n$$)$ and $I$ be standard. Suppose that $(n_m, \overline n_{\overline m})=(\overline m-1, 0)$ or $(0, m-1)$ $($resp. $n_m=0$ or $(n_m, \overline n_{\overline m})=(\overline m, m)$$)$. Let $(S, \overline S)\in \caD$. Then the following conditions are equivalent.
\begin{itemize}
  \item [(1)] $s_{e_i}\lambda$ is $\Phi_I$-singular for any $\Phi_I$-regular weight $\lambda\in W\overline\lambda$ with $q_{s-1}<i\leq q_s$, $\lambda_i\neq0$ and $s\not\in S$;
  \item [(2)]  There exists $(S', \overline{S'})\sim (S, \overline S)$ so that $S'\supset S$, $|\overline{S'}|=1$ and $|S'\backslash\{m\}|=m-1$.
\end{itemize}
Under these hypothesis, $\lambda_i$ is fixed for any $q_{s-1}<i\leq q_s$ with $s\not\in S$ and $\lambda\in{}^IW^J\overline\lambda$.
\end{lemma}


\subsection{Special weights} In this subsection, we show that any $\Phi_I$-regular weight of the system $(\Phi, \Phi_I, \Phi_J)$ is connected to some special weights. First we recall the restriction of weights. If $\Phi'$ is a subsystem of $\Phi$, there exists a unique weight $\lambda|_{\Phi'}$ in the subspace $\bbC\Phi'$ so that
\begin{equation}\label{reseq1}
\langle\lambda|_{\Phi'}, \alpha^\vee\rangle=\langle\lambda, \alpha^\vee\rangle
\end{equation}
for all $\alpha\in\Phi'$. We write $\lambda\perp\Phi'$ if $\lambda|_{\Phi'}=0$. Of course $\lambda|_{\Phi'}$ can also be viewed as a weight in $\frh^*$. Thus one has $(\lambda-\lambda|_{\Phi'})\perp\Phi'$ for any $\lambda\in\frh^*$. For convenience, we also write $S\perp\Phi'$ for $S\subset\bbC\Phi$ when $\lambda\perp\Phi'$ for every weight $\lambda\in S$. The following result is evident.

\begin{lemma}\label{reslem1}
Let $\Phi'$ be a subsystem of $\Phi$. Choose $\nu\in\frh^*$ $($ not necessarily a root$)$.
\begin{itemize}
  \item [(1)] If $\nu\perp\Phi'$, then $(s_\nu\lambda)|_{\Phi'}=\lambda|_{\Phi'}$ for any $\lambda\in\frh^*$.
  \item [(2)] If $\nu\in\bbC\Phi'$, then $(s_\nu\lambda)|_{\Phi'}=s_\nu(\lambda|_{\Phi'})$ for any $\lambda\in\frh^*$.
\end{itemize}
\end{lemma}

Start with the case of type $A$.

\begin{lemma}\label{bkswlem1}
Let $\Phi=A_{n-1}$. Choose $I, J\subset\Delta$. Let $\lambda$ be a $\Phi_I$-regular weight with $\Phi_{\overline\lambda}=\Phi_J$. Suppose that $i_1, \ldots, i_{\overline m}$ is a permutation of $\{1, \ldots, \overline m\}$ such that $\overline n_{i_1}\geq\ldots\geq\overline n_{i_{\overline m}}$. Set $b_l=a_{i_l}$ for $1\leq l\leq\overline m$. Then we can find a $\Phi_I$-regular weight
\begin{equation}\label{bkswl1eq1}
{\hat{\lambda}}=(b_1, \ldots, b_{q_1}, {\hat{\lambda}}_{q_1+1}, \ldots, {\hat{\lambda}}_{n})
\end{equation}
so that $\lambda\lera\hat\lambda$.
\end{lemma}
\begin{proof}
We can construct a $\Phi_I$-regular weight $\mu\lera\lambda$ and $n_1^\mu(b_1)\geq\ldots\geq n_1^\mu(b_{\overline m})$ step by step. Start with $\mu=\lambda$. Suppose we already have $n_1^\mu(b_1)\geq\ldots\geq n_1^\mu(b_k)$ and $n_1^\mu(b_k)<n_1^\mu(b_{k+1})$ for some $1\leq k<\overline m$. Choose the smallest $0\leq j<k$ so that $n_1^\mu(b_{j+1})=0=n_1^\mu(b_k)<1=n_1^\mu(b_{k+1})$. Since
\[
\sum_{s=1}^mn_s^\lambda(b_{j+1})=\overline n_{i_{j+1}}\geq \overline n_{i_{k+1}}
=\sum_{s=1}^mn_s^\lambda(b_{k+1}),
\]
there is $1<s\leq m$ so that $n_s^\lambda(b_{j+1})>n_s^\lambda(b_{k+1})$. We can assume that $\mu_l=b_{k+1}$ for some $1\leq l\leq q_1$ and $\mu_h=b_{j+1}$ for some $q_{s-1}<h\leq q_s$. Then $\nu=s_{e_l-e_h}\mu$ is $\Phi_I$-regular and $\mu\lera\nu$ in view of Proposition \ref{lrprop2}. Moreover, $n_1^\nu(b_1)\geq\ldots\geq n_1^\nu(b_k)\geq n_1^\nu(b_{k+1})$. Replacing $\mu$ by $\nu$ and increasing $k$ in this fashion, we can eventually get
\[
n_1^\mu(b_1)=\ldots=n_1^\mu(b_{q_1})=1>0=n_1^\mu(b_{q_1+1})=\ldots=n_1^\mu(b_{\overline m})
\]
with $\mu\lera\lambda$. There exists $w\in W_I$ so that
\[
w\mu=(b_1, \ldots, b_{q_1}, {\hat{\lambda}}_{q_1+1}, \ldots, {\hat{\lambda}}_{n+1}).
\]
Hence one has $\lambda\lera\hat\lambda$, where $\hat\lambda=w\mu$.
\end{proof}

\begin{lemma}\label{bkswlem2}
Let $\Phi=A_{n-1}$. Choose $I, J\subset\Delta$. Let $\overline\lambda$ be a dominant weight with $\Phi_{\overline\lambda}=\Phi_J$. There exists a $\Phi_I$-regular weight $\mu$ such that $\lambda\lera\mu$ for any $\Phi_I$-regular weight $\lambda\in W\overline\lambda$.
\end{lemma}
\begin{proof}
We can prove the lemma by induction on $m$. The case $m=1$ is evident since $W_I=W$ and $\lambda=w\overline\lambda\lera\overline\lambda$. If $m>1$, let $\Phi'=\Phi\cap\sum_{i=q_1+1}^{n}\bbC e_i$ and $W'$ be the corresponding Weyl group. Denote $I'=I\cap\Phi'$. Choose $\Phi_I$-regular weight $\lambda\in W\overline\lambda$. With Lemma \ref{bkswlem1}, we can find $\Phi_I$-regular weight ${\hat{\lambda}}\lera\lambda$ satisfying (\ref{bkswl1eq1}). The induction hypothesis implies ${\hat{\lambda}}|_{\Phi'}\lera \mu'$ ( relative to $(\Phi_{I'}, \Phi')$) for some $\Phi'_{I'}$-regular weight $\mu'\in\sum_{i=q_1+1}^{n}\bbC e_i$. Set $\mu'_i=\langle\mu', e_i\rangle$ for $q_1<i\leq n$.   Denote
\[
\mu=(\hat\lambda_1, \ldots, \hat\lambda_{q_1}, \mu'_{q_1+1}, \ldots, \mu'_n).
\]
One has $\hat{\lambda}\lera \mu$, in view of Lemma \ref{lrlem0} and Proposition \ref{lrprop2}. Hence $\lambda\lera\hat{\lambda}\lera\mu$.
\end{proof}

With Lemma \ref{bkswlem2}, we can recover the following result of Brundan \cite{Br}.

\begin{theorem}\label{bkswthm1}
Let $\Phi=A_n$. The system $(\Phi, \Phi_I, \Phi_J)$ contains at most one block for any $I, J\subset\Delta$.
\end{theorem}

The above results of type $A$ inspire us to find similar results for the other types.

\begin{lemma}\label{bkswlem3}
Assume that $\Phi=B_n$, $C_n$. Choose $I, J\subset\Delta$. Let $\overline\lambda$ be a dominant weight with $\Phi_{\overline\lambda}=\Phi_J$ and $\lambda\in W\overline\lambda$ be $\Phi_I$-regular. Suppose that $m>1$. Let $i_1, \ldots, i_{\overline m-1}$ be a permutation of $\{1, \ldots, \overline m-1\}$ such that $\overline n_{i_1}\geq\ldots\geq\overline n_{i_{\overline m-1}}$. Set $b_l=a_{i_l}$ for $1\leq l<\overline m$ and $p=[\frac{n_1+1}{2}]$.
\begin{itemize}
\item [(1)] If $p=\overline m$ or $\overline n_{i_p}\leq 2\overline n_{\overline m}$, we can find $\Phi_I$-regular weight
\[
\hat\lambda=(b_1, \ldots, b_k, 0, -b_{j}, \ldots, -b_1, \hat\lambda_{n_1+1} \ldots, \hat\lambda_{n})
\]
so that $\lambda\lera\hat\lambda$ or $\lambda\lera s_{e_k}\hat\lambda$, where $k$ is determined by $\overline\lambda$ and $j=n_1-1-k$.

\item [(2)] If $p<\overline m$ and $\overline n_{i_p}>2\overline n_{\overline m}$, we can find $\Phi_I$-regular weight
\[
\hat\lambda=(b_1, \ldots, b_k, -b_{j}, \ldots, -b_1, \hat\lambda_{n_1+1} \ldots, \hat\lambda_{n})
\]
so that $\lambda\lera\hat\lambda$ or $\lambda\lera s_{e_k}\hat\lambda$, where $k$ is determined by $\overline\lambda$ and $j=n_1-k$.
\end{itemize}
\end{lemma}
\begin{proof}
(1) If $p=\overline m$, then $n_1=2\overline m-1$ and all $n_s^\lambda(a_s)$ are maximal, keeping in mind of $n_1=\sum_{s=1}^{\overline m}n_s^\lambda(a_s)\leq 2\overline m-1$. There exists $w\in W_{I_1}$ so that
\[
\lambda\lera w\lambda=\hat\lambda=(b_1, \ldots, b_{p-1}, 0, -b_{p-1}, \ldots, -b_1, \hat\lambda_{n_1+1} \ldots, \hat\lambda_{n}).
\]
Here $k=j=p-1$.

Assume that $p<\overline m$ and $\overline n_{i_p}\leq 2\overline n_{\overline m}$. We claim there is $\mu\lera\lambda$ satisfying:
\begin{itemize}
  \item [(\rmnum{1})] $n^\mu_1(0)=1$;
  \item [(\rmnum{2})] $n^\mu_1(b_1)\geq \ldots\geq  n^\mu_1(b_{\overline m-1})$.
\end{itemize}
With (\rmnum{1}) and (\rmnum{2}), one obtains $n^\mu_1(b_1)=\ldots=n^\mu_1(b_j)=2$, $n^\mu_1(b_{j+1})=\ldots=n^\mu_1(b_k)=n^\mu_1(0)=1$ and $n^\mu_1(b_{k+1})=\ldots=n^\mu_1(b_{\overline m-1})=0$ for some $0\leq j\leq k<\overline m$. We can continue to show that there exists such a $\mu\lera\lambda$ satisfying
\begin{itemize}
  \item [(\rmnum{3})] $\overline n_{i_{j+1}}=\overline n_{i_k}$ when $j<k$;
  \item [(\rmnum{4})] $\overline n_{i_j}>\overline n_{i_{k+1}}$ when $j>0$ and $k<\overline m-1$.
\end{itemize}
Suppose we already have $\mu\lera\lambda$ satisfying (\rmnum{1})-(\rmnum{4}). Obviously $k$ is unique when $\overline\lambda$ is fixed. Since $w\mu\lera\mu$ for any $w\in W_{I_1}$, one can assume that
\[
\mu=(b_1, \ldots, b_{j}, \mu_{j+1}, \ldots, \mu_k, 0, -b_{j}, \ldots, -b_1, \mu_{n_1+1}, \ldots, \mu_n),
\]
where $|\mu_l|=b_l$ for $j<l\leq k$. For any $j<l<k$ with $\mu_l=-b_l$, replace $\mu$ by $s_{e_l+e_k}s_{e_l-e_k}\mu$. We still have $\mu\lera\lambda$ by Lemma \ref{cclem1} and eventually get $\mu_l=b_l$ for $l<k$. If $\mu_k=b_k$, set $\hat{\lambda}=\mu$, otherwise set $\hat{\lambda}=s_{e_k}\mu$.

It remains to find $\mu\lera\lambda$ satisfying (\rmnum{1})-(\rmnum{4}). The main tool is Lemma \ref{cclem4}. Start with $\mu=\lambda$. If $n^\lambda_1(0)=1$, then (\rmnum{1}) holds. If $n^\lambda_1(0)=0$, note that $\sum_{j=1}^{p-1}n^\lambda_1(b_j)\leq 2(p-1)<n_1$. There exists $p\leq j<\overline m$ with $n^\lambda_1(b_j)\geq 1$. In view of
\[
n^\lambda_1(b_j)+\sum_{\substack{1<s\leq m\\
n^\lambda_s(0)=1}}n^\lambda_s(b_j)
\leq \sum_{s=1}^{m}n^\lambda_s(b_j)=\overline n_{i_j}\leq \overline n_{i_p}\leq2\overline n_{\overline m}=2\sum_{s=1}^{m}n^\lambda_s(0),
\]
We can find $1<s<m$ with $n^\lambda_s(0)=1$ and $n^\lambda_s(b_j)\leq 1$ ($n_m^\lambda(0)=0$ in view of Lemma \ref{lrlem2}). Applying Lemma \ref{cclem4} (relative to $b_j, 0$), there exist $\beta\in\Phi\backslash\Phi_I$ so that $\mu=s_\beta\lambda$ is $\Phi_I$-regular. Moreover, $\mu\lera\lambda$ and $n^\mu_1(0)=n_1^\lambda(0)+1=1$. Now consider (\rmnum{2}). If $n^\mu_1(b_1)\geq\ldots\geq n^\mu_1(b_k)<n^\mu_1(b_{k+1})$ (which is $1$ or $2$) for $1\leq k<\overline m-1$, choose the smallest $0\leq j<k$ so that $n^\mu_1(b_{j+1})=n^\mu_1(b_k)$. One can choose $1<s\leq m$ with $n^\mu_s(b_{j+1})>n^\mu_s(b_{k+1})$, keeping in mind that
\[
\sum_{s=1}^mn^\mu_s(b_{j+1})=\overline n_{i_{j+1}}\geq \overline n_{i_{k+1}}=\sum_{s=1}^mn^\mu_s(b_{k+1}).
\]
With Lemma \ref{cclem4} (relative to $b_{j+1}, b_{k+1}$), there is $\nu\lera \mu$ so that $n^\nu_1(b_{k+1})=n^\mu_1(b_{k+1})-1$. If $n^\nu_1(b_k)\geq n^\nu_1(b_{k+1})$, we can replace $\mu$ by $\nu$ and get $n^\mu_1(b_1)\geq\ldots\geq n^\mu_1(b_k)\geq n^\mu_1(b_{k+1})$. If $n^\nu_1(b_k)<n^\nu_1(b_{k+1})$ (which must be $1$), imitate the previous argument (with $j$ replaced by $j+1$). We get $\eta\lera\nu$ and $n^\eta_1(b_1)\geq\ldots\geq n^\eta_1(b_k)\geq n^\eta_1(b_{k+1})=0$. Then replace $\mu$ by $\eta$. Increasing $k$ in this fashion, one obtains (\rmnum{2}).

Now we have (\rmnum{1}) and (\rmnum{2}) for ($\mu$, $k$). If $\overline n_{i_{j+1}}>\overline n_{i_k}$ and $j<k$, there is $1<t\leq m$ with $n^\mu_t(b_{j+1})>n^\mu_t(b_k)$, keeping in mind that $n^\mu_1(b_{j+1})=n^\mu_1(b_k)=1$. Lemma \ref{cclem4} yields $\Phi_I$-regular weight $\nu\lera\mu\lera\lambda$ such that $n^\nu_1(b_{j+1})=2$ and $n^\nu_1(b_k)=0$. It is evident that (\rmnum{1}) and (\rmnum{2}) also hold for ($\nu$, $k-1$). Replacing $\mu$ by $\nu$ and decreasing $k$ stepwise in this fashion, we arrive at (\rmnum{3}). Now $\overline n_{i_j}\geq\overline n_{i_{k+1}}$. We only need to consider the case $\overline n_{i_j}=\overline n_{i_{k+1}}$ with $j>0$ and $k<\overline m-1$. There exists $1<t\leq m$ with $n^\mu_t(b_{j})<n^\mu_t(b_{k+1})$, in view of $n^\mu_1(b_{j})=2>0=n^\mu_1(b_{k+1})$. Lemma \ref{cclem4} gives $\nu\lera\mu\lera\lambda$ such that $n^\nu_1(b_{j})=n^\nu_1(b_{k+1})=1$. Obviously (\rmnum{1})-(\rmnum{3}) still hold with $(\mu, k)$ replaced by $(\nu, k+1)$. Increasing $k$ stepwise, we will obtain (\rmnum{4}).

(2) If $p<\overline m$ and $\overline n_{i_p}>2\overline n_{\overline m}$, we first give $\mu\lera\lambda$ with $n^\mu_1(0)=0$. Then (2) follows from an argument similar to (1). If $n^\lambda_1(0)=0$, set $\mu=\lambda$. If $n^\lambda_1(0)=1$, with $n^\lambda_1(0)+\sum_{j=1}^{p}n^\lambda_1(b_j)\leq n_1<2p+1$, we must have $n^\lambda_1(b_j)\leq 1$ for some $1\leq j\leq p$. With $1\geq n^\lambda_m(b_j)>0=n^\lambda_m(0)$ and $\overline n_{i_j}\geq\overline n_{i_p}>2\overline n_{\overline m}$, one gets $n^\lambda_s(b_j)>n^\lambda_s(0)=0$ for some $1<s<m$. So Lemma \ref{cclem4} gives $\mu\lera\lambda$ with $n^\mu_1(0)=0$.

\end{proof}

If $\Phi=D_n$, the argument is almost the same. The only difference is that $n^\lambda_m(0)$ might be nonzero for some $\Phi_I$-regular weight $\lambda$ in this case. So we need to make corresponding changes on conditions about $\overline n_{i_p}$ and $2\overline n_{\overline m}$. The result can be summarized as follows:

\begin{lemma}\label{bkswlem4}
Assume that $\Phi=D_n$. Choose $I, J\subset\Delta$ with $I$ being standard. Let $\overline\lambda$ be a dominant weight with $\Phi_{\overline\lambda}=\Phi_J$ and $\lambda\in W\overline\lambda\cup W s_{e_n}\overline\lambda$ be $\Phi_I$-regular. Suppose that $m>1$. Let $i_1, \ldots, i_{\overline m-1}$ be a permutation of $\{1, \ldots, \overline m-1\}$ such that $\overline n_{i_1}\geq\ldots\geq\overline n_{i_{\overline m-1}}$. Set $b_l=a_{i_l}$ for $1\leq l<\overline m$ and $p=[\frac{n_1+1}{2}]$.
\begin{itemize}
\item [(1)] If $p=\overline m$ or $\overline n_{i_p}<2\overline n_{\overline m}$, we can find $\Phi_I$-regular weight
\[
\hat\lambda=(b_1, \ldots, b_k, 0, -b_{j}, \ldots, -b_1, \hat\lambda_{n_1+1} \ldots, \hat\lambda_{n})
\]
so that $\lambda\lera\hat\lambda$ or $\lambda\lera s_{e_k}\hat\lambda$, where $k$ is determined by $\overline\lambda$ and $j=n_1-1-k$.

\item [(2)] If $p<\overline m$ and $\overline n_{i_p}\geq2\overline n_{\overline m}$, we can find $\Phi_I$-regular weight
\[
\hat\lambda=(b_1, \ldots, b_k, -b_{j}, \ldots, -b_1, \hat\lambda_{n_1+1} \ldots, \hat\lambda_{n})
\]
so that $\lambda\lera\hat\lambda$ or $\lambda\lera s_{e_k}\hat\lambda$, where $k$ is determined by $\overline\lambda$ and $j=n_1-k$.
\end{itemize}
\end{lemma}

\begin{remark}\label{bkswrmk2}
The proof of Lemma \ref{bkswlem3} and Lemma \ref{bkswlem4} shows that $k\geq\frac{n_1-1}{2}\geq0$. If $k=0$ (e.g., when $n_1=p=\overline m=1$), one has $\lambda\lera\hat\lambda$ for all $\Phi_I$-regular weight $\lambda\in W\overline\lambda$.

If $\Phi=D_n$ and $\overline\lambda_n\neq0$, then $W\overline\lambda\cap W s_{e_n}\overline\lambda=\emptyset$. If $\Phi=B_n, C_n$ or $\overline\lambda_n=0$, we always have $W\overline\lambda=W s_{e_n}\overline\lambda$.
\end{remark}

Similar to Lemma \ref{bkswlem2}, we have the following result.

\begin{lemma}\label{bkswlem5}
Assume that $\Phi=B_n$ or $C_n$. Choose $I, J\subset\Delta$. Let $\overline\lambda$ be a dominant weight with $\Phi_{\overline\lambda}=\Phi_J$. There exists $\Phi_I$-regular weight $\mu$ so that
\[
\lambda\lera s_{e_{k_1}}^{\eps_1}s_{e_{k_2}}^{\eps_2}\ldots s_{e_{k_{m-1}}}^{\eps_{m-1}}\mu
\]
for any $\Phi_I$-regular weight $\lambda\in W\overline\lambda$. Here $q_{s-1}<k_s\leq q_s$ and $\eps_s\in\{0, 1\}$ for $1\leq s<m$. In particular, if $\eps_s=1$, then $\mu_{k_s}\neq0$ and $s_{e_{k_s}}\mu$ is $\Phi_I$-regular.
\end{lemma}
\begin{proof}
We prove the lemma by induction on $m$. If $m=1$, that is, $I=\Delta$, there exists $w\in W_I=W$ so that $\lambda=w\overline\lambda\lera \overline \lambda$. Set $\mu=\overline\lambda$. If $m>1$, it follows from Lemma \ref{bkswlem3} that we can find $\hat\lambda$ for each $\Phi_I$-regular weight $\lambda\in W\overline\lambda$ such that either $\lambda\lera\hat\lambda$ or $\lambda\lera s_{e_k}\hat\lambda$ with $0\leq k\leq q_1$, where $k$ and $\hat\lambda_1, \ldots, \hat\lambda_{q_1}$ are independent of the choice of $\lambda$. If $k=0$ (see Remark \ref{bkswrmk2}), set $k_1=q_1$ and $\eps_1=0$, otherwise set $k_1=k$. One has $\lambda\lera s_{e_{k_1}}^{\eps_1}\hat\lambda$ with $\eps_1\in\{0, 1\}$. Moreover, $\eps_1=0$ unless $\lambda\not\lera\hat\lambda$. Put
\[
\Phi'=(\sum_{q_1<i\leq n}\bbQ e_i)\cap\Phi\ \mbox{and}\ I'=\bigcup_{1<t\leq m}I_t.
\]
If $\Phi'=\emptyset$, then $q_1=n$, $n_2=0$ and $m=2$. Thus $\mu=\hat\lambda$ is independent of the choice of $\lambda$. We obtain $\lambda\lera s_{e_{k_1}}^{\eps_1}\mu$. Next assume that $\Phi'\neq\emptyset$. Then
\[
\lambda'=\hat\lambda|_{\Phi'}=(s_{e_{k_1}}^{\eps_1}\hat\lambda)|_{\Phi'}=(\hat\lambda_{q_1+1}, \ldots, \hat\lambda_n)
\]
is $\Phi'_{I'}$-regular. There exists $w'\in W'$ so that $\overline{\lambda'}=w'\lambda'$ is a dominant weight of $\Phi'$, where $W'$ is the Weyl group of $\Phi'$. Obviously $\overline{\lambda'}$ is independent of the choice of $\lambda$. By the induction hypothesis, there exists $\Phi'_{I'}$-regular weight $\mu'\in W'\overline{\lambda'}$ such that
\[
\lambda'\lera s_{e_{k_2}}^{\eps_2}s_{e_{k_3}}^{\eps_3}\ldots s_{e_{k_{m-1}}}^{\eps_{m-1}}\mu'.
\]
(relative to $(\Phi'_{I'}, \Phi')$) for all possible $\lambda'$. Denote
\[
\mu=(\hat\lambda_1, \ldots \hat\lambda_{q_1}, \mu'_{{q_1}+1}, \ldots, \mu'_n),
\]
where $\mu'_i=\langle\mu', e_i\rangle$ for $q_1<i\leq n$. Then Lemma \ref{lrlem0} and Proposition \ref{lrprop2} yield $\lambda\lera s_{e_{k_1}}^{\eps_1}\hat\lambda\lera s_{e_{k_1}}^{\eps_1}s_{e_{k_2}}^{\eps_2}\ldots s_{e_{k_{m-1}}}^{\eps_{m-1}}\mu$. The second statement is an easy consequence of the above argument.
\end{proof}

If $\Phi=D_n$, we have to consider $\overline\lambda$ and $s_{e_n}\overline\lambda$ at the same time, where $s_{e_n}\overline\lambda$ is also a dominant weight. The proof (which we leave to the reader) is similar.

\begin{lemma}\label{bkswlem6}
Assume that $\Phi=D_n$. Choose $I, J\subset\Delta$ with $I$ being standard. Let $\overline\lambda$ be a dominant weight with $\Phi_{\overline\lambda}=\Phi_J$. There exists $\Phi_I$-regular weight $\mu$ so that
\[
\lambda\lera s_{e_{k_1}}^{\eps_1}s_{e_{k_2}}^{\eps_2}\ldots s_{e_{k_{m-1}}}^{\eps_{m-1}}s_{e_n}^{\eps_{m}}\mu
\]
for any $\Phi_I$-regular weight $\lambda\in W\overline\lambda\cup Ws_{e_n}\overline\lambda$. Here $q_{s-1}<k_s\leq q_s$ and $\eps_s\in\{0, 1\}$ for $1\leq s<m$. Moreover, $\eps_m\in\{0, 1\}$ for $n_m>0$, while $\eps_m=0$ for $n_m=0$. In particular, if $\eps_s=1$, then $s_{e_{k_s}}\mu$ is $\Phi_I$-regular.
\end{lemma}

\begin{remark}\label{bkswrmk3}
It follows from Lemma \ref{bkswlem5} and Lemma \ref{bkswlem6} that $(\Phi, \Phi_I, \Phi_J)$ contains at most $2^{m-1}$ blocks for $\Phi=B_n, C_n$ and $2^m$ blocks for $\Phi=D_n$.
\end{remark}

\subsection{Blocks of non separable systems and weakly separable systems} We need one more lemma before we can prove the main results of this section.

\begin{lemma}\label{bkswlem7}
Let $\Phi=B_n, C_n$ or $D_n$ and $I$ be standard. Let $\lambda\in\Lambda_I^+$. Assume that $\beta=e_i\pm e_j\in\Psi_\lambda^+$ for $q_{s-1}<i\leq q_s$ and $q_{t-1}<j\leq q_t (i<j)$. Set $a=|\lambda_i|$ and $b=|\lambda_j|$.
\begin{itemize}
  \item [(1)] If $n^\lambda_s(b)$ is maximal, then $s_\beta\lambda$ is $\Phi_I$-singular unless $\lambda_i=a>b=0$, $s=t$ and $n^\lambda_s(a)=1$;
  \item [(2)] If $n^\lambda_t(a)$ is maximal, then $s_\beta\lambda$ is $\Phi_I$-singular unless $a=b>0$, $t=m$ and $n^\lambda_s(a)=1$,
\end{itemize}

\end{lemma}
\begin{proof}
Since $\beta\in\Phi^+\backslash\Phi_I$, we have $s<m$. Denote $\mu=s_\beta\lambda$.

(1) First assume that $s=t$. We must have $\beta=e_i+e_j$ ($e_i-e_j\in\Phi_I$) and $\langle\lambda, \beta\rangle=\lambda_i+\lambda_j>0$. Then $\mu_i=-\lambda_j$, $\mu_j=-\lambda_i$ and $\mu_k=\lambda_k$ when $k\neq i, j$. Keeping in mind that $\lambda\in\Lambda_I^+$ and $e_i-e_j\in\Phi_I$, one has $\lambda_i>\lambda_j$ and thus $a>b\geq 0$. If $b\neq 0$, then $n^\lambda_s(b)=2$ and there exists $q_{s-1}<k\leq q_s$ with $\lambda_k=-\lambda_j$ (thus $k\neq i, j$). So $\langle \mu, e_k-e_i\rangle=\mu_k-\mu_i=\lambda_k-(-\lambda_j)=0$ and $\mu$ is $\Phi_I$-singular. If $b=0$ and $n^\lambda_s(a)=2$, we can also show that $\mu$ is $\Phi_I$-singular in a similar spirit. If $b=0$ and $n^\lambda_s(a)=1$, it can be verified that $\mu$ is $\Phi_I$-regular. It remains to consider the case $s\neq t$. If $a\neq b$, then $n^\mu_s(b)=n^\lambda_s(b)+1$. This is impossible since $n^\lambda_s(b)$ is maximal. If $a=b$, we get $\lambda_i=a>0$ in view of $\langle\lambda, \beta\rangle=\lambda_i\pm\lambda_j>0$. It follows that $\mu=s_\beta\lambda=s_{e_i}s_{e_j}\lambda$. Since $n^\lambda_s(b)$ is maximal and $s<m$, one obtains $n^\lambda_s(b)=2$. There exists $q_{s-1}<k\leq q_s$ so that $\lambda_k=-\lambda_i$. We get $\mu_k=\lambda_k=-\lambda_i=\mu_i$. Hence $\langle\mu, e_i-e_k\rangle=0$ and $\mu$ is $\Phi_I$-singular.

(2) First assume that $s=t$. We can still get $a>b\geq 0$ like (1). With $t=s<m$ and $a>0$, we get $n^\lambda_t(a)=2$. There exists $q_{s-1}<k\leq q_s$ with $\lambda_k=-\lambda_i$. Then $\langle \mu, e_k-e_j\rangle=0$ and $\mu$ is $\Phi_I$-singular. If $s\neq t$ and $\mu$ is  $\Phi_I$-regular, then $a=b>0$ and $\mu=s_{e_i}s_{e_j}\lambda$ by an argument similar to (1). It follows that $n^\lambda_s(a)=n^\lambda_t(b)=1$. This forces $t=m$ since $n^\lambda_t(a)=n^\lambda_t(b)=1$ is maximal. Conversely, with $a=b>0$, $t=m$ and $n^\lambda_s(a)=1$, $\mu=s_\beta\lambda=s_{e_i}s_{e_j}\lambda$ must be $\Phi_I$-regular.
\end{proof}

For $\lambda\in\Lambda_I^+$, let $P(\lambda)$ be a {\it parity function} on $\Lambda_I^+$. More precisely, $P(\lambda)=0$ if $|\{1\leq i\leq n\mid \lambda_i<0\}|$ is even and $P(\lambda)=1$ if $|\{1\leq i\leq n\mid \lambda_i<0\}|$ is odd.

\begin{theorem}\label{bkswthm2}
Let $\Phi=B_n, C_n$ and $I, J\subset\Delta$. Let $\overline\lambda$ be a dominant weight with $\Phi_{\overline\lambda}=\Phi_J$. Assume that $(\Phi, \Phi_I, \Phi_J)$ have nonzero blocks. Then
\begin{itemize}
  \item [(1)] If $(\Phi, \Phi_I, \Phi_J)$ is not separable, it has only one block.
  \item [(2)] If $(\Phi, \Phi_I, \Phi_J)$ is weakly separable, then $(n_m, \overline n_{\overline m})=(0, m-1)$ or $(\overline m-1, 0)$. The system has two blocks.
  \item [(3)] Let $\lambda, \mu\in {}^IW^J\overline\lambda$. Suppose $(\Phi, \Phi_I, \Phi_J)$ is weakly separable. Then $\lambda$ and $\mu$ belong to the same block if and only if $P(\lambda)=P(\mu)$.
\end{itemize}
\end{theorem}
\begin{proof}
If $(\Phi, \Phi_I, \Phi_J)$ is not strongly separable, we show that it contains at most two blocks. First Lemma \ref{bkswlem5} yields a $\Phi_I$-regular weight $\nu\in W\overline\lambda$ so that $\lambda\lera s_{e_{k_1}}^{\eps_1}\ldots s_{e_{k_{m-1}}}^{\eps_{m-1}}\nu$ for all $\lambda\in{}^IW^J\overline\lambda$. Here $\eps_s\in\{0, 1\}$ and $q_{s-1}<k_s\leq q_s$ for $1\leq s<m$. Moreover, if $\eps_s=1$, then $s_{e_{k_s}}^{\eps_s}\nu$ is $\Phi_I$-regular. If $\eps_1=\ldots=\eps_{\overline m-1}=0$ for any $\lambda\in{}^IW^J\overline\lambda$, then $\lambda\lera\nu$ and $(\Phi, \Phi_I, \Phi_J)$ has only one block. Otherwise choose the maximal $1\leq s<m$ so that $\nu_{k_s}\neq0$ and $s_{e_{k_{s}}}\nu$ is $\Phi_I$-regular. Fix $\lambda\in{}^IW^J\overline\lambda$. For any $t<s$ with $\eps_t=1$. Proposition \ref{ssprop3} implies
\[
\lambda\lera s_{e_{k_1}}^{\eps_1}\ldots s_{e_{k_{s-1}}}^{\eps_{s-1}}s_{e_{k_s}}^{\eps_s}\nu\lera s_{e_{k_1}}^{\eps_1}\ldots s_{e_{k_{t-1}}}^{\eps_{t-1}}s_{e_{k_{t+1}}}^{\eps_{t+1}}\ldots s_{e_{k_{s-1}}}^{\eps_{s-1}}s_{e_{k_s}}^{1-\eps_s}\nu.
\]
With this kind of operation, we can eventually get $\lambda\lera s_{e_k}^{\eps}\nu$ for some $\eps\in\{0, 1\}$ and $1\leq k=k_s\leq q_{m-1}$. This shows that $(\Phi, \Phi_I, \Phi_J)$ has at most two blocks.

If $m=1$ or $\overline m=1$, then $(\Phi, \Phi_I, \Phi_J)$ is not separable and contains only one block. Now suppose $m, \overline m>1$. If $(n_m, \overline n_{\overline m})\neq(0, m-1), (\overline m-1, 0)$, Proposition \ref{ssprop1} and the previous reasoning yield $\lambda\lera s_{e_k}^\eps\nu\lera\nu$. Thus the system has only one block.

Now assume that $(n_m, \overline n_{\overline m})=(0, m-1)$. If $\lambda, \mu\in{}^IW^J\overline\lambda$ are linked by $\beta\in\Psi_\lambda^+$, we show that $P(\lambda)=P(\mu)$. In fact, we have $\langle\lambda, \beta\rangle>0$ and $\lambda=ws_\beta\mu$ for $w\in W_I$. If $\beta=e_i$ or $2e_i$ for some $q_{s-1}<i\leq q_s$ and $1\leq s<m$, then $a=\lambda_i=\langle\lambda, e_i\rangle>0$. With $n_m=0$ and $\overline n_{\overline m}=m-1$, one obtains $n^\lambda_s(a)=n^\lambda_s(0)=1$ and $n^\lambda_m(a)=n^\lambda_m(0)=0$. It follows from Proposition \ref{lrprop2} that $\beta$ is not a linked root from $\lambda$, a contradiction. Next assume that $\beta=e_i\pm e_j$ for some $q_{s-1}<i\leq q_s$, $q_{t-1}<j\leq q_t$ and $1\leq s, t<m$. If $\lambda_i=0$, then $n_t^\lambda(0)=1$ is maximal, in view of $n_{\overline m}=m-1$. Lemma \ref{bkswlem7} yields $|\lambda_i|=|\lambda_j|>0$, a contradiction. Similarly, if $\lambda_j=0$, Lemma \ref{bkswlem7} gives $s=t$ and $n_s^\lambda(a)=1$, where $a=|\lambda_i|$. With $n_m=0$, Proposition \ref{lrprop2} implies that $\beta$ is not a linked root, another contradiction. Now we obtain $\lambda_i, \lambda_j\neq 0$. So $P(\lambda)=P(s_\beta\lambda)$ in this case. On the other hand, $n_m=0$ implies any $w\in W_I$ is a permutation, that is, $P(\mu)=P(ws_\beta\lambda)=P(s_\beta\lambda)=P(\lambda)$. As a consequence, $P(\lambda)=P(\mu)$ when $\lambda\lera\mu$. In view of Lemma \ref{splem8} (choose $(S, \overline S)=(\{1, \ldots, m-1\}, \{\overline m\})$), there exists $\lambda\in{}^IW^J\overline\lambda$ and $1\leq i\leq q_{m-1}$ so that $s_{e_i}\lambda$ ($\neq\lambda$) is $\Phi_I$-regular (otherwise the system is strongly separable relative to $(\{s_0\}, \{1, \ldots, \overline m\})$ for some $1\leq s_0<m$, a contradiction). Choose $w\in W_I$ so that $\mu=ws_{e_i}\lambda\in\Lambda_I^+$. So $P(\mu)\neq P(\lambda)$ in this case. This shows that $(\Phi, \Phi_I, \Phi_J)$ possess at least two blocks.

Now consider $(n_m, \overline n_{\overline m})=(\overline m-1, 0)$, then $n^\lambda_s(0)=0$ for any $1\leq s\leq m$ and $\lambda\in{}^IW^J\overline\lambda$. If $\lambda, \mu$ are linked by $\beta$, we can also show that $\beta\neq e_i$ by Proposition \ref{lrprop2}, while $P(s_\beta\lambda)=P(\lambda)$ always holds for $\beta=e_i\pm e_j$ ($\lambda_i, \lambda_j\neq 0$). The following reasoning is similar (with Lemma \ref{splem8} replaced by Lemma \ref{splem9}).
\end{proof}

\begin{theorem}\label{bkswthm3}
Let $\Phi=D_n$ and $I, J\subset\Delta$. Let $\overline\lambda$ be a dominant weight with $\Phi_{\overline\lambda}=\Phi_J$. Assume that $(\Phi, \Phi_I, \Phi_J)$ have nonzero blocks. Then
\begin{itemize}
  \item [(1)] If $(\Phi, \Phi_I, \Phi_J)$ is not separable, it has only one block.
  \item [(2)] If $(\Phi, \Phi_I, \Phi_J)$ is weakly separable, it has two blocks when $(n_m, \overline n_{\overline m})=(\overline m, m)$, otherwise it has only one block.
  \item [(3)] Let $\lambda, \mu\in {}^IW^J\overline\lambda$. If $(\Phi, \Phi_I, \Phi_J)$ is weakly separable and has two blocks, then $\lambda$ and $\mu$ belong to the same block if and only if $P(\lambda)=P(\mu)$.
\end{itemize}
\end{theorem}
\begin{proof}
If $I$ is not standard, we can consider the isomorphic system $(\Phi, \Phi_{\vf(I)}, \Phi_{\vf(J)})$. Assume that $I$ is standard. First we show that $(\Phi, \Phi_I, \Phi_J)$ has at most two blocks. It follows from Lemma \ref{bkswlem6} that there exists a $\Phi_I$-regular weight $\nu\in W\overline\lambda\cup W s_{e_n}\overline\lambda$ so that $\lambda\lera s_{e_{k_1}}^{\eps_1}\ldots s_{e_{k_{m-1}}}^{\eps_{m-1}}s_{e_n}^{\eps_m}\nu$. In particular, $\eps_m=0$ when $n_m=0$. If $\eps_s=1$, then $\nu_{k_s}\neq0$ and $s_{e_{k_s}}^{\eps_s}\nu$ is $\Phi_I$-regular. Choose the maximal $1\leq s<m$ so that $\nu_{k_s}\neq0$ and $s_{e_{k_{s}}}\nu$ is $\Phi_I$-regular. Fix $\lambda\in{}^IW^J\overline\lambda$. For any $t<s$ with $\eps_t=1$. If $(\Phi, \Phi_I, \Phi_J)$ is not strongly separable, Proposition \ref{ssprop3} implies
\[
\lambda\lera s_{e_{k_1}}^{\eps_1}\ldots s_{e_{k_{s-1}}}^{\eps_{s-1}}s_{e_{k_s}}^{\eps_s}s_{e_n}^{\eps_m}\nu\lera s_{e_{k_1}}^{\eps_1}\ldots s_{e_{k_{t-1}}}^{\eps_{t-1}}s_{e_{k_{t+1}}}^{\eps_{t+1}}\ldots s_{e_{k_{s-1}}}^{\eps_{s-1}}s_{e_{k_s}}^{1-\eps_s}s_{e_n}^{\eps_m}\nu.
\]
We can eventually get $\lambda\lera s_{e_k}^{\eps}s_{e_n}^{\eps_m}\nu$ for some $\eps\in\{0, 1\}$ and $1\leq k=k_s\leq q_{m-1}$. If $m=1$ or $\overline m=1$, the result is trivial. It suffices to consider the case $m, \overline m>1$.

First assume that $(n_m, \overline n_{\overline m})\neq(\overline m, m)$. If $\overline n_{\overline m}>0$, then Proposition \ref{ssprop2} yields $\lambda\lera s_{e_k}^\eps s_{e_{n}}^{\eps_{m}}\nu\lera\nu$. Thus $(\Phi, \Phi_I, \Phi_J)$ possess only one block. If $n_m>0$, then $(\Phi, \Phi_J, \Phi_I)$ contains only one block by a dual argument. In view of Lemma \ref{sslem2}, $(\Phi, \Phi_I, \Phi_J)$ has one block. It remains to consider the case $(n_m, \overline n_{\overline m})=(0, 0)$. In this case, $\eps_m=0$ and $n^\lambda_s(0)=0$ for all $1\leq s\leq m$. It follows that $P(\mu)=P(\lambda)$ whenever $\mu\in W\lambda$. Therefore $P(\lambda)=P(\overline\lambda)=P(s_{e_k}^\eps\nu)$. If $\nu\in W\overline\lambda$ (resp. $\nu\in W s_{e_n}\overline\lambda$), then $\eps=0$ (resp. $\eps=1$) for all possible $\lambda$. Hence $(\Phi, \Phi_I, \Phi_J)$ has one block.

Next assume that $(n_m, \overline n_{\overline m})=(\overline m, m)$. By the proof of Lemma \ref{bkswlem5} and Lemma \ref{bkswlem6}, we can assume that
\[
\nu_{q_{m-1}+1}>\ldots>\nu_{n-1}>|\nu_n|\geq0.
\]
Then $n^\nu_m(0)=1$ implies $\nu_n=0$. So $s_{e_{n}}^{\eps_{m}}\nu=\nu$, which means $(\Phi, \Phi_I, \Phi_J)$ has at most two blocks. If $\lambda$ and $\mu$ are linked by $\beta\in\Phi^+\backslash\Phi_I$, we show that $P(\lambda)=P(\mu)$. In fact, we have $\mu=ws_\beta\lambda$ for $w\in W_I$. Assume that $\beta=e_i\pm e_j$ for some $q_{s-1}<i\leq q_s$, $q_{t-1}<j\leq q_t$ and $1\leq s\leq t\leq m$. Then $w\in W_{I_s\cup I_t}$. Moreover, $\lambda_i, \lambda_j\neq0$ by a proof similar to the case when $\Phi=B_n$. If $t=m$, then $n_t^\lambda(a)=1$ is maximal for $a=|\lambda_i|$. Lemma \ref{bkswlem7} yields $|\lambda_j|=a$. Proposition \ref{lrprop2} implies $\beta$ is not a linked root. This forces $s\leq t<m$. So we must have $P(s_\beta\lambda)=P(\lambda)$. On the other hand, $w\in W_{I_s\cup I_t}$ is a permutation for $s\leq t<m$, that is, $P(\mu)=P(ws_\beta\lambda)=P(s_\beta\lambda)=P(\lambda)$. In general, $P(\lambda)=P(\mu)$ when $\lambda\lera\mu$. By Lemma \ref{splem8} (choose $(S, \overline S)=(\{1, \ldots, m\}, \{\overline m\})$), there is $\lambda\in{}^IW^J\overline\lambda$ and $1\leq s<m$ so that $q_{s-1}< i\leq q_{s}$ and $s_{e_i}\lambda=s_{e_i+e_n}s_{e_i-e_n}\lambda\in W\overline\lambda$ is $\Phi_I$-regular (otherwise $(S, \overline S)\sim (\{m\}, \{1, \ldots, \overline m\})$ and the system is strongly separable in view of Lemma \ref{splem7}). Choose $w\in W_{I_s}$ so that $\mu=ws_{e_i}\lambda\in\Lambda_I^+$. So $P(\mu)\neq P(\lambda)$ in this case. This shows that $(\Phi, \Phi_I, \Phi_J)$ has at least two blocks.

\end{proof}

%
%
\section{Blocks of strongly separable systems}
%
%
In this section, we will determine the blocks of strongly separable systems. We say a separable pair $(S, \overline S)\in\caD$ is \emph{trivial} if it satisfies the following conditions:
\begin{itemize}
\item [(\rmnum{1})] When $\Phi=B_n, C_n$, $(S, \overline S)\sim (\{s_0\}, \{1,\ldots, \overline m\})$ for some $1\leq s_0<m$ or $(S, \overline S)\sim(\{1,\ldots, m\}, \{t_0\})$ for some $1\leq t_0<\overline m$;
\item [(\rmnum{2})] When $\Phi=D_n$, $(S, \overline S)\sim (\{s_0\}, \{1,\ldots, \overline m-1\})$ for some $1\leq s_0<m$ or $(S, \overline S)\sim(\{1,\ldots, m-1\}, \{t_0\})$ for some $1\leq t_0<\overline m$.
\end{itemize}

\subsection{strongly separable systems for $\Phi=A_{n-1}$} In this subsection, we show that the strongly separable systems for $\Phi=A_{n-1}$ can be written as product of two subsystems. Let $\overline\lambda$ be a dominant weight with $\Phi_{\overline\lambda}=\Phi_J$. There exists strongly separable pair $(S, \overline S)$ so that $n_s^\lambda(a_t)=1$ for $s\in S$, $t\in\overline S$ and $n_s^\lambda(a_t)=0$ for $s\not\in S$, $t\not\in\overline S$. Here $\lambda$ is any $\Phi_I$-regular weight contained in $W\overline\lambda$. Denote $T=\{1, \ldots, m\}\backslash S$ and $\overline\caA=\{a_t\in\caA\mid t\in \overline S\}$. Obviously $\overline\caA\subset \{\lambda_{q_{s-1}+1}, \ldots, \lambda_{q_s}\}$ when $s\in S$. Set $p=|\overline\caA|$. For $\lambda\in{}^IW^J\overline\lambda$, we can choose $w_s\in W_{I_s}$ so that $\nu=(\prod_{s\in S}w_s)\lambda$ satisfies
\begin{equation}\label{bksst1eq1}
\begin{aligned}
&\{\nu_{q_{s-1}+1}, \ldots, \nu_{q_{s-1}+p}\}=\overline\caA,\\
\nu_{q_{s-1}+1}>&\ldots>\nu_{q_{s-1}+p},\quad \nu_{q_{s-1}+p+1}>\ldots>\nu_{q_s}
\end{aligned}
\end{equation}
when $s\in S$. Denote $w_\lambda=\prod_{r\in S}w_r$. Then $w_\lambda$ is unique, so is $\nu$. Put
\begin{equation}\label{bksst1eq3}
\Phi_1=(\sum_{\substack{q_{s-1}+p<i\leq q_s,\\
s\in S}}\bbQ e_i)\cap\Phi\quad \mbox{and}\quad \Phi_2=(\sum_{\substack{q_{s-1}<i\leq q_s,\\
s\in T}}\bbQ e_i)\cap\Phi.
\end{equation}
Obviously $\Phi_1, \Phi_2$ are proper subsystems of $\Phi$ (the empty set is also viewed as a subsystem of $\Phi$, see Remark \ref{bkssrmk1}).

Let $W_i$ be the Weyl group of $\Phi_i$ for $i=1, 2$. Let $\Delta_i$ be the simple system corresponding to $\Phi_i^+=\Phi_i\cap\Phi^+$. Denote $I^i=\Phi_I\cap\Delta_i$. Then $I^1=I\cap\Phi_1$ and $I^2=I\cap\Phi_2$. Obviously, $\Phi_{I^i}$ is a subsystem of $\Phi_i$ generated by $I^i$. Set $\lambda^i=\nu|_{\Phi_i}$. Evidently, $\lambda^i\in\Lambda^+(\Phi_{I^i}, \Phi_i)$. Denote by $n^{\lambda^i}(a)$ the number of entries of $\lambda^i$ with absolute value $a$. It follows from the construction that
\begin{equation}\label{bksst1eq4}
n^{\lambda^1}(a_t)=\left\{\begin{aligned}
&0,\quad\ \mbox{if}\ t\in\overline S;\\
&\overline n_t,\quad \mbox{if}\ t\not\in\overline S.
\end{aligned}\
\qquad
\right.
\quad
n^{\lambda^2}(a_t)=\left\{\begin{aligned}
&\overline n_t-|S|,\quad\ \mbox{if}\ t\in\overline S;\\
&0,\qquad\qquad\ \mbox{if}\ t\not\in\overline S.
\end{aligned}\
\qquad
\right.
\end{equation}
Let $\overline{\lambda^i}\in W_i\lambda^i$ be a dominant weight of $\Phi_i^+=\Phi_i\cap\Phi^+$. In view of (\ref{bksst1eq4}), $\overline{\lambda^i}$ is uniquely determined by $\overline\lambda$ and $(S, \overline S)$. Put $J^i=\{\alpha\in \Delta_i\mid \langle\overline{\lambda^i}, \alpha\rangle=0\}$. It can be verified that $\lambda\ra(\lambda^1, \lambda^2)$ gives a bijection between ${}^IW^J\overline\lambda$ and $({}^{I^1}W_1^{J^1}\overline{\lambda^1}, {}^{I^2}W_2^{J^2}\overline{\lambda^2})$. This actually gives the first part of the following theorem.

\begin{theorem}\label{bkssthm1}
Let $\Phi=A_{n-1}$. If $(\Phi, \Phi_I, \Phi_J)$ is strongly separable, then
\begin{itemize}
\item [(1)] There exist proper subsystem $\Phi_i$ of $\Phi$ for $i=1, 2$ such that
\[
{}^IW^J\simeq{}^{I^1}W_1^{J^1}\times{}^{I^2}W_2^{J^2},
\]
where $W_i$ is the Weyl group of $\Phi_i$ and $I^i, J^i$ are subsets of simple system $\Delta_i$ corresponding to $\Phi_i^+=\Phi_i\cap\Phi^+$.

\item [(2)] Let $\overline\lambda$ be a dominant weight with $\Phi_{\overline\lambda}=\Phi_J$. There exist dominant weight $\overline{\lambda^i}$ of $\Phi_i$ for $i=1, 2$ such that each $\nu\in{}^IW^J\overline\lambda$ correspond $\nu^i\in {}^{I^i}W_i^{J^i}\overline{\lambda^i}$. Moreover, $\lambda\lera\mu$ if and only if $\lambda^i\lera\mu^i$ relative to $(\Phi_{I^i}, \Phi_i)$.
\end{itemize}
\end{theorem}

\begin{proof}
The root system $\Phi_i$ is either empty or of type $A$. Thus (2) follows from Theorem \ref{bkswthm2}.
\end{proof}

\begin{remark}\label{bkssrmk1}
When $\Phi_i$ is an empty root system for some $i\in\{1, 2\}$, then $I^i$, $J^i$ are also empty. The Weyl group $W_i=1={}^{I^i}W_i^{J^i}$. The weights $\overline{\lambda^i}=\nu^i=0$. For convenience, we write $\overline{\lambda^i}\lera\nu^i$ for any $\nu\in {}^IW^J\overline\lambda$.
\end{remark}

Inspired by the above results for type $A$, we will prove similar results for the other types.

\subsection{strongly separable systems for $\Phi=B_n$, $C_n$} We only need to consider the case $\Phi=B_n$. The case $\Phi=C_n$ is similar. Let $\overline\lambda$ be a dominant weight with $\Phi_{\overline\lambda}=\Phi_J$. There exists strongly separable pair $(S, \overline S)$ so that $n_s^\lambda(a_t)$ is maximal for $s\in S$, $t\in\overline S$ and $n_s^\lambda(a_t)=0$ for $s\not\in S$, $t\not\in\overline S$, where $\lambda\in W\overline\lambda$ is a $\Phi_I$-regular weight. Denote $\overline\caA=\{a_t\mid t\in\overline S\}$ and $\overline\caA^*=\{\pm a \mid 0<a\in \overline\caA\}$. It is evident that $\overline\caA^*\subset \{\lambda_{q_{s-1}+1}, \ldots, \lambda_{q_s}\}$ when $s\in S\backslash\{m\}$. Moreover, $\overline\caA\backslash\{0\}\subset\{\lambda_{q_{m-1}+1}, \ldots, \lambda_n\}$ when $m\in S$. Set $p=|\overline\caA\backslash\{0\}|$ and $h=q_{m-1}+p$. Then $|\overline\caA^*|=2p$. Note that $n-q_{m-1}\leq|\overline\caA\backslash\{0\}|=p$ when $m\not\in S$ (since $n_m^\lambda(0)=0$). For $\lambda\in{}^IW^J\overline\lambda$, we can choose permutations $w_s\in W_{I_s}$ so that $\nu=(\prod_{s\in S}w_s)\lambda$ satisfies
\begin{equation}\label{bksst3eq1}
\begin{aligned}
&\{\nu_{q_{s-1}+1}, \ldots, \nu_{q_{s-1}+2p}\}=\overline\caA^*,\\
\nu_{q_{s-1}+1}>&\ldots>\nu_{q_{s-1}+2p},\quad \nu_{q_{s-1}+2p+1}>\ldots>\nu_{q_s}
\end{aligned}
\end{equation}
when $s\in S\backslash\{m\}$ and
\begin{equation}\label{bksst3eq2}
\begin{aligned}
&\{\nu_{q_{m-1}+1}, \ldots, \nu_h\}=\overline\caA\backslash\{0\},\\
\nu_{q_{m-1}+1}>&\ldots>\nu_h,\quad\nu_h>\ldots>\nu_n>0
\end{aligned}
\end{equation}
when $m\in S$. Denote $w_\lambda=\prod_{r\in S}w_r$. Then $w_\lambda$ is unique, so is $\nu$. Set $T=\{1, 2, \ldots, m\}\backslash S$. Put
\begin{equation}\label{bksst3eq3}
\Phi_1=(\sum_{\substack{q_{s-1}+2p<i\leq q_s,\\
s\in S\backslash\{m\}}}\bbQ e_i+\sum_{h<j\leq n}\bbQ e_{j})\cap\Phi
\end{equation}
and
\begin{equation}\label{bksst3eq4}
\Phi_2=(\sum_{\substack{q_{s-1}<i\leq q_s,\\
s\in T\backslash\{m\}}}\bbQ e_i+\sum_{q_{m-1}<j\leq h}\bbQ e_{j})\cap\Phi.
\end{equation}
Thus $\Phi_1, \Phi_2$ are proper subsystems of $\Phi$ (the empty set is also viewed as a subsystem of $\Phi$, see Remark \ref{bkssrmk1}). Moreover, $\Phi_1\perp\Phi_2$.

Let $W_i$ be the Weyl group of $\Phi_i$ for $i=1, 2$. Let $\Delta_i$ be the simple system corresponding to $\Phi_i^+=\Phi_i\cap\Phi^+$. Denote $I^i=\Phi_I\cap\Delta_i$. In particular, $I^1=I\cap\Phi_1$. If $q_{m-1}<h<n$, $I^2=(I\cap\Phi_2)\cup\{e_h\}$, otherwise $I^2=I\cap\Phi_2$. Let $\Phi_{I^i}$ is a subsystem of $\Phi_i$ generated by $I^i$. Set $\lambda^i=\nu|_{\Phi_i}$. Evidently, $\lambda^i\in\Lambda^+(\Phi_{I^i}, \Phi_i)$.
Let $\overline{\lambda^i}\in W_i\lambda^i$ be a dominant weight of $\Phi_i^+=\Phi_i\cap\Phi^+$. Similar to the case of type $A$, $\overline{\lambda^i}$ is uniquely determined by $\overline\lambda$ and $(S, \overline S)$. Put $J^i=\{\alpha\in \Delta_i\mid \langle\overline{\lambda^i}, \alpha\rangle=0\}$. Then $\lambda\ra(\lambda^1, \lambda^2)$ gives a bijection between ${}^IW^J\overline\lambda$ and $({}^{I^1}W_1^{J^1}\overline{\lambda^1}, {}^{I^2}W_2^{J^2}\overline{\lambda^2})$. We obtain the first part of the following theorem.

\begin{theorem}\label{bkssthm2}
Let $\Phi=B_n$ or $C_n$. If $(\Phi, \Phi_I, \Phi_J)$ is strongly separable, then
\begin{itemize}
\item [(1)] There exist proper subsystem $\Phi_i$ of $\Phi$ for $i=1, 2$ such that
\[
{}^IW^J\simeq{}^{I^1}W_1^{J^1}\times{}^{I^2}W_2^{J^2},
\]
where $W_i$ is the Weyl group of $\Phi_i$ and $I^i, J^i$ are subsets of simple system $\Delta_i$ corresponding to $\Phi_i^+=\Phi_i\cap\Phi^+$.

\item [(2)] Let $\overline\lambda$ be a dominant weight with $\Phi_{\overline\lambda}=\Phi_J$. There exist dominant weight $\overline{\lambda^i}$ of $\Phi_i$ for $i=1, 2$ such that each $\nu\in{}^IW^J\overline\lambda$ correspond $\nu^i\in {}^{I^i}W_i^{J^i}\overline{\lambda^i}$. Moreover, $\lambda\lera\mu$ if and only if $\lambda^i\lera\mu^i$ relative to $(\Phi_{I^i}, \Phi_i)$.
\end{itemize}
\end{theorem}

\begin{proof}
Suppose $\lambda^i\lera\mu^i$ for $i=1, 2$. Let $\eta\in{}^IW^J\overline\lambda$ be the weight corresponding to $(\lambda^1, \mu^2)$. Then $(w_\eta\eta)_j=(w_\lambda\lambda)_j$ for $q_{s-1}<j\leq q_{s}$ with $s\in S\backslash\{m\}$ and $h<j\leq n$. It follows from Lemma \ref{lrlem0} and Proposition \ref{lrprop2} that $w_\eta\eta \lera w_\lambda\lambda$. Similarly we can get $w_\eta\eta\lera w_\mu\mu$. Therefore $\lambda\lera\mu$.

For the other direction, it suffices to consider the case when $\lambda, \mu\in{}^IW^J\overline\lambda\subset\Lambda_I^+$ are linked by $\beta\in\Phi^+\backslash\Phi_I$, that is, $\mu=ws_\beta\lambda$ for $w\in W_I$ and $\langle\lambda, \beta\rangle>0$. Put $\zeta=w_\mu\mu$. First assume that $\beta=e_i\pm e_j$ for $1\leq i<j\leq n$. There exist $1\leq s\leq t\leq m$ such that $q_{s-1}<i\leq q_s$ and $q_{t-1}<j\leq q_t$. Set $a=|\lambda_i|$ and $b=|\lambda_j|$. With $\beta\not\in\Phi_I$, one obtains $s<m$. Since $(s_\beta\lambda)_u=\lambda_u$ when $u\neq i, j$, one has $w\in W_{I_s\cup I_t}=W_{I_s}W_{I_t}$. There exist $w'\in W_{I_s}$ and $w''\in W_{I_t}$ with $w=w'w''$. Note that $w_\lambda$ is a permutation in $W_I$ and $\nu=w_\lambda\lambda$. We can assume that $w_\lambda e_i=e_k$ for $q_{s-1}<k\leq q_s$ and $w_\lambda e_j=e_l$ for $q_{t-1}<l\leq q_t$. Then $\gamma=w_\lambda\beta=e_k\pm e_l$.

If $s\in S$ and $t\in T$, then $s\neq t$. With $n_t^\lambda(b)\geq1$, we obtain $b\in\overline\caA$. So $n^\lambda_s(b)$ is maximal. Thus $s_\beta\lambda$ is $\Phi_I$-singular by Lemma \ref{bkswlem7}, a contradiction.

If $s\in T$ and $t\in S$, then $a\in \overline\caA$ and $n^\lambda_t(a)$ is maximal. With Lemma \ref{bkswlem7}, one must have $a=b>0$, $t=m$ and $s_\beta\lambda=s_{e_i}s_{e_j}\lambda$. This forces $\beta=e_i+e_j$ in view of (\ref{lreq2}). Then $\mu=(w' s_{e_i})(w'' s_{e_j})\lambda\in\Lambda_I^+$. We obtain $w'' s_{e_j}\lambda\in\Lambda_{I_m}^+$ and $w' s_{e_i}\lambda\in\Lambda_{I_s}^+$. Evidently $w''=s_{e_j}$. Thus
\[
w_\lambda\mu=w_\lambda ws_\beta\lambda=w_\lambda w'w''s_{e_i}s_{e_j}\lambda=w's_{e_i}w_\lambda\lambda=w's_{e_i}\nu.
\]
Since $s\not\in S\cup\{m\}$, (\ref{bksst3eq1}) and (\ref{bksst3eq2}) are satisfied with $\nu$ replaced by $w_\lambda\mu$. The uniqueness of $w_\mu$ forces $w_\mu=w_\lambda$ and $\zeta=w_\mu\mu=w_\lambda\mu$. Moreover, $s\not\in S\cup\{m\}$ also implies $(I_s\cup\{e_i\}) \perp\Phi_1$. By Lemma \ref{reslem1}, $\mu^1=\zeta|_{\Phi_1}=(w's_{e_i}\nu)|_{\Phi_1}=\nu|_{\Phi_1}=\lambda^1$. In view of $I_s\cup\{e_i\}\subset \Phi_2$, we also get $\mu^2=\zeta|_{\Phi_2}=w's_{e_i}(\nu|_{\Phi_2})=w's_{e_i}\lambda^2$. Note that
\[
|\nu_l|=|\langle \nu, w_\lambda e_j\rangle|=|\langle \lambda, e_j\rangle|=b=a\in \overline\caA\backslash\{0\}.
\]
Thus $l\leq h$ (by \ref{bksst3eq2}) and $\gamma=w_\lambda\beta=e_i+e_l\in \Phi_2$ ($w_\lambda e_i=e_i$ because $s\not\in S$). Notice that $w's_{e_l}s_{e_i+e_l}\nu=w's_{e_l}s_{e_i}s_{e_l}\nu=\zeta$. We get $\mu^2=w's_{e_l}s_{e_i+e_l}\lambda^2$ in view of $\{e_l, e_i+e_l\}\subset\Phi_2$. Since $\beta$ is a linked root from $\lambda$, $\gamma=e_i+e_l$ is a linked root from $\nu=w_\lambda\lambda$. Hence $e_i+e_l$ is a linked roots from $\lambda^2$ to $\mu^2$ by Proposition \ref{lrprop2}.

If $s, t\in S$, we claim that $a, b\not\in \overline\caA\backslash\{0\}$. In fact, if $a\in \overline\caA\backslash\{0\}$, then $n^\lambda_s(a)=n^\lambda_t(a)=2$. By Lemma \ref{bkswlem7}, we must have $n^\lambda_s(a)=1$, a contradiction. If $b\in \overline\caA\backslash\{0\}$, Lemma \ref{bkswlem7} implies $b=0$, also a contradiction. Now the claim is proved. One gets $\nu_k, \nu_l\not\in \overline\caA^*$ in view of $|\nu_k|=a$ and $|\nu_l|=b$ that. Therefore $\gamma=e_k\pm e_l\in\Phi_1$. Note that $(s_{\gamma}\nu)_u=\nu_u$ unless $u=k, l$. It follows from
\[
\zeta=w_\mu w s_\beta\lambda=w_\mu w w_\lambda^{-1}s_{\gamma}w_\lambda\lambda=(w_\mu w w_\lambda^{-1})s_{\gamma}\nu
\]
that $w_1=w_\mu w w_\lambda^{-1}\in W_1\cap \prod_{s\in S}W_{I_s}=W_{I^1}$. Thus $\mu^1=w_1 s_\gamma\lambda^1$ and $\mu^2=\lambda^2$. In view of Proposition \ref{lrprop2}, $\gamma$ is a linked root from $\mu^1$.

If $s, t\in T$, then $w\in W_{I^2}$. Moreover, $(I_s\cup I_t)\perp\Phi_1$ and $\beta\perp\Phi_1$. Thus $w_\lambda\mu=ws_\beta w_\lambda\lambda=ws_\beta\nu$ will satisfy (\ref{bksst3eq1}) and (\ref{bksst3eq2}). We get $w_\mu=w_\lambda$. Hence $\mu^1=(ws_\beta\nu)|_{\Phi_1}=\nu|_{\Phi_1}=\lambda^1$ and $\mu^2=ws_\beta\lambda^2$. Also, $\beta$ is a linked root from $\lambda^2$.

Next assume that $\beta=e_i$ for $q_{s-1}<i\leq q_s$ and $s< m$. Then $a=\lambda_i=\langle\lambda, \beta\rangle>0$. With $\mu_u=\lambda_u$ when $u\neq i$, one obtains $w\in W_{I_s}$. Since $s_{e_i}\lambda$ is $\Phi_I$-regular, we must have $n^\lambda_s(a)=1$, which is not maximal for $s<m$. If $s\in S$, then $a\not\in \overline\caA$ (since $n^\lambda_s(a)=1<2$). Similar argument shows that $\mu^1=(w_\mu w w_\lambda^{-1})s_\gamma\lambda^1$ and $\mu^2=\lambda^2$, where $\gamma=w_\lambda\beta=e_k$ is a linked root from $\lambda^1$ by Proposition \ref{lrprop2}. In a similar spirit, if $s\in T$, we get $\mu^1=\lambda^1$, $\mu^2=ws_{e_i}\lambda^2$ and $e_i$ is a linked root from $\lambda^2$.
\end{proof}

The following lemma can be used to get the number of blocks.

\begin{lemma}\label{bksslem1}
Let $\Phi=B_n$ or $C_n$. Let $(S, \overline S)\in \caD$ be a strongly separable pair. Set
\[
\begin{aligned}
\caD^1=&\{(S', \overline{S'})\in \caD\mid (S', \overline{S'})\leq(S, \overline S),\ \overline{S'}\neq \overline S\backslash\{\overline m\}\};\\
\caD^2=&\{(S', \overline{S'})\in \caD\mid (S', \overline{S'})\geq(S, \overline S),\ S'\neq S\backslash\{m\}\}.
\end{aligned}
\]
Fix $i\in\{1, 2\}$. If $I^i, J^i\neq\Delta_i$, there is bijection $\psi_i$ between $\caD(\lambda^i, \Phi_I\cap\Phi_i, \Phi_i)$ and $\caD^i$ satisfying the following conditions:
\begin{itemize}
  \item [(1)] $\psi_i$ sends equivalent pairs to equivalent pairs, so does $\psi_i^{-1}$.
  \item [(2)] If $(S'_i, \overline{S'_i})\in\caD(\lambda^i, \Phi_I\cap\Phi_i, \Phi_i)$ is strongly separable, then $\psi_i(S'_i, \overline{S'_i})\in\caD^i$ is strongly separable.
  \item [(3)] $(S'_i, \overline{S'_i})\in \caD(\lambda^i, \Phi_I\cap\Phi_i, \Phi_i)$ is trivial if and only if $\psi_i(S'_i, \overline{S'_i})\in\caD^i$ is trivial or equivalent to $(S, \overline S)\not\in \caD^i$.
\end{itemize}
Otherwise $(S', \overline{S'})\sim (S, \overline S)\in \caD^{3-i}$ for any nontrivial $(S', \overline{S'})\in \caD^i$.
\end{lemma}

\begin{proof}
Denote
\[
\begin{aligned}
H_1=&\{q_{s-1}+2p<i\leq q_s\mid m>s\in S\}\cup\{h<i\leq n\};\\
H_2=&\{q_{s-1}<i\leq q_s\mid m>s\in T\}\cup\{q_{m-1}<i\leq h\}.
\end{aligned}
\]
The entries of $\lambda^i$ are indexed by $H_i$ for $i=1, 2$. Set $\overline T=\{1, \ldots, \overline m\}\backslash\overline S$. Denote
\begin{equation}\label{bkssl1eq1}
\begin{aligned}
S_i=\{1\leq s<m\mid n_s^{\lambda^i}(a_t)>0\ \mbox{for some}\ 1\leq t<\tilde m\}\cup\{m\};\\
\overline S_i=\{1\leq t<\overline m\mid n_s^{\lambda^i}(a_t)>0\ \mbox{for some}\ 1\leq s<\tilde m\}\cup\{\overline m\}.
\end{aligned}
\end{equation}
If $(S'_i, \overline{S'_i})\in \caD(\lambda^i, \Phi_I\cap\Phi_i, \Phi_i)$, it is evident that $S'_i\subset S_i$ and $\overline{S'_i}\subset\overline S_i$. Since $(S, \overline S)$ is separable, we obtain $S_1=\{s\in S\mid n_s>2p\}\cup\{m\}$ for $\overline m\not\in\overline S$ and $S_1=S\cup\{m\}$ for $\overline m\in\overline S$, while $\overline S_1=\overline T\cup\{\overline m\}$. Similarly, $\overline S_2=\overline S\cup\{\overline m\}$ for $m\in S$ and $\overline S_2=\{t\in\overline S\mid \overline n_t>2|S|\}\cup\{\overline m\}$ for $m\not\in S$, while $S_2=T\cup\{m\}$.

It suffices to consider the case $i=1$, while the proof for $i=2$ is similar. First assume that $I^1, J^1\neq\Delta_1$. If $(S'_1, \overline{S'_1})\in \caD(\lambda^1, \Phi_I\cap\Phi_1, \Phi_1)$, we get $S'_1\subset S\cup\{m\}$ and $\overline{S'_1}\subset\overline T\cup\{\overline m\}$. If $S'_1\not\subset S$, one must have $S'_1=S\cup\{m\}$ and $m\not\in S$. Thus $\overline m\in\overline S$ and $\overline m\not\in\overline{S'_1}$. We can choose $t\in\overline{S'_1}\subset\overline T$ so that $n_m^{\lambda}(a_t)=n_m^{\lambda^1}(a_t)=1$ is maximal. On the other hand, $n_m^\lambda(a_t)=0$ since $m\not\in S$ and $t\not\in\overline S$. This forces $S'_1\subset S$. It is routine to check that $\psi_1: (S'_1, \overline{S'_1})\ra(S'_1, \overline{S'_1}\cup(\overline S\backslash\{\overline m\}))$ gives an bijective map from $D(\lambda^1, \Phi_I\cap\Phi_1, \Phi_1)$ to $\caD^1$ (with inverse map $(S', \overline{S'})\ra (S', \overline{S'}\backslash(\overline S\backslash\{\overline m\}))$). Then (1) and (2) are obvious. If $(S'_1, \overline{S'_1})$ and $(\{s_0\}, \overline S_1)$ are equivalent separable pairs contained in $\caD(\lambda^i, \Phi_I\cap\Phi_i, \Phi_i)$ for some $m>s_0\in S_1$, then
\[
\psi_1(S'_1, \overline{S'_1})\sim\psi_1(\{s_0\}, \overline S_1)\sim(\{s_0\}, (\overline T\cup\{\overline m\})\cup(\overline S\backslash\{\overline m\})=(\{s_0\}, \{1, \ldots, \overline m\})
\]
is trivial. If $(S'_1, \overline{S'_1})$ and $(S_1, \{t_0\})$ are equivalent for some $\overline m>t_0\in \overline S_1$, then $m\in S'_1\subset S$ yields $\overline m\not\in\overline S$. Moreover,
\[
\psi_1(S'_1, \overline{S'_1})\sim (S_1, \{t_0\}\cup\overline S)\sim (S_1, \overline S)\sim (S, \overline S)\not\in \caD^1.
\]
Conversely, if $(S', \overline{S'})\in \caD^1$ is trivial, it is easy to check that $\psi_1^{-1}(S', \overline{S'})$ is trivial. If $\psi_1(S'_1, \overline{S'_1})\sim (S, \overline S)\not\in \caD^1$, then $m\in S'_1\subset S_1\subset S$. Lemma \ref{splem7} shows that $(S, \overline S)\sim (S_1, \overline S\cup\{t_0\})$ for some $\overline m>t_0\in \overline{S'_1}$. Hence $(S'_1, \overline{S'_1})\sim (S_1, \{t_0\})$ by (1).

Next assume that $J^1=\Delta_1$, that is, $\overline S\supset\{1, \ldots, \overline m-1\}$. If $(S', \overline{S'})\in \caD^1$ is nontrivial, we have $\overline{S'}=\overline S=\{1, \ldots, \overline m-1\}$. Thus $(S', \overline{S'})\sim (S, \overline S)\in \caD^2$ in view of $m\in S$. Then assume that $I^1=\Delta_1$ and $J^1\neq\Delta_1$. In this case, $n_s^\lambda(a_t)=0$ for $m>s\in S$ and $t\in\overline T\cup\{\overline m\}$. Moreover, $|\overline T\backslash\{\overline m\}|>0$. This forces $(\{m\}, \overline S)\in \caD$. If $(S', \overline{S'})\in \caD^1$ is nontrivial, Lemma \ref{splem3} yields $\overline{S'}=\overline S$ or $S'=\{m\}$. In either case, one has $(S', \overline{S'})\sim (S, \overline S)\in \caD^2$.

\end{proof}

\begin{theorem}\label{bkssthm22}
For $\Phi=B_n, C_n$, the system $(\Phi, \Phi_I, \Phi_J)$ has $2^k$ blocks, where $k$ is the number of equivalence classes of nontrivial separable pairs.
\end{theorem}
\begin{proof}
We adopt notations in Lemma \ref{bksslem1}. If $(\Phi, \Phi_I, \Phi_J)$ is not separable, then $\caD$ is empty and $(\Phi, \Phi_I, \Phi_J)$ contains $2^0=1$ blocks by Theorem \ref{bkswthm2}. If $(\Phi, \Phi_I, \Phi_J)$ is weakly separable, then $(n_m, \overline n_{\overline m})=(m-1, 0)$ or $(0, \overline m-1)$. So $\caD$ contains only $(\{m\}, \{1, \ldots, \overline m-1\})$ or $(\{1, \ldots, m-1\}, \{\overline m\})$. We obtain $k=1$ in either case. In view of Theorem \ref{bkswthm2}, the system contains $2^1=2$ blocks.

Now suppose $(\Phi, \Phi_I, \Phi_J)$ is strongly separable. We use induction on $n>1$ (the case $n=1$ is not separable). Choose $(S, \overline S)$ which is a minimal element of the set of strongly separable pairs. Therefore $(\Phi_1, \Phi_{I^1}, \Phi_{J^1})$ is not strongly separable by Lemma \ref{bksslem1} ($\psi_1^{-1}(S, \overline S)=(S, \{\overline m\})$ is weakly separable when $\overline m\in\overline S$). If $(S', \overline{S'})\in \caD\backslash \caD^2$ is strongly separable, the minimality of $(S, \overline S)$ and Lemma \ref{splem3} implies $S'=S$ or $\overline{S'}=\overline{S}$ (otherwise $(S, \overline S)<(S', \overline{S'})\in \caD^2$), that is, $(S', \overline{S'})\sim (S, \overline S)$. If $(S', \overline{S'})\in \caD\backslash \caD^2$ is weakly separable, Lemma \ref{splem3} yields $(S', \overline{S'})=(\{m\}, \{1, \ldots, \overline m-1\})$. Moreover, one obtains either $(S', \overline{S'})\sim (S, \overline S)$ or $(\Phi_1, \Phi_{I^1}, \Phi_{J^1})$ is weakly separable with respect to $\psi_1^{-1}(S', \overline{S'})$.

First assume that $(\Phi_1, \Phi_{I^1}, \Phi_{J^1})$ contains only one block. Then this system is not separable or $\Phi_1$ is empty by Theorem \ref{bkssthm1}. If $m\in S$, then $(S, \overline S)\in \caD^2\backslash\caD^1$. The previous argument shows that $D\backslash \caD^2$ is contained in the equivalence class of $(S, \overline S)$, that is, $\caD^2$ have $k$ nontrivial equivalence classes as $\caD$, so does $D(\lambda^2, \Phi_I\cap\Phi_2, \Phi_2)$ when $I^2, J^2\neq\Delta_2$. Combined with Theorem \ref{bkswthm2} and Theorem \ref{bkssthm2}, the induction hypothesis implies that $(\Phi, \Phi_I, \Phi_J)$ has $2^k$ blocks. When $I^2=\Delta_2$ or $J^2=\Delta_2$, $(S, \overline S)$ is trivial and $\caD^2$ contains only trivial pairs. Otherwise Lemma \ref{bksslem1} shows that $(S, \overline S)\in \caD^1$, a contradiction. It follows that $\caD=(\caD\backslash \caD^2)\cup \caD^2$ contains only trivial pairs. Theorem \ref{bkssthm1} and Theorem \ref{bkssthm2} shows that $(\Phi, \Phi_I, \Phi_J)$ has $2^0=1$ block. If $m\not\in S$, then $I^1=\Delta_1$ or $J^1=\Delta_1$ (otherwise $(\Phi_1, \Phi_{I^1}, \Phi_{J^1})$ is weakly separable with respect to $\psi_1^{-1}(S, \overline{S})$). This forces $\Phi_1=\emptyset$ and $(S, \overline S)=(\{s_0\}, \{1, \ldots, \overline m\})$ ($1\leq s_0<m$) to be trivial. Lemma \ref{bksslem1} shows that $D(\lambda^2, \Phi_I\cap\Phi_2, \Phi_2)$ has $k$ nontrivial equivalence classes as $\caD$ and $(\Phi, \Phi_I, \Phi_J)$ has $1\times 2^k=2^k$ blocks.

Then assume that $(\Phi_1, \Phi_{I^1}, \Phi_{J^1})$ contains two blocks. In view of Theorem \ref{bkssthm1}, this system is weakly separable with respect to $(\{m\}, \overline S_1)$ or $(S_1, \{\overline m\})$ (see \ref{bkssl1eq1}). We get either $(n_m, \overline n_{\overline m})=(\overline m-1, 0)$ or $\overline m\in \overline S$. It follows from $I^1, J^1\neq\Delta_1$ that $|S\backslash\{m\}|>0$ and $|\overline S\backslash\{\overline m\}|<\overline m-1$. First consider the case $(n_m, \overline n_{\overline m})=(\overline m-1, 0)$, then $(\{m\}, \{1, \ldots, \overline m-1\})\in \caD\backslash\caD^2$. We already show that $(S', \overline{S'})\sim(S, \overline S)$ for any other $(S', \overline{S'})\in \caD\backslash \caD^2$. Lemma \ref{splem3} implies $m\in S$ and $(S, \overline S)\in \caD^2$. We claim that $(\{m\}, \{1, \ldots, \overline m-1\})\not\sim (S', \overline{S'})$ for any strongly separable pair $(S', \overline{S'})\in \caD$. Otherwise Lemma \ref{splem7} and the minimality of $(S, \overline S)$ force $S=\{m\}$ or $\overline S=\{1, \ldots, \overline m-1\}$, contradicts $|S\backslash\{m\}|>0$ or $|\overline S\backslash\{\overline m\}|<\overline m-1$. With $(\{m\}, \{1, \ldots, \overline m-1\})\not\in \caD^2$ and $(S, \overline S)\in \caD^2$, $\caD^2$ contains $k-1$ equivalence class of nontrivial separable pairs, so does $D(\lambda^2, \Phi_I\cap\Phi_2, \Phi_2)$ when $I^2, J^2\neq\Delta_2$ in view of Lemma \ref{bksslem1}. The induction hypothesis implies that $D(\lambda^2, \Phi_I\cap\Phi_2, \Phi_2)$ has $2^{k-1}$ blocks. Theorem \ref{bkssthm2} shows that $(\Phi, \Phi_I, \Phi_J)$ contains $2\times 2^{k-1}=2^k$ blocks. When $I^2=\Delta_2$ or $J^2=\Delta_2$, $(S, \overline S)$ is trivial. Hence $\caD$ contains only one nontrivial equivalence class and $(\Phi, \Phi_I, \Phi_J)$ has $2^1$ blocks in view of Theorem \ref{bkssthm1} and Theorem \ref{bkssthm2}. Next consider the case $\overline m\in \overline S$. Lemma \ref{splem8} shows that $(S, \overline S)\not\in \caD^2$ is not trivial. In view of Lemma \ref{bksslem1}, $\caD(\lambda^2, \Phi_I\cap\Phi_2, \Phi_2)$ contains $k-1$ nontrivial equivalence classes. Similar argument shows that $(\Phi, \Phi_I, \Phi_J)$ has $2^k$ blocks.

\end{proof}

\subsection{strongly separable systems for $\Phi=D_n$} In view of Remark \ref{lrrmk2}, it suffices to consider the case when $I$ is standard. We say the separable pair $(S, \overline S)\in \caD$ is \emph{odd} if $m\in S$. Otherwise we say $(S, \overline S)$ is \emph{even}. The following result is an immediate consequence of Lemma \ref{splem3}.

\begin{lemma}\label{bksslem2}
Let $\Phi=D_n$. If $(S, \overline S)\in \caD$ is an odd $($resp. even$)$ separable pair, then all the other separable pairs of $\caD$ are odd $($resp. even$)$. In particular, if $(S, \overline S)$ is odd, it is not trivial. If $(S, \overline S)$ is even, it is strongly separable.
\end{lemma}

\begin{definition}\label{bkssdef0}
We say the system $(\Phi, \Phi_I, \Phi_J)$ is \emph{odd} (resp. \emph{even}) if $\caD(\Phi, \Phi_I, \Phi_J)$ contains an odd (resp. even) separable pair.
\end{definition}

Suppose $(\Phi, \Phi_I, \Phi_J)$ is strongly separable relative to $(S, \overline S)$. Let $\overline\lambda$ be a dominant weight with $\Phi_{\overline\lambda}=\Phi_J$. There exists strongly separable pair $(S, \overline S)$ so that $n_s^\lambda(a_t)$ is maximal for $s\in S$, $t\in\overline S$ and $n_s^\lambda(a_t)=0$ for $s\not\in S$, $t\not\in\overline S$, where $\lambda\in W\overline\lambda$ is a $\Phi_I$-regular weight. Denote $\overline\caA=\{a_t\mid t\in\overline S\}$ and $\overline\caA^*=\{\pm a \mid 0<a\in \overline\caA\}$. Set $p=|\overline\caA\backslash\{0\}|$, $g=|S\backslash\{m\}|$ and $h=q_{m-1}+p$. For $\lambda\in{}^IW^J\overline\lambda$, there exists a unique permutation $w_s\in W_{I_s}$ for each $s\in S$ so that $\nu=(\prod_{s\in S}w_s)\lambda$ satisfies
\begin{equation}\label{bksst4eq1}
\begin{aligned}
&\{\nu_{q_{s-1}+1}, \ldots, \nu_{q_{s-1}+2p}\}=\overline\caA^*,\\
\nu_{q_{s-1}+1}>&\ldots>\nu_{q_{s-1}+2p},\quad \nu_{q_{s-1}+2p+1}>\ldots>\nu_{q_s}
\end{aligned}
\end{equation}
when $m>s\in S$ and
\begin{equation}\label{bksst4eq2}
\begin{aligned}
&\{\nu_{q_{m-1}+1}, \ldots, \nu_{h-1}, \nu_{h}\}=\overline\caA\backslash\{0\},\\
\nu_{q_{m-1}+1}>&\ldots>\nu_{h-1}>\nu_{h}>0,\quad\nu_{h+1}>\ldots>\nu_{n-1}>\nu_n=0
\end{aligned}
\end{equation}
when $m\in S$ (then $\overline m\in \overline S$ and $0\in \overline\caA$). Denote $w_\lambda=\prod_{r\in S}w_r$. Then $w_\lambda$ is unique, so is $\nu$. Set $T=\{1, 2, \ldots, m\}\backslash S$.

\subsubsection{The case when $(S, \overline S)$ is odd} In this case $0\in \overline\caA$ and $h<n$. Set
\[
\Phi_1=(\sum_{\substack{q_{s-1}+2p<i\leq q_s\\
s\in S\backslash\{m\}}}\bbQ e_i+\sum_{h<j\leq n}\bbQ e_j)\cap\Phi.
\]
and
\[
\Phi_2=(\sum_{\substack{q_{s-1}<i\leq q_s\\
s\in T\backslash\{m\}}}\bbQ e_i+\sum_{q_{m-1}<j\leq h}\bbQ e_j+\bbQ e_n)\cap\Phi.
\]
Then $\Phi_1, \Phi_2$ are proper subsystem of $\Phi$ (could be empty, see Remark \ref{bkssrmk1}). With $h<n$, then $\eta|_{\Phi^1}$ and $\eta|_{\Phi^2}$ share an entry $\eta_n$ for any $\eta\in\frh^*$.

For $i=1, 2$, let $W_i$ be the Weyl group of $\Phi_i$ and $\Delta_i$ be the simple system corresponding to $\Phi_i^+=\Phi_i\cap\Phi^+$. Denote $I^i=\Phi_I\cap\Delta_i$. One has $I^1=I\cap\Phi_1$ and $I^2=(I\cap\Phi_2)\cup\{e_h\pm e_n\}$. Put $\lambda^i=\nu|_{\Phi_i}$. Evidently, $\lambda^i\in\Lambda^+(\Phi_{I^i}, \Phi_i)$. Let $\overline{\lambda^i}\in W_i\lambda^i$ be a dominant weight. Since $\lambda^i$ contains a $0$-entry $\nu_n$. Thus $s_{e_n}\overline{\lambda^i}=\overline{\lambda^i}$. Similar to the proof for $\Phi=B_n$, $\overline{\lambda^i}$ is uniquely determined by $\overline\lambda$ and $(S, \overline S)$. Set $J^i=\{\alpha\in \Delta_i\mid \langle\overline{\lambda^i}, \alpha\rangle=0\}$. It can be checked that $\lambda\ra(\lambda^1, \lambda^2)$ gives a bijection between ${}^IW^J\overline\lambda$ and $({}^{I^1}W_1^{J^1}\overline{\lambda^1}, {}^{I^2}W_2^{J^2}\overline{\lambda^2})$. We have proved the first part of the following theorem.

\begin{theorem}\label{bkssthm3}
Let $\Phi=D_n$ and $I, J\subset\Delta$. If $(\Phi, \Phi_I, \Phi_J)$ is strongly separable relative to $(S, \overline S)$ with $m\in S$ and $\overline m\in \overline S$, then
\begin{itemize}
\item [(1)] There exist proper subsystem $\Phi_i$ of $\Phi$ for $i=1, 2$ such that
\[
{}^IW^{J}\simeq{}^{I^1}W_1^{{J^1}}\times{}^{I^2}W_2^{{J^2}}
\]
where $I^i, J^i$ are determined by $\Phi_i$ and $W_i$ is the Weyl group of $\Phi_i$.

\item [(2)] Let $\overline\lambda$ be a dominant weight with $\Phi_{\overline\lambda}=\Phi_J$. There exist dominant weights $\overline{\lambda^i}$ of $\Phi_i$ such that each $\nu\in{}^IW^J\overline\lambda$ correspond $\nu^i\in {}^{I^i}W_i^{J^i}\overline{\lambda^i}$. Moreover, $\lambda\lera\mu$ if and only if $\lambda^i\lera\mu^i$ relative to $(\Phi_{I^i}, \Phi_i)$.
\end{itemize}
\end{theorem}
\begin{proof}
The proof is similar to the case $\Phi=B_n$. If $\lambda^i\lera\mu^i$ for $i=1, 2$, then $\lambda\lera\mu$ follows from Lemma \ref{lrlem0} and Proposition \ref{lrprop2}. Now suppose $\lambda, \mu\in{}^IW^J\overline\lambda$ satisfying $\mu=ws_\beta\lambda$ for $w\in W_I$ and $\beta\in\Psi_\lambda^+$. Denote $\zeta=w_\mu\mu$. Suppose that $\beta=e_i\pm e_j$ for $1\leq i<j\leq n$, where $1\leq s\leq t\leq m$ with $q_{s-1}<i\leq q_s$ and $q_{t-1}<j\leq q_t$. Set $|\lambda_i|=a$ and $|\lambda_j|=b$. One obtains $s<m$ since $\beta\not\in\Phi_I$. Suppose that $w=w'w''$ for $w'\in W_{I_s}$ and $w''\in W_{I_t}$. In particular, $w's_\beta\lambda\in\Lambda_{I_s}^+$ and $w''s_\beta\lambda\in\Lambda_{I_t}^+$. We can assume that $w_\lambda e_i=e_k$ for $q_{s-1}<k\leq q_s$ and $w_\lambda e_j=e_l$ for $q_{t-1}<l\leq q_t$. Set $\gamma=w_\lambda\beta=e_k\pm e_l$.

If $s\in S$ and $t\in T$, then $b\in \overline\caA$ and $s<t$. So $n^\lambda_s(b)$ is maximal. In view of Lemma \ref{bkswlem7}, $s_\beta\lambda$ is $\Phi_I$-singular, a contradiction.

If $s\in T$ and $t\in S$, then $a\in \overline\caA$ and $n^\lambda_t(a)$ is maximal. Lemma \ref{bkswlem7} implies $a=b>0$, $t=m\in S$ and $s_\beta\lambda=s_{e_i}s_{e_j}\lambda$. So $\lambda_i=\lambda_j=a>0$ (since $\langle\lambda, \beta\rangle>0$) and $\beta=e_i+e_j$. Then $\mu=(w's_{e_i})(w'' s_{e_j})\lambda\in\Lambda_I^+$ implies $w's_{e_i}\lambda\in\Lambda_{I_s}^+$ and $w''s_{e_j}\lambda\in\Lambda_{I_m}^+$. It follows that $w''=s_{e_j+e_n}s_{e_j-e_n}$ and $\mu=w's_{e_i}s_{e_n}\lambda=w's_{e_i}\lambda\in\Lambda_I^+$. With $s\not\in S$, it can be verified that
$w_\lambda\mu=w_\lambda w's_{e_i}\lambda=w's_{e_i}\nu$ satisfies (\ref{bksst4eq1}) and (\ref{bksst4eq2}). Therefore $w_\mu=w_\lambda$ and $\zeta=w's_{e_i}\nu$. We also obtain $l\leq h$ in view of $|\nu_l|=b=a\in \overline\caA$. Thus $\gamma=w_\lambda\beta=e_i+e_l\in \Phi_2$ ($w_\lambda e_i=e_i$ for $s\not\in S$) and $s_{e_i+e_l}\nu=s_{e_i}s_{e_l}\nu$. With $0\in \overline\caA$, one gets $e_l\pm e_n\in\Phi_{I^2}$. Therefore, $\mu^1=(w's_{e_i}\nu)|_{\Phi_1}=\nu|_{\Phi_1}=\lambda^1$ and
\[
\mu^2=(w_\lambda\mu)|_{\Phi_2}
=(w_\lambda w''w_\lambda^{-1}w's_{e_i+e_l}\nu)|_{\Phi_2}
=(s_{e_l+e_n}s_{e_l-e_n}w')s_{e_i+e_l}\lambda^2.
\]
Since $e_i+e_l$ is a linked root from $\nu$, Proposition \ref{lrprop2} implies that it is also a linked root from $\lambda^2$.

If $s, t\in S$, Lemma \ref{bkswlem7} shows that $a, b\not\in \overline\caA$. It follows from $|\nu_k|=a$ and $|\nu_l|=b$ that $\nu_k, \nu_l\not\in \overline\caA^*$. Therefore $\gamma=e_k\pm e_l\in\Phi_1$ and $\gamma\perp\Phi_2$. We can find $x_s\in W_{I^1\cap I_s}$ and $x_t\in W_{I^1\cap I_t}$ such that (\ref{bksst4eq1}) and (\ref{bksst4eq2}) are satisfied with $\nu$ replaced by $\eta=x_sx_ts_{\gamma}\nu=x_sx_tw_\lambda {(w'w'')}^{-1}\mu$. The uniqueness of $w_\mu$ implies that $\eta=\zeta=w_\mu\mu$ with $w_\mu=x_sx_tw_\lambda {w}^{-1}$. We get $\mu^2=\lambda^2$ and $\mu^1=(x_sx_ts_{\gamma}\nu)|_{\Phi_1}=x_sx_ts_{\gamma}\lambda^1$.
With Proposition \ref{lrprop2}, $\gamma$ is a linked root from $\nu$ if and only if it is a linked root from $\lambda^1$.

If $s, t\in T$, then $s, t<m$ and $\mu_u=\lambda_u$ for $q_{r-1}<u\leq q_r$ with $r\in S$. We obtain $w_\mu=w_\lambda$ and $\zeta=w_\lambda ws_\beta\lambda=ws_\beta\nu$. It follows that $\mu^1=\lambda^1$ and $\mu^2=ws_{\beta}\lambda^2$ with linked root $\beta$ from $\lambda^2$.
\end{proof}

The proof of the following result is similar to the proof of Lemma \ref{bksslem1}.

\begin{lemma}\label{bksslem3}
Let $\Phi=D_n$. Let $(S, \overline S)\in \caD$ be an odd strongly separable pair. Set
\[
\begin{aligned}
\caD^1=&\{(S', \overline{S'})\in \caD\mid (S', \overline{S'})\leq(S, \overline S)\};\\
\caD^2=&\{(S', \overline{S'})\in \caD\mid (S', \overline{S'})\geq(S, \overline S)\}.
\end{aligned}
\]
Fix $i\in\{1, 2\}$. If $I^i, J^i\neq\Delta_i$, there exists bijection $\psi_i$ between $D(\lambda^i, \Phi_I\cap\Phi_i, \Phi_i)$ and $\caD^i$ satisfying the following conditions:
\begin{itemize}
  \item [(1)] $\psi_i$ sends equivalent pairs to equivalent pairs, so does $\psi_i^{-1}$.
  \item [(2)] If $(S'_i, \overline{S'_i})\in \caD(\lambda^i, \Phi_I\cap\Phi_i, \Phi_i)$ is strongly separable, then $\psi_i(S'_i, \overline{S'_i})$ is strongly separable.
\end{itemize}
Otherwise $(S_i, \overline S_i)\sim (S, \overline S)$ for any $(S_i, \overline S_i)\in \caD^i$.
\end{lemma}

\begin{theorem}\label{bkssthm32}
For $\Phi=D_n$, suppose the system $(\Phi, \Phi_I, \Phi_J)$ has an odd separable pair. Then it contains $2^{k-1}$ blocks, where $k$ is the number of equivalence classes of separable pairs.
\end{theorem}

\begin{proof}
When $(\Phi, \Phi_I, \Phi_J)$ is weakly separable, then either $(\{m\}, \{1, \ldots, \overline m\})\in \caD$ or $(\{1, \ldots, m\}, \{\overline m\})\in \caD$. If only one of these two pairs is contained in $\caD$, Theorem \ref{bkswthm3} shows that the system has $2^{1-1}=1$ block. Otherwise it has $2^{2-1}=2$ blocks.

Now suppose $(\Phi, \Phi_I, \Phi_J)$ is strongly separable. We use induction on $n>2$, choose $(S, \overline S)$ which is a minimal element of the set of strongly separable pairs. Therefore $(\Phi_1, \Phi_{I^1}, \Phi_{J^1})$ is not strongly separable by Lemma \ref{bksslem3}. One has $(S, \overline S)\in\caD^1\cap\caD^2$. So $\psi_1^{-1}(S, \overline S)=(S, \{\overline m\})$ is a weakly separable pair of $(\Phi_1, \Phi_{I^1}, \Phi_{J^1})$. If $(S', \overline{S'})\in \caD\backslash \caD^2$ is strongly separable, Lemma \ref{splem3} implies $(S', \overline{S'})\sim (S, \overline S)$. If $(S', \overline{S'})\in \caD\backslash \caD^2$ is weakly separable, Lemma \ref{splem3} yields $(S', \overline{S'})=(\{m\}, \{1, \ldots, \overline m\})$. Thus $\psi_1^{-1}(S', \overline{S'})=(\{m\}, \overline S_1)$ (we adopt the definition of $\overline S_1$ given in \ref{bkssl1eq1}) is another weakly separable pair of $(\Phi_1, \Phi_{I^1}, \Phi_{J^1})$. Theorem \ref{bkswthm3} shows that $(\Phi_1, \Phi_{I^1}, \Phi_{J^1})$ has two blocks in this case.

First assume that $(\Phi_1, \Phi_{I^1}, \Phi_{J^1})$ contains only one block. The previous reasoning implies that $D\backslash \caD^2$ is contained in the equivalence class of $(S, \overline S)\in \caD^2$. Thus $\caD^2$ have $k$ equivalence classes as $\caD$, so does $D(\lambda^2, \Phi_I\cap\Phi_2, \Phi_2)$ when $I^2, J^2\neq\Delta_2$. Combined with Theorem \ref{bkswthm3} and Theorem \ref{bkssthm3}, the induction hypothesis implies that $(\Phi, \Phi_I, \Phi_J)$ has $2^k$ blocks. When $I^2=\Delta_2$ or $J^2=\Delta_2$, Lemma \ref{bksslem3} yields that $D=(D\backslash \caD^2)\cup \caD^2$ contains only one equivalence class represented by $(S, \overline S)$, while $(\Phi, \Phi_I, \Phi_J)$ has $2^{1-1}=1$ block in view of Theorem \ref{bkswthm3} and Theorem \ref{bkssthm3}.

Next assume that $(\Phi_1, \Phi_{I^1}, \Phi_{J^1})$ contains two blocks. Then $(\{m\}, \{1, \ldots, \overline m\})\in \caD\backslash\caD^2$ and $I^1, J^1\neq\Delta_1$. We obtain $|S|>1$ and $|\overline S|<\overline m$. With Lemma \ref{splem3} and Lemma \ref{splem7}, the minimality of $(S, \overline S)$ implies that $(\{m\}, \{1, \ldots, \overline m\})\not\sim (S', \overline{S'})$ for any other separable pair $(S', \overline{S'})\in \caD$. It follows that $\caD^2$ contains $k-1$ equivalence class of separable pairs, so does $D(\lambda^2, \Phi_I\cap\Phi_2, \Phi_2)$ when $I^2, J^2\neq\Delta_2$ in view of Lemma \ref{bksslem3}. The induction hypothesis and Theorem \ref{bkssthm2} show that $(\Phi, \Phi_I, \Phi_J)$ contains $2\times 2^{k-1}=2^k$ blocks. When $I^2=\Delta_2$ or $J^2=\Delta_2$, $\caD$ has two equivalence classes represented by $(\{m\}, \{1, \ldots, \overline m)$ and $(S, \overline S)$, while Theorem \ref{bkswthm3} and Theorem \ref{bkssthm3} show that $(\Phi, \Phi_I, \Phi_J)$ has $2^{2-1}=2$ blocks.
\end{proof}

\subsubsection{The case when $(S, \overline S)$ is even} In this case $m\not\in S$ and $0\not\in \overline\caA$. Set
\[
\Phi_1=(\sum_{\substack{q_{s-1}+2p<i\leq q_s\\
s\in S}}\bbQ e_i)\cap\Phi\ \mbox{and}\  \Phi_2=(\sum_{\substack{q_{s-1}<i\leq q_s\\
s\in T}}\bbQ e_i)\cap\Phi.
\]
Then $\Phi_1, \Phi_2$ are proper subsystem of $\Phi$ with $\Phi_1\perp\Phi_2$. Let $W_i$ be the Weyl group of $\Phi_i$ for $i=1, 2$. Let $\Delta_i$ be the simple system associated with $\Phi_i^+=\Phi_i\cap\Phi^+$. Denote $I^i=\Phi_I\cap\Delta_i=I\cap\Phi_i$. Put $\lambda^i=\nu|_{\Phi_i}$. Let $\overline{\lambda^i}$ be the dominant weight of $W_i\lambda^i$.

\begin{lemma}\label{bksslem4}
The pair $(S, \overline S)$ is trivial if and only if $\overline{\lambda^i}$ is unique for $i=1, 2$ and $\lambda\in{}^IW^J\overline\lambda$.
\end{lemma}
\begin{proof}
First assume that $(S, \overline S)$ is trivial. If $(S, \overline S)\sim(\{s_0\}, \{1, \ldots, \overline m-1\})$ for some $1\leq s_0<m$, we get $\overline n_{\overline m}=0$ (since $\overline n_{\overline m}\neq1$). In view of Lemma \ref{splem8}, $\lambda_j$ is fixed for any $q_{s-1}<j\leq q_s$ with $s\in S$. It follows from (\ref{bksst4eq1}) that $\lambda^1$ is unique, so is $\overline{\lambda^1}$. On the other hand, $n^{\lambda^2}(a_t)$ are fixed for any $1\leq t\leq\overline m$ (see \ref{bksst1eq4} for type $A$). It suffices to show that $P(\lambda^2)$ is fixed. With $\overline n_{\overline m}=0$, this follows from
\begin{equation}\label{bkssl7eq1}
P(\lambda)\equiv P(\lambda^1)+P(\lambda^2)+pg\ (\mathrm{mod} 2).
\end{equation}
In fact, keeping in mind of (\ref{bksst4eq1}) and (\ref{bksst4eq2}), we have
\[
\begin{aligned}
&|\{1\leq j\leq n\mid \nu_j<0\}|\\
=&pg+|\{q_{s-1}+2p<j\leq q_s\mid s\in S, \nu_j<0\}|+|\{q_{s-1}<j\leq q_s\mid s\in T, \nu_j<0\}|.
\end{aligned}
\]
In this case, $P(\overline\lambda)=P(\lambda)=P(\nu)$. Thus (\ref{bkssl7eq1}) is proved and $\overline{\lambda^2}$ is unique. If $(S, \overline S)\sim(\{1, \ldots, m-1\}, \{t_0\})$ for some $1\leq t_0<\overline m$, the proof is similar.

Conversely, assume that $\overline{\lambda^1}, \overline{\lambda^2}$ are fixed. It can be verified that $\lambda\ra(\lambda^1, \lambda^2)$ gives a bijection between ${}^IW^J\overline\lambda$ and ${}^{I^1}W_1^{J^1}\overline{\lambda^1}\times{}^{I^2}W_2^{J^2}\overline{\lambda^2}$. Since $(S, \overline S)$ is even, we have $n_s^\lambda(0)=0$ for $s\not\in S$. First consider the case $n_m, \overline n_{\overline m}>0$. Then $\lambda\neq \mu=s_{e_n}\lambda\in{}^IW^J\overline\lambda$. We obtain $\mu^2=s_{e_n}\lambda^2\neq\lambda^2$ and thus $\overline\mu^2\neq\overline{\lambda^2}$, a contradiction. Next suppose that $n_m>0$ and $\overline n_{\overline m}=0$. If $s_{e_j}\lambda$ is $\Phi_I$-regular for some $q_{s-1}<j\leq q_s$ with $\lambda_j\neq0$ and $s\in S$, then $\mu=ws_{e_j}s_{e_n}\lambda\in{}^IW^J\overline\lambda$ for some $w\in W_{I_s}$ and $\overline\mu^2\neq\overline{\lambda^2}$. Otherwise $(S, \overline S)$ is trivial in view of Lemma \ref{splem8}. Then assume that $n_m=0$ and $\overline n_{\overline m}>0$. We can also show that $(S, \overline S)$ is trivial by a similar argument using Lemma \ref{splem9} instead of Lemma \ref{splem8}. Finally suppose $n_m=\overline n_{\overline m}=0$. If $s_{e_j}\lambda$ is $\Phi_I$-regular for some $q_{s-1}<j\leq q_s$ with $\lambda_j\neq0$ and $s\in S$ and If $s_{e_k}\lambda$ is $\Phi_I$-regular for some $q_{t-1}<k\leq q_t$ with $\lambda_k\neq0$ and $t\not\in S$, then $\mu=ws_{e_j}s_{e_k}\lambda\in{}^IW^J\overline\lambda$ for some $w\in W_{I_s}W_{I_t}$. We arrive at a similar contradiction. Otherwise $(S, \overline S)$ is trivial by applying Lemma \ref{splem8} or Lemma \ref{splem9}.
\end{proof}

If $(S, \overline S)$ is trivial, set $J^i=\{\alpha\in \Delta_i\mid \langle\overline{\lambda^i}, \alpha\rangle=0\}$. We can get the isomorphism ${}^IW^J\overline\lambda\simeq{}^{I^1}W_1^{J^1}\overline{\lambda^1}\times{}^{I^2}W_2^{J^2}\overline{\lambda^2}$. If $(S, \overline S)$ is not trivial, choose $\overline{\lambda^i_+}\in\{\overline{\lambda^i}, \vf(\overline{\lambda^i})\}$ so that $P(\overline{\lambda^1_+})=\min\{P(\overline{\lambda^1}), P(\vf(\overline{\lambda^1}))\}$ and $P(\overline{\lambda^2_+})\equiv P(\overline\lambda)-pg\ (\mathrm{mod} 2)$. Put $\overline{\lambda^i_-}=\vf(\overline{\lambda^i_+})$ ($\overline{\lambda^i_-}\neq\overline{\lambda^i_+}$ if and only if $n^{\lambda^i}(0)=0$). Then $\lambda^1\in W_1\overline{\lambda^1_+}$ if and only if $\lambda^2\in W_2\overline{\lambda^2_+}$. Set $J^i=\{\alpha\in \Delta_i\mid \langle\overline{\lambda^i_+}, \alpha\rangle=0\}$. It is easy to check that $\vf: \lambda\ra(\lambda^1, \lambda^2)$ gives a bijection between ${}^IW^J\overline\lambda$ and $({}^{I^1}W_1^{J^1}\overline{\lambda^1_+}, {}^{I^2}W_2^{J^2}\overline{\lambda^2_+})\sqcup({}^{I^1}W_1^{\vf(J^1)}\overline{\lambda^i_-}, {}^{I^2}W_2^{\vf(J^2)}\overline{\lambda^i_-})$. We have proved the first part of the following theorem.

\begin{theorem}\label{bkssthm4}
Let $\Phi=D_n$ and $I, J\subset\Delta$. If $(\Phi, \Phi_I, \Phi_J)$ is strongly separable relative to $(S, \overline S)$ with $m\not\in S$, $\overline m\not\in \overline S$, then
\begin{itemize}
\item [(1)] There exist proper subsystem $\Phi_i$ of $\Phi$ for $i=1, 2$ such that
\[
{}^IW^J\simeq\left\{\begin{aligned}
&{}^{I^1}W_1^{{J^1}}\times{}^{I^2}W_2^{{J^2}},\qquad\qquad\qquad\qquad\qquad\quad\ \mbox{if}\ (S, \overline S)\ \mbox{is trivial};\\
&{}^{I^1}W_1^{{J^1}}\times{}^{I^2}W_2^{{J^2}}\sqcup{}^{I^1}W_1^{\vf({J^1})}\times{}^{I^2}W_2^{\vf({J^2})},\quad \mbox{if}\ (S, \overline S)\ \mbox{is nontrivial},
\end{aligned}
\right.
\]
where $I^i, J^i$ are determined by $\Phi_i$ and $W_i$ is the Weyl group of $\Phi_i$.

\item [(2)] Let $\overline\lambda$ be a dominant weight with $\Phi_{\overline\lambda}=\Phi_J$. There exist dominant weights $\overline{\lambda^i_+}$ of $\Phi_i$ such that each $\nu\in{}^IW^J\overline\lambda$ corresponds $\nu^i\in {}^{I^i}W_i^{J^i}\overline{\lambda^i_+}$ or ${}^{I^i}W_i^{\vf(J^i)}\vf(\overline{\lambda^i_+})$. Moreover, $\lambda\lera\mu$ if and only if $\lambda^i\lera\mu^i$ relative to $(\Phi_{I^i}, \Phi_i)$ or $(\Phi_{\vf(I^i)}, \Phi_i)$.
\end{itemize}
\end{theorem}
\begin{proof}
As in the proof of Theorem \ref{bkssthm3}, let $\mu=ws_\beta\lambda$ for $\lambda, \mu\in{}^IW^J\overline\lambda$, $w\in W_I$ and $\beta\in\Psi_\lambda^+$. Denote $\zeta=w_\mu\mu$. Suppose that $\beta=e_i\pm e_j$ for $1\leq i<j\leq n$, where $1\leq s\leq t\leq m$ with $q_{s-1}<i\leq q_s$ and $q_{t-1}<j\leq q_t$. Put $|\lambda_i|=a$ and $|\lambda_j|=b$. We have $s<m$ for $\beta\not\in\Phi_I$. Moreover, $w=w'w''$ for $w'\in W_{I_s}$ and $w''\in W_{I_t}$. In particular, $w's_\beta\lambda\in\Lambda_{I_s}^+$ and $w''s_\beta\lambda\in\Lambda_{I_t}^+$. We can assume that $w_\lambda e_i=e_k$ for $q_{s-1}<k\leq q_s$ and $w_\lambda e_j=e_l$ for $q_{t-1}<l\leq q_t$ and $\gamma=w_\lambda\beta$.

If $s\in S$ and $t\in T$, Lemma \ref{bkswlem7} implies that $s_\beta\lambda$ is $\Phi_I$-singular.

If $s\in T$ and $t\in S$. Lemma \ref{bkswlem7} yields $t=m\in S$, a contradiction.

If $s, t\in S$, we get $a, b\not\in \overline\caA$ as in the proof of Theorem \ref{bkssthm3}. Thus $\nu_k, \nu_l\not\in \overline\caA^*$. One has $\gamma=w_\lambda\beta=e_k\pm e_l\in\Phi_1$ and $\gamma\beta\perp\Phi_2$. We can find $x_s\in W_{I^1\cap I_s}$ and $x_t\in W_{I^1\cap I_t}$ such that (\ref{bksst4eq1}) and (\ref{bksst4eq2}) are satisfied with $\nu$ replaced by $\eta=x_sx_ts_{\gamma}\nu=x_sx_tw_\lambda(w'w'')^{-1}\mu$. The uniqueness of $w_\mu$ implies that $\zeta=\eta$. Hence $\mu^2=\lambda^2$ and $\mu^1=x_sx_ts_{ \gamma}\lambda^1$ with linked root $w_\lambda\beta$, in view of Proposition \ref{lrprop2}.

If $s, t\in T$, then $\gamma=\beta\in\Phi_2$. Then (\ref{bksst4eq1}) and (\ref{bksst4eq2}) are satisfied with $\nu$ replaced by $ws_{\beta}\nu=w_\lambda\mu$. We obtain $w_\mu= w_\lambda$ and $\zeta=ws_{\beta}\nu$. Hence $\mu^1=\lambda^1$ and $\mu^2=ws_{\beta}\lambda^2$ with linked root $\beta$ from $\lambda^2$.

\end{proof}

\begin{remark}\label{bkssrmk2}
If $(S, \overline S)$ is trivial, the previous argument shows that either ${}^{I^1}W_1^{{J^1}}=1$ or ${}^{I^2}W_1^{{J^2}}=1$.
\end{remark}

Imitate the proof of Lemma \ref{bkswlem7}, we get:

\begin{lemma}\label{bksslem5}
Let $\Phi=D_n$. Let $(S, \overline S)\in \caD$ be an even strongly separable pair. Set
\[
\begin{aligned}
\caD^1=&\{(S', \overline{S'})\in \caD\mid (S', \overline{S'})\leq(S, \overline S),\ \overline{S'}\neq \overline S\backslash\{\overline m\}\};\\
\caD^2=&\{(S', \overline{S'})\in \caD\mid (S', \overline{S'})\geq(S, \overline S),\ S'\neq S\backslash\{m\}\}.
\end{aligned}
\]
Fix $i\in\{1, 2\}$. If $I^i, J^i\neq\Delta_i$, there exists bijection $\psi_i$ between $D(\lambda^i, \Phi_I\cap\Phi_i, \Phi_i)$ and $\caD^i$ satisfying the following conditions:
\begin{itemize}
  \item [(1)] $\psi_i$ sends equivalent pairs to equivalent pairs, so does $\psi_i^{-1}$.
  \item [(2)] $(S'_i, \overline{S'_i})$ is trivial if and only if $\psi_i(S'_i, \overline{S'_i})$ is trivial or $\psi_i(S'_i, \overline{S'_i})\sim (S, \overline S)$.
\end{itemize}
Otherwise $(S', \overline{S'})\sim (S, \overline S)$ for any nontrivial $(S', \overline{S'})\in \caD^i$.
\end{lemma}

\begin{theorem}\label{bkssthm42}
For $\Phi=D_n$, suppose the system $(\Phi, \Phi_I, \Phi_J)$ has no odd separable pair. Then it contains $2^k$ blocks, where $k$ is the number of equivalence classes of nontrivial separable pairs.
\end{theorem}

\begin{proof}
Lemma \ref{bksslem2} shows that any separable pair is strongly separable in this case. If $(\Phi, \Phi_I, \Phi_J)$ is not separable, then $\caD$ is empty and $(\Phi, \Phi_I, \Phi_J)$ contains $2^0=1$ blocks by Theorem \ref{bkswthm3}.

Next suppose $(\Phi, \Phi_I, \Phi_J)$ is separable. Using induction on $n$, the case $n=3$ is easy. For $n>3$, choose $(S, \overline S)$ which is a minimal element of $\caD$. Therefore $(\Phi_1, \Phi_{I^1}, \Phi_{J^1})$ is not separable by Lemma \ref{bksslem5}. Moreover, $D\backslash \caD^2$ is contained in the equivalence class of $(S, \overline S)\not\in\caD^1\cup\caD^2$.

If $(S, \overline S)$ is trivial, then $\caD^2$ contains $k$ nontrivial equivalence classes, so does $D(\lambda^2, \Phi_I\cap\Phi_2, \Phi_2)$ when $I^2, J^2\neq\Delta_2$. When $I^2=\Delta_2$ or $J^2=\Delta_2$, $\caD^2$ contains only trivial pairs. It follows that $\caD=(\caD\backslash \caD^2)\cup \caD^2$ contains only trivial pairs. With the induction hypothesis, Theorem \ref{bkswthm3} and Theorem \ref{bkssthm4} shows that $(\Phi, \Phi_I, \Phi_J)$ has $2^k$ blocks.

If $(S, \overline S)$ is not trivial, then $\caD(\lambda^2, \Phi_I\cap\Phi_2, \Phi_2)$ has $k-1$ nontrivial equivalence classes when $I^2, J^2\neq\Delta_2$. When $I^2=\Delta_2$ or $J^2=\Delta_2$, $\caD$ contains one equivalence class represented by $(S, \overline S)$. Applying the induction hypothesis, Theorem \ref{bkswthm3} and Theorem \ref{bkssthm4} implies that $(\Phi, \Phi_I, \Phi_J)$ has $2^k$ blocks.

\end{proof}

\begin{cor}\label{bksscor1}
Let $\Phi=B_n$, $C_n$ or $D_n$. The system $(\Phi, \Phi_I, \Phi_J)$ has $2^p$ blocks, where $0\leq p<\min\{m, \overline m\}$.
\end{cor}
\begin{proof}
With Lemma \ref{splem3}, we can choose a representative $(S_i, \overline S_i)$ $(1\leq i\leq k)$ of each nontrivial equivalence class such that $0<|S_1|<\ldots<|S_k|$ and $0<|\overline{S}_1|<\ldots<|\overline S_k|$. If $\Phi=B_n$, $C_n$, then $|S_k|<m$ and $|\overline S_k|<\overline m$ (otherwise $(S_k, \overline S_k)$ is trivial). Thus $k<\min\{m, \overline m\}$. If $\Phi=D_n$ and $m\not\in S_k$, we also get $k<\min\{m, \overline m\}$. If $\Phi=D_n$ and $m\in S_k$, then $k\leq\min\{m, \overline m\}$. Now the assertion follows from Theorem \ref{bkssthm22}, Theorem \ref{bkssthm32} and Theorem \ref{bkssthm42}.
\end{proof}

\begin{definition}\label{bkssdef1}
Suppose $(\Phi, \Phi_I, \Phi_J)$ contains at least two simple modules. We say $\caO_\lambda^\frp$ is \emph{pseudo-indecomposable} if it is not strongly separable, where $\Phi_J=\Phi_{\overline\lambda}$.
\end{definition}

Theorem \ref{bkswthm1}, Theorem \ref{bkswthm2} and Theorem \ref{bkswthm3} give the following result.

\begin{cor}\label{bksscor2}
Suppose $\caO_\lambda^\frp$ is pseudo-indecomposable. If $m, \overline m>1$ and
\[
(n_m, \overline n_{\overline m})=\left\{\begin{aligned}
 &(0, m-1), (\overline m-1, 0)\quad\mbox{for}\ \Phi=B_n, C_n \\
 &(\overline m, m)\qquad\qquad\qquad\quad \mbox{for}\ \Phi=D_n,
\end{aligned}
\right.
\]
then $\caO_\lambda^\frp$ has two blocks, otherwise it contains only one block. In case $\caO_\lambda^\frp$ contains two blocks, $L(\mu), L(\nu)\in\caO_\lambda^\frp$ are in the same block if and only if $P(\mu)=P(\nu)$.
\end{cor}

Choose $I, J\subset\Delta\subset\frh^*$. If $\frg\simeq\frso(2n, \bbC)$ for $n\geq1$, set $\overline {{}^IW{}^J}={}^IW^J\sqcup{}^IW^{\vf(J)}$. Combined with Theorem \ref{bkswthm2}, Theorem \ref{bkswthm3}, Theorem \ref{bkssthm1}, Theorem \ref{bkssthm2} Theorem \ref{bkssthm3} and Theorem \ref{bkssthm4} (with Remark \ref{bkssrmk2}), we have the following result.

\begin{theorem}\label{main}
Let $I, J\subset\Delta$ such that ${}^IW^J\neq\emptyset$. Choose $\lambda$ so that $\Phi_\lambda=\Phi_J$.
\begin{itemize}
\item [(1)] There exist $k\geq0$ and subalgebras $\frg_i$ $(1\leq i\leq k)$ with root systems $\Phi_i$ and simple system $\Delta_i$ and Weyl groups $W_i$ such that
\[
\overline{{}^IW^J}\simeq \overline{{}^{I^1}W_1^{J^1}}\times\ldots\times\overline{{}^{I^k}W_1^{J^k}}
\]
when $\Phi=D_n$ with $(\Phi, \Phi_I, \Phi_J)$ being even and
\[
{}^IW^J\simeq {}^{I^1}W_1^{J^1}\times\ldots\times{}^{I^k}W_1^{J^k}
\]
otherwise. Here $I^i, J^i\subset\Delta_i$ are determined by $I, J$. Moreover, the categories associated with ${}^{I^i}W_i^{J^i}$ are pseudo-indecomposable.

\item [(2)] $L(\mu), L(\nu)\in\caO_\lambda^\frp$ are in the same block if and only if $L(\mu^i), L(\nu^i)$ are in the same block for each subcategories. The weights $\mu^i, \nu^i\in\frh_i^*$ are determined by $\mu, \nu$, where $\frh_i^*$ is a Cartan subalgebra of $\frg_i$.
\end{itemize}

\end{theorem}
%
%
\section{Blocks and partitions}
%
%
The last section is devoted to discover the relation between blocks and partitions which classified the nilpotent orbits \cite{CM}.

\subsection{Nonzero criterion of blocks} We always assume that ${}^IW^J\neq\emptyset$ in the previous sections. Now we recall the criterion obtained in \cite{P3} for ${}^IW^J$ to be nonempty. For a positive integer $N$, define the set
\[
\caP(N):=\{\pi=(\pi_1, \pi_2, \ldots, \pi_N)\in\bbZ^N\mid\pi_1\geq\pi_2\geq\ldots\geq\pi_N\geq0,\ \sum_{i=1}^n\pi_i=N\}
\]
of {\it partitions} of $N$. For any $\pi\in\caP(N)$, the {\it dual partition} $\pi^t$ of $\pi$ is defined by
\[
\pi^t_i:=|\{j\mid\pi_j\geq i\}|\quad\mbox{for}\quad1\leq i\leq N.
\]
We will frequently omit the trailing $0$'s when writing down a partition. There exists a partial ordering on $\caP(N)$, called the {\it dominance ordering}, defined as follows. If $\pi, \eta\in\caP(N)$, we write $\pi\unlhd\eta$ if
\[
\pi_1+\pi_2+\ldots+\pi_i\leq\eta_1+\eta_2+\ldots+\eta_i.
\]
for all $i\leq N$.

The adjoint group $G$ of $\frg$ is a connected complex Lie group with Lie algebra $\frg$. The variety $\caN(\frg)$ of nilpotent elements of $\frg$ has finitely many $G$-orbits (so called nilpotent orbit) under the natural action of $G$. For any $I\subset\Delta$, define the nilpotent orbit
\[
O_I=G\cdot(\sum_{\alpha\in I}E_\alpha),
\]
where $E_\alpha$ are nonzero elements of $\frg_\alpha$.


Let $\Phi$ be a classical root system of type $X$, where $X$ is one of $A$, $B$, $C$ and $D$. It is well known that the nilpotent orbits of a classical Lie algebra can be parameterized by certain partitions of $N(X_n)$ (see for example \cite{CM}, \S 5), where $N(X_n)$ is given as follows. The nilpotent orbits of type $A_{n-1}$ are in one-to-one correspondence with the set $\caP_A(n)=\caP(n)$ of all partitions of $N(A_{n-1})=n$. The nilpotent orbits of type $B_n$ are in one-to-one correspondence with the set $\caP_B(2n+1)$ of partitions of $N(B_n)=2n+1$ in which even parts occur with even multiplicity. The nilpotent orbits of type $C_n$ are in one-to-one correspondence with the set $\caP_C(2n)$ of partitions of $N(C_n)=2n$ in which odd parts occur with even multiplicity. The nilpotent orbits of type $D_n$ are parameterized by partitions $\caP_D(2n)$ of $N(D_n)=2n$ in which even parts occur with even multiplicity, except the very even partitions $\pi$ (i.e., those with only even parts, each having even multiplicity) correspond to two orbits. For type $D_n$ such that $\pi$ is a very even partition, denote by $O_{\pi}^{\mathrm{\Rmnum{1}}}$ and $O_{\pi}^{\mathrm{\Rmnum{2}}}$ the two orbits associated with $\pi$. For all other cases, the orbit corresponding to the partition $\pi$ will be denoted by $O_\pi$. The partition $\pi_I$ associated with $O_I$ is given as follows (see \cite{Sp} or \cite{P3} or \cite{CM} \S 5).
\begin{itemize}
\item Let $\Phi=A_{n-1}$. Then $\pi_I\in\caP(n)$ is given by arranging the set
\[
\{n_1, n_2, \ldots, n_{m-1}, n_{m}\}.
\]

\item Let $\Phi=B_n$. Then $\pi_I\in\caP(2n+1)$ is given by arranging the set
\[
\{n_1, n_1, n_2, n_2, \ldots, n_{m-1}, n_{m-1}, 2n_m+1\}.
\]

\item Let $\Phi=C_n$. Then $\pi_I\in\caP(2n)$ is given by arranging the set
\[
\{n_1, n_1, n_2, n_2, \ldots, n_{m-1}, n_{m-1}, 2n_m\}.
\]

\item Let $\Phi=D_n$. If $n_m\geq2$,  $\pi_I\in\caP(2n)$ is given by arranging the set
\[
\{n_1, n_1, n_2, n_2, \ldots, n_{m-1}, n_{m-1}, 2n_m-1, 1\}.
\]
If $n_m=0$ (note that $n_{m}\neq1$),  $\pi_I\in\caP(2n)$ is given by arranging the set
\[
\{n_1, n_1, n_2, n_2, \ldots, n_{m-1}, n_{m-1}\}.
\]
When $I$ is standard and $\pi_I$ is very even, $O_I=O_{\pi_I}^{\mathrm{\Rmnum{1}}}$. When $I$ is not standard and $\pi_I$ is very even, then $\pi_I=\pi_{\vf(I)}$ and $O_I=O_{\pi_I}^{\mathrm{\Rmnum{2}}}$.
\end{itemize}

There also exists a nilpotent orbit $R_I$, which is defined to be the nilpotent orbit associated with $G\cdot\fru_I$. It is called the {\it Richardson orbit} corresponding to $I$. Let $\pi$ be a partition contained in $\caP_X(N(X_n))$, where $X=A$, $B$, $C$ or $D$. There is a unique partition $\pi_X\in\caP_X(N(X_n))$ so that $\eta\unlhd\pi_X\unlhd\pi$ for any $\eta\in\caP_X(N(X_n))$ with $\eta\unlhd\pi$. The partition $\pi_X$ is called the $X$-{\it collapse} (see Lemma 6.3.3 in \cite{CM}) of $\pi$. In particular, $\pi_X=\pi$ when $X=A$. Now we restate Theorem 4 in \cite{P3}, which is due to Kraft \cite{Kr} and Spaltenstein \cite{Sp}.

\begin{theorem}[\cite{P3}, Theorem 4]\label{bkptthm1}
Let $I\subset \Delta$. The Richardson orbit $R_I=O_{(\pi_I^t)_X}$ except that $X$ is of type $D$ and $\pi_I$ is very even. In this exceptional case, $R_I=O_{(\pi_I^t)_X}^{\mathrm{\Rmnum{1}}}$ when one of the following condition is satisfied:
\begin{itemize}
  \item [(1)] $I$ is standard and $4\mid n$;
  \item [(2)] $I$ is not standard and $4\nmid n$.
\end{itemize}
Similarly, $R_I=O_{(\pi_I^t)_X}^{\mathrm{\Rmnum{2}}}$ when one of the following condition is satisfied:
\begin{itemize}
  \item [(1)] $I$ is standard and $4\nmid n$;
  \item [(2)] $I$ is not standard and $4\mid n$.
\end{itemize}
\end{theorem}

\begin{theorem}[\cite{P3}, Theorem 8]\label{bkptthm2}
Let $I, J\subset \Delta$. The following three conditions are equivalent.
\begin{itemize}
  \item [(1)] $(\Phi, \Phi_I, \Phi_J)$ contains a nonzero block;
  \item [(2)] $O_I\leq R_J$;
  \item [(3)] $O_J\leq R_I$.
\end{itemize}
\end{theorem}

The following corollary is Corollary 9 in \cite{P3} with slight correction. Based on the idea in \cite{P3}, we write down the full proof for self containment.

\begin{cor}[\cite{P3}, Corollary 9]\label{bkptcor1}
Let $\Phi$ be a classical root system with $I, J\subset\Delta$ such that $I$ is standard. If $(\Phi, \Phi_I, \Phi_J)$ is nonempty, then $\pi_I\unlhd\pi^t_J$ and $\pi_J\unlhd\pi^t_I$. The converse is true unless the following conditions are satisfied:
\begin{itemize}
  \item [(1)] $\pi_I$ and $\pi_J$ are very even;
  \item [(2)] $\pi_I=\pi_J^t$ and $\pi_J=\pi_I^t$;
  \item [(3)] $J$ is not standard when $4\mid n$ or $J$ is standard when $4\nmid n$.
\end{itemize}
\end{cor}
The original corollary was given without considering the exceptions. The counterexample can be found in Example \ref{bkptex1}.

\begin{proof}
It was obtained by Gerstenhaber \cite{G} and Hesselink \cite {Hes} (see also Theorem 6.2.5 in \cite{CM}) that $O_\pi<O_\eta$ if and only if $\pi\lhd\eta$ for any $\pi, \eta\in\caP_X(N(X_n))$. Moreover, $O_\pi\leq O_\eta$ if and only if $\pi\unlhd\eta$ except that $\pi=\eta$ is very even and $\{O_\pi, O_\eta\}=\{O_\pi^{\mathrm{\Rmnum{1}}}, O_\pi^{\mathrm{\Rmnum{2}}}\}$. If $(\Phi, \Phi_I, \Phi_J)$ is nonempty, then $\pi_I\unlhd (\pi_J^t)_X$ in view of Theorem \ref{bkptthm1} and \ref{bkptthm2}. It is evident that $\pi_I\unlhd (\pi_J^t)_X$ if and only if $\pi_I\unlhd \pi_J^t$.

Conversely, assume that $\pi_I\unlhd\pi^t_J$. Then $\pi_I\unlhd (\pi_J^t)_X$. If $\pi_I\lhd (\pi_J^t)_X$, we always have $O_I<R_J$ in view of Theorem \ref{bkptthm1} and \ref{bkptthm2}. If $\pi_I\unlhd (\pi_J^t)_X$ and $\pi_I$ is not very even, then $O_I\leq R_J$. It suffices to consider the case $\pi_I=(\pi_J^t)_X$ and $\pi_I$ is very even. If $(\pi_J^t)_X\lhd\pi_J^t$, it follows from the construction of $X$-collapse (see Lemma 6.3.3 in \cite{CM}) that $(\pi_J^t)_X$ has at least one odd part. This is impossible since $(\pi_J^t)_X=\pi_I$ is very even. It forces $(\pi_J^t)_X=\pi_J^t$, that is, $\pi_J$ is very even. This gives (1) and (2). Recall that $O_I=O_{\pi_I}^{\mathrm{\Rmnum{1}}}$ for standard $I$. With $\pi_I=(\pi_J^t)_X$ and Theorem \ref{bkptthm1}, $O_I\leq R_J$ unless $R_J=O_{\pi_I}^{\mathrm{\Rmnum{2}}}$. In this exceptional case, we get (3) by Theorem \ref{bkptthm1}.
\end{proof}

\begin{remark}\label{bkptrmk1}
The above corollary can also be proved by a computational approach using our notation $n^\lambda_s(a)$. But this can only be achieved in a case-by-case fashion. The proof will be much more time consuming than the original one.
\end{remark}

\begin{example}\label{bkptex1}
Let $\Phi=D_4$. Suppose that
\[
I=\{e_1-e_2, e_3-e_4\}\ \mbox{and}\ J=\{e_1-e_2, e_2-e_3, e_3-e_4\}.
\]
Then both $\pi_I$ and $\pi_J$ are very even. Set $\lambda=(1, 1, 1, 1)$. Then $\Phi_\lambda=\Phi_J$. Moreover, $\vf(\lambda)=(1, 1, 1, -1)$ and $\vf(J)=\{e_1-e_2, e_2-e_3, e_3+e_4\}$. It can be verified that $\mu=(1, -1, 1, -1)\in\Lambda_I^+$ and thus $\mu\in {}^IW^J\lambda$. So ${}^IW^J$ is not empty. If ${}^IW^{\vf(J)}$ is not empty, there exists $w\in {}^IW^{\vf(J)}$ such that $\nu=w\vf(\lambda)\in\Lambda_I^+$. So $|\nu_1|=|\nu_2|=|\nu_3|=|\nu_4|=1$. We must have $\nu=(1, -1, 1, -1)=w(1, 1, 1, -1)$ since $\nu\in \Lambda_I^+$ also implies $\nu_1>\nu_2$ and $\nu_3>\nu_4$. This is absurd since $w\in W$ changes only even number of signs. Hence ${}^IW^{\vf(J)}=\emptyset$, even though $\pi_I=\pi_{\vf(J)}^t$.
\end{example}

\subsection{Criterion by partitions} The following definition is somewhat opaque looking at first sight. We should consider this definition with Lemma \ref{bkptlem3}.

\begin{definition}\label{bkptdef3}
We say $1\leq k\leq 2m-1$ is {\it compatible} with $(\Phi, \Phi_I, \Phi_J)$ if the following conditions are satisfied:
\begin{itemize}
  \item [(\rmnum{1})] For $\Phi=B_n$, $(\pi_J)_2\geq k\geq 2\overline n_{\overline m}+1$ and $(\pi_I)_k\leq 2n_m+1$ when $k$ is odd; $k\leq 2\overline n_{\overline m}$ and $(\pi_I)_k\geq 2n_m+1$ when $k$ is even.
  \item [(\rmnum{2})] For $\Phi=C_n$, $(\pi_J)_2\geq k\geq 2\overline n_{\overline m}+1$ and $(\pi_I)_k\leq 2n_m$ when $k$ is odd; $k\leq 2\overline n_{\overline m}$ and $(\pi_I)_k\geq 2n_m+1$ when $k$ is even.
  \item [(\rmnum{3})] For $\Phi=D_n$, $(\pi_J)_2\geq k\geq 2\overline n_{\overline m}$ and $(\pi_I)_k\geq 2n_m$ when $k$ is even; $k\leq 2\overline n_{\overline m}-1$ and $(\pi_I)_k\leq 2n_m-1$ when $k$ is odd.
\end{itemize}
\end{definition}

is routine to check the following result.

\begin{lemma}\label{bkptlem3}
If $(S, \overline S)$ is a separable pair of $(\Phi, \Phi_I, \Phi_J)$, then
\[
k=\left\{\begin{aligned}
&2|S|-1,\qquad\quad\ \mbox{if}\ m\in S;\\
&2|S|,\qquad\qquad\quad \mbox{if}\ m\not\in S
\end{aligned}
\right.
\]
is compatible with $(\Phi, \Phi_I, \Phi_J)$.
\end{lemma}

Let $\overline{\lambda'}$ be a dominant weight with $\Phi_{\overline{\lambda'}}=\Phi_I$. Denote $a'_s=\overline{\lambda'}_{q_s}$ for $1\leq s<m$ and $a'_m=0$. Set $\caA'=\{a'_1, \ldots, a'_m\}$. Here we also assume that $a'_{m-1}>0$ (see Remark \ref{lrrmk3}).

\begin{lemma}\label{bkptlem4}
Let $\Phi=B_n$, $C_n$ or $D_n$. Suppose $I, J\neq\Delta$ and ${}^IW^J\neq\emptyset$. Let $\overline\lambda$ $($resp. $\overline{\lambda'})$ be a dominant weight with $\Phi_{\overline\lambda}=\Phi_J$ $($resp. $\Phi_{\overline{\lambda'}}=\Phi_I)$. Then the following conditions are equivalent:
\begin{itemize}
  \item [(1)] $(\Phi, \Phi_I, \Phi_J)$ is separable.
  \item [(2)] $(\Phi, \Phi_J, \Phi_I)$ is separable.
  \item [(3)] There exists $\lambda\in {}^IW^J\overline\lambda$ such that $\lambda$ is separable.
  \item [(4)] There exists $\lambda'\in {}^JW^I\overline{\lambda'}$ such that $\lambda'$ is separable.
  \item [(5)] There exists $k$ which is compatible with $(\Phi, \Phi_I, \Phi_J)$ and $\sum_{i=1}^k(\pi_I)_i=\sum_{i=1}^k(\pi_J^t)_i.$
  \item [(6)] There exists $l$ which is compatible with $(\Phi, \Phi_J, \Phi_I)$ and $\sum_{i=1}^l(\pi_J)_i=\sum_{i=1}^l(\pi_I^t)_i.$
\end{itemize}
\end{lemma}
\begin{proof}
Evidently (1) implies (3), while (2) implies (4). We will first show that (5) implies (1), then prove that (3) implies (5) and (6). Details are provided for $\Phi=B_n$, while the other cases are similar.

Now (5) is true. If $k$ is odd, then $k\geq 2\overline n_{\overline m}+1$ and $(\pi_I)_k\leq 2n_m+1$. Denote $S_1=\{1\leq s<m\mid n_s>(\pi_I)_k\}\cup\{m\}$ and $S_2=\{1\leq s<m\mid n_s\geq (\pi_I)_k\}\cup\{m\}$. There exists $S_1\subset S\subset S_2$ so that $k=2|S|-1$. Set $\overline S_1=\{1\leq t<\overline m\mid \overline n_t>k\}$ and $\overline S_2=\{1\leq t<\overline m\mid \overline n_t\geq k\}$. Note that $(\pi_J)_2=\max\{\overline n_s\mid 1\leq s<\overline m\}$. It follows from $(\pi_J)_2\geq k$ that $\overline S_2$ is nonempty. Choose nonempty set $\overline S_1\subset \overline S\subset \overline S_2$. Then $\overline m\not\in \overline S$.
Note that $\pi_J$ is obtained by arranging the set
\[
\{\overline n_1, \overline n_1, \overline n_2, \overline n_2,\ldots, \overline n_{\overline m-1}, \overline n_{\overline m-1}, 2\overline n_{\overline m}+1\}.
\]
Choose any $\lambda\in{}^IW^J\overline\lambda$. One has
\begin{equation}\label{bkptlem4eq1}
\begin{aligned}
(\pi_J^t)_1+\ldots+(\pi_J^t)_{k}&=2\sum_{t\in \overline S}k+2\sum_{t\not\in \overline S\cup\{\overline m\}}\overline n_t+(2\overline n_{\overline m}+1)\\
&=2k|\overline S|+2\sum_{t\not\in \overline S}\sum_{s=1}^mn^\lambda_s(a_t)+1\\
&\geq 2\sum_{t\in \overline S}\sum_{s\in S}n^\lambda_s(a_t)+2\sum_{t\not\in \overline S}\sum_{s\in S}n^\lambda_s(a_t)+1\\
&=2\sum_{s\in S}\sum_{t=1}^{\overline m}n^\lambda_s(a_t)+1=2\sum_{s\in S\backslash\{m\}}n_s+(2n_m+1)\\
&=(\pi_I)_1+\ldots+(\pi_I)_{k}.
\end{aligned}
\end{equation}
The equality holds in view of $\sum_{i=1}^k(\pi_I)_i=\sum_{i=1}^k(\pi_J^t)_i$. So $n^\lambda_s(a_t)$ is maximal for $s\in S$ and $t\in \overline S$. Moreover, $n^\lambda_s(a_t)=0$ for $s\not\in S$ and $t\not\in \overline S$. It is easy to verify that $(S, \overline S)$ is a separable pair for $(\Phi, \Phi_I, \Phi_J)$. If $k$ is even, the argument is similar.

Next assume that (3) is true, that is, $\lambda$ is separable relative to a separable pair $(S, \overline S)$. Then $n^\lambda_s(a_t)$ is maximal for any $s\in S$, $t\in \overline S$ and $n^\lambda_s(a_t)=0$ for any $s\not\in S$, $t\not\in \overline S$. If $m\in S$, then $\overline m\not\in \overline S$. We get $\overline n_t\geq\sum_{s\in S}n^\lambda_s(a_t)=2|S|-1$ for $t\in \overline S$ and $\overline n_t=\sum_{s\in S}n^\lambda_s(a_t)\leq 2|S|-1$ for $t\not\in \overline S\cup\{\overline m\}$. Moreover, $\overline n_{\overline m}=\sum_{s\in S}n^\lambda_s(0)\leq |S|-1$. Similarly, one has $n_s\geq \sum_{t\in \overline S}n^\lambda_s(a_t)=2|\overline S|$ for $s\in S\backslash\{m\}$ and $n_s=\sum_{t\in \overline S}n^\lambda_s(a_t)\leq 2|\overline S|-1$ for $s\not\in S$. Moreover, $n_m\geq |\overline S|$. Thus the equality in (\ref{bkptlem4eq1}) holds, so does (5) when $k=2|S|-1$. On the other hand,
\[
\begin{aligned}
(\pi_I^t)_1+\ldots+(\pi_I^t)_{2|\overline S|}&=2|\overline S|(2|S|-1)+1+2\sum_{s\not\in S}n_s\\
&=2\sum_{s\in S}\sum_{t\in \overline S}n^\lambda_s(a_t)+2\sum_{s\not\in S}\sum_{t\in \overline S}n^\lambda_s(a_t)\\
&=2\sum_{t\in \overline S}\overline n_t=(\pi_J)_1+\ldots+(\pi_J)_{2|\overline S|}.
\end{aligned}
\]
Choose $l=2|\overline S|$. We obtain (6). If $m\not\in S$ or $\overline m\in \overline S$, the argument is similar.

Now we have obtained the equivalences of (1), (3) and (5). Moreover, (3) implies (6). In a similar spirit, we can also get the equivalences of (2), (4) and (6). Thus (3) implies (4). Hence the lemma can be proved by symmetry.
\end{proof}

\subsection{Number of blocks} In this subsection, we will give the number of blocks without invoking any associated dominant weights (see Theorem \ref{bkssthm22}, Theorem \ref{bkssthm32} and Theorem \ref{bkssthm42}). It follows from Lemma \ref{bkptlem4} that a separable pair of $(\Phi, \Phi_I, \Phi_J)$ determines a pair of integers $(k, l)$, which we called \emph{compatible pair}. Denote by $\caC=\caC(\Phi, \Phi_I, \Phi_J)$ the set of compatible pairs. For $\Phi=B_n, C_n$ (resp. $D_n$), $(k, l)\in\caC$ if and only if the following conditions are satisfied.
\begin{itemize}
  \item [(i)] $k$ is compatible with $(\Phi, \Phi_I, \Phi_J)$ and $\sum_{i=1}^k(\pi_I)_i=\sum_{i=1}^k(\pi_J^t)_i$;
  \item [(ii)] $l$ is compatible with $(\Phi, \Phi_J, \Phi_I)$ and $\sum_{i=1}^\ell(\pi_J)_i=\sum_{i=1}^\ell(\pi_I^t)_i$;
  \item [(iii)] if $k$ is odd (resp. even), then $l$ is even and
  \[
  2|\{1\leq t<\overline m\mid\overline n_t>k\}|\leq l\leq 2|\{1\leq t<\overline m\mid\overline n_t\geq k\}|;
  \]
  \item [(iv)] if $k$ is even (resp. odd), then $l$ is odd and
  \[
  2|\{1\leq t<\overline m\mid\overline n_t>k\}|+1\leq l\leq 2|\{1\leq t<\overline m\mid\overline n_t\geq k\}|+1.
  \]
\end{itemize}

Let $p$ be a positive integer. Define the following equivalent relations ``$\sim$'' on $\caC$.
\begin{itemize}
\item $(k, l)\sim(k+2p, l)$ when $(k, l), (k+2p, l)$ are compatible pairs;
\item $(k, l)\sim(k, l+2p)$ when $(k, l), (k, l+2p)$ are compatible pairs.
\end{itemize}
The pair $(k, l)\in C$ is called \emph{trivial} if $(k, l)\sim (2, 2\overline m-1)$ or $(2m-1, 2)$. The following results are easy consequence of Theorem \ref{bkssthm22}, Theorem \ref{bkssthm32} and Theorem \ref{bkssthm42}.

\begin{theorem}\label{bkptthm3}
Let $\Phi=B_n$, $C_n$. The system $(\Phi, \Phi_I, \Phi_J)$ has $2^p$ blocks, where $p$ is the number of equivalence classes of nontrivial compatible pairs.
\end{theorem}

\begin{theorem}\label{bkptthm4}
Let $\Phi=D_n$. Suppose $(\Phi, \Phi_I, \Phi_J)$ has an odd compatible pair. Then it contains $2^{p-1}$ blocks, where $p$ is the number of equivalence classes of compatible pairs.
\end{theorem}

\begin{theorem}\label{bkptthm5}
Let $\Phi=D_n$. Suppose $(\Phi, \Phi_I, \Phi_J)$ has no odd compatible pair. Then it contains $2^{p}$ blocks, where $p$ is the number of equivalence classes of nontrivial compatible pairs.
\end{theorem}

\begin{example}\label{bkptex2}
Let $\Phi=B_{k(k+1)}$ for $k\in\bbZ^{>0}$. Let $I=\Delta\backslash\{\alpha_{i(i+1)}\mid 1\leq i\leq k\}$ and $J=\Delta\backslash\{\alpha_{i^2}\mid 1\leq i\leq k\}$. This is the general case in Example \ref{ccex2}. One has
\[
\pi_I=\{2k, 2k, \ldots, 2, 2, 1\}\ \mbox{and}\ \pi_J=\{2k+1, 2k-1, 2k-1, \ldots, 1, 1\};
\]
Thus $\pi_I^t=\{2k+1, 2k, 2k-2, 2k-2, \ldots, 2, 2, 1\}$ and $\pi_J^t=\pi_J$. We obtain
\[
\caC=\{(2, 2k-1), (4, 2k-3), \ldots, (2k, 1)\}
\]
Hence the system has $k$ equivalence classes and $2^k$ blocks.
\end{example}


\begin{thebibliography}{EHW}
\bibitem[B]{B} E. Backlin, Koszul duality for parabolic and singular category $\caO$. Represent. Theory \textbf{3}(1999), 139-152.


\bibitem[BX]{BX} Z. Q. Bai and W. Xiao, Irreducibility of generalized Verma modules for hermitian symmetric pairs, submitted.



\bibitem[BB]{BB} A. Beilinson and J. Bernstein, A proof of Jantzen conjecture, I.M. Gelfand seminar, 1-50, Adv. Soviet math., 16, Part 1, Amer. Math. Soc., Providence, RI, 1993.

\bibitem [BGS]{BGS} A. Beilinson, V. Ginzburg and W. Soergel, Koszul duality patterns in representation theory, J. Amer. Math. Soc. \textbf{9} (1996), 473-527.

\bibitem [BGG]{BGG} Bernstein, I. N., Gelfand, I. M., Gelfand, S. I. : Structure of representations generated by vectors of highest weight. Functional Analysis Appl. 5, 1--9 (1971)



\bibitem [Bo]{Bo}B. D. Boe, Homomorphisms between generalized Verma modules, Trans. Amer. Math. Soc. 288(1985), 791-799

\bibitem [BC]{BC} B. D. Boe and D. Collingwood, A comparison theory for the structure of induced
representations, J. Algebra 94(1985), 511-545.

\bibitem [BEJ]{BEJ} B. D. Boe, T. J. Enright , and B . SHELTON , Determination of the intertwining
operators for holomorphically induced representations of Hermitian symmetric
pairs, Pacific J. Math. 131 (1988), 39-50.

\bibitem [BN]{BN} B. D. Boe and L. Nakano, Representation type of the blocks of category $\caO_s$, Adv. in Math., 196 (2005) 193-256.


\bibitem [Br]{Br} J. Brundan, Centers of degenerate cyclotomic Hecke algebras and parabolic category $\caO$, Represent. Theory, 12(2008), 236-259.


\bibitem [CM]{CM} D. H. Collingwood and W. M. McGovern, Nilpotent orbits in semisimple Lie algebrasb, Van Nostrand Reinhold Mathematics Series. Van Nostrand Reinhold Co., New York, 1993.




\bibitem [ES]{ES}T. J. Enright and B. Shelton, Categories of highest weight modules: applications to classical Hermitian symmetric pairs, Mem. Amer. Math. Soc. 67(1987), no. 367


\bibitem [G]{G} M. Gerstenhaber, Dominance over the classical groups, Ann. of Math.(2) 74(1961) 532-569.

\bibitem [Hes]{Hes} W. Hesselink, Singularities in the nilpotent scheme of a classical group, Trans. Amer. Math. Soc. 222(1976), 1-32.

\bibitem [He]{He} H. He, On the reducibility of scalar generalized Verma modules of abelian type, Algebr. Represent. Theory 19 (2016), no. 1, 147-170.

\bibitem [HKZ]{HKZ}H. He, T. Kubo and R. Zierau, On the reducibility of scalar generalized Verma modules
associated to maximal parabolic subalgebras, to appear in Kyoto. J. Math.

\bibitem [HXZ]{HXZ} J. R. Hu, W. Xiao and A. L. Zhang, blocks of type $E$, preprint.

\bibitem [HX]{HX} J. Hu and W. Xiao, Jantzen coefficients and radical filtrations of generalized Verma modules, preprint.

\bibitem [H1]{H1} J. Humphreys, Introduction to Lie algebras and representation theory, Springer-Verlag, New York, 1972.

\bibitem [H2]{H2} J. Humphreys, Reflection groups and Coexter groups, Cambridge Studies in Advanced Mathematics, 29, Cambridge Univ. Press, Cambridge, 1990.

\bibitem [H3]{H3} J. Humphreys, Representations of Semisimple Lie Algebras in the BGG Category $\mathcal{O}$, GSM. 94, Amer. Math. Soc. Providence, 2008.


\bibitem [J1]{J1} J. C. Jantzen, Moduln mit einem h\"{o}chsten Gewicht. Lecture Notes in Mathematics, 750. Springer, Berlin, 1979

\bibitem [J2]{J2} J. C. Jantzen, Kontravariante Formen auf induzierten Darstellungen halbeinfacher Lie-Algebren, Math. Ann. 226(1977), 53-65.



\bibitem [KL]{KL} D. Kazhdan and G. Lusztig, Representations of Coxeter groups and Hecke algebras, Invent. Math. 53 (1979) 165-184.

\bibitem [Kr]{Kr} H. Kraft, Parametrisierung von Konjugationsklassen in $\frsl_n$, Math. Ann.  234  (1978), no. 3, 209-220.


\bibitem [LX]{LX} S. Leonard and N. Xi, Some non-trivial Kazhdan-Lusztig coefficients of an affine Weyl group of type $\tilde A_n$, Sci. China Math. 53(2010), no. 8, 1919-1930.

\bibitem [L1]{L1} J. Lepowsky, Conical vectors in induced modules, Trans. Amer. Math. Soc. 208(1975), 219-272.

\bibitem [L2]{L2} J. Lepowsky, Existence of conical vectors in induced modules, Ann. of Math. (2) 102(1975), 17-40.

\bibitem [Lu]{Lu} G. Lusztig, Cells in affine Weyl groups, in: Algebraic Groups and Related Topics, in: Adv. Stud. Pure Math., vol. 6, Kinokunia¨CNorth-Holland, 1985, 255-287.


\bibitem [M1]{M1} H. Matumoto, The homomorphisms between scalar generalized Verma modules associated to maximal parabolic subalgebras, Duke Math. J. 131(2006), no. 1, 75-118.

\bibitem [M2]{M2} H. Matumoto, On the homomorphisms between scalar generalized Verma modules. Compos. Math. 150 (2014),  no. 5, 877-892.

\bibitem [M3]{M3} H. Matumoto, Homomorphisms between scalar generalized Verma modules for gl(n,C) . Int. Math. Res. Not. no. 12(2016),3525-3547.

\bibitem [P1]{P1} K. J. Platt, Classifying the representation type of infinitesimal blocks of category $\caO_S$, Ph.D. thesis, The University of Georgia, 2008.

\bibitem [P2]{P2} K. J. Platt, Representation type of the blocks of category $\caO_S$ in types $F_4$ and $G_2$, J. Algebra 322 (2009), no. 11, 3823-3838.

\bibitem [P3]{P3} K. J. Platt, Nonzero Infinitesimal Blocks of Category $\caO_S$, Algebr Represent Theory 14 (2011), no. 4, 665-689.

\bibitem [R]{R} A. Rocha-Caridi, Splitting criteria for $\frg$-modules induced from a parabolic and the Bernstein-Gelfand-Gelfand resolution of a finite-dimensional, irreducible $\frg$-module, Trans. Amer. Math. Soc. 262 (1980), 335-366.

\bibitem [Sh]{Sh} P. Shan, Graded decomposition matrices of $v$-Schur algebras via Jantzen filtration. Represent. Theory 16(2012), 212¨C269.

\bibitem [So]{So} W. Soergel, Kategorie O, Perverse Garben Und Moduln
Uber Den Koinvariantez Zur Weylgruppe, J. Amer. Math. Soc.
\textbf{3} (1990) 421-445.


\bibitem [Sp]{Sp} N. Spaltenstein, Classes Unipotentes et Sous-groupes de Borel. Lecture Notes in Mathematics, vol. 946, Springer-Verlag, Berlin (1982).

\bibitem [V]{V} D. A. Vogan, Irreducible characters of semisimple Lie groups. I, Duke Math. J. 46(1) (1979)



\bibitem [W]{W} N. Wallach, Real reductive groups. I. Pure and Applied Mathematics, 132. Academic Press, Inc., Boston, MA, 1988.

\bibitem [X]{X} N. Xi, The leading coefficient of certain Kazhdan-Lusztig polynomials of the permutation group, J. Algebra,  285(2005), no. 1, 136-145.

\bibitem [Xi]{Xi} W. Xiao, Leading weight vectors and homomorphisms between generalized Verma modules, J. Algebra, 430(2015), 62-93.

\bibitem [XZ]{XZ} W. Xiao and A. Zhang, Jantzen coefficients and irreducibility of generalized Verma modules, preprint.

\end{thebibliography}
\end{document}